\documentclass[12pt, a4paper, reqno]{amsart}

\usepackage{graphicx}
\usepackage{amsmath}
\usepackage{amssymb}
\usepackage{amsthm}
\usepackage[utf8]{inputenc}
\usepackage[T1]{fontenc}
\usepackage{amsfonts}
\usepackage{mathtools}
\usepackage{mathrsfs}
\usepackage{enumitem}
\usepackage{xfrac}
\usepackage{bbm}
\usepackage{xcolor}
\usepackage{lmodern}
\usepackage{xparse}
\usepackage[
backend=biber,
style=alphabetic,
sorting=nyt,
maxbibnames=6,
]{biblatex}
\usepackage{geometry}
\usepackage{blindtext}
\usepackage{stmaryrd} 
\usepackage{url}
\usepackage{bbold}
\usepackage{tikz}\usetikzlibrary{
arrows.meta, cd, decorations.markings, 3d, shapes.geometric, math, colorbrewer}
\usepackage{pgfplots}\pgfplotsset{compat=1.18}
\usepackage{standalone}
\usepackage[hidelinks]{hyperref}
\usepackage[title]{appendix}
\usepackage{cleveref}
\usepackage{comment}

\title[]{Periodic Orbits for Perturbations of Zoll Contact Forms}
\author[Tom Stalljohann]{Tom Stalljohann$^\ast$}
\thanks{$\ast$ Universit\"at Heidelberg, \url{tstalljohann@mathi.uni-heidelberg.de}}
\date{2026}

\newtheorem{theorem}{Theorem}[section]
\newtheorem*{theorem*}{Theorem}
\newtheorem{MainThm}{Theorem}

\newtheorem{lemma}[theorem]{Lemma}
\newtheorem*{lemma*}{Lemma}
\newtheorem{corollary}[theorem]{Corollary}
\newtheorem*{corollary*}{Corollary}

\newtheorem{proposition}[theorem]{Proposition}
\newtheorem*{proposition*}{Proposition}
\newtheorem{step}{Step}

\newtheorem*{claim}{Claim}
\newtheorem{claimI}{Claim}
\newtheorem*{conjecture*}{Conjecture}

\theoremstyle{definition}
\newtheorem{definition}[theorem]{Definition}
\newtheorem*{definition*}{Definition}
\newtheorem*{convention}{Convention}
\newtheorem*{question*}{Question}

\theoremstyle{remark}
\newtheorem{remark}[theorem]{Remark}

\newtheorem{example}[theorem]{Example}

\newenvironment{proofClaim}{ \proof[Proof of the claim]}{\endproof}

\newcommand{\Forall}[0]{\forall\,}
\newcommand{\Exists}[0]{\exists\,}

\newcommand{\R}[0]{\mathbb{R}}
\newcommand{\Rbar}[0]{\overline{\R}}
\newcommand{\C}[0]{\mathbb{C}}
\newcommand{\Z}[0]{\mathbb{Z}}
\newcommand{\N}[0]{\mathbb{N}}

\newcommand{\bbS}[0]{\mathbb{S}}

\newcommand{\eps}[0]{\varepsilon}
\newcommand{\comma}[0]{\, , \, }
\newcommand{\Cinfty}[0]{C^{\infty}}
\newcommand{\Cinftyloc}[0]{\Cinfty_{\mathrm{loc}}}

\newcommand{\ie}[0]{i.e.\ }

\newcommand{\cf}[0]{cf.\ }

\newcommand{\ArzelaAscoli}[0]{Arzel\`a-Ascoli\ }

\newcommand{\set}[2]{\left\{#1 \,\big| \, #2 \right\}}
\newcommand{\Bigset}[2]{\left\{#1 \,\Big| \, #2 \right\}}

\renewcommand{\d}[0]{\mathrm{d}}
\newcommand{\D}[0]{\mathrm{D}}
\newcommand{\Dvert}[0]{\mathrm{D}^{\mathrm{v}}}
\newcommand{\omegacan}[0]{\omega_{\mathrm{can}}}

\newcommand{\lambdacan}[0]{\lambda_{\mathrm{can}}}
\newcommand{\interior}[0]{\mathrm{int}}

\newcommand{\image}[0]{\mathrm{im}}
\newcommand{\supp}[0]{\mathrm{supp}}
\newcommand{\id}[0]{\mathrm{id}}

\newcommand{\convergence}[1]{\overset{#1\rightarrow \infty}{\longrightarrow}}

\newcommand{\Crit}[0]{\mathrm{Crit}}
\newcommand{\loopspace}[0]{\mathscr{L}}
\newcommand{\moduli}[0]{\mathcal{M}}

\newcommand{\Dscr}[0]{\mathscr{D}}
\newcommand{\COneDatum}[0]{\mathscr{D} = (U_\Sigma, \, c, \, \mathscr{B}_{C^1})}
\newcommand{\BCOne}[0]{\mathscr{B}_{C^1}}
\newcommand{\Rabinowitz}[1]{\mathcal{A}^{#1}}
\newcommand{\RabinowitzHZero}[0]{\Rabinowitz{H_0}}
\newcommand{\RabinowitzH}[0]{\Rabinowitz{H}}
\newcommand{\ActionClass}[1]{\mathcal{A}_{\mathrm{class}}^{#1}}
\newcommand{\Hcal}[0]{\mathcal{H}}
\newcommand{\Hclassparam}[0]{\Hcal^{h_0}(H_0, \Dscr, \delta)}
\newcommand{\Hclassdelta}[1]{\Hcal^{h_0}(H_0,\Dscr, #1)}
\newcommand{\RabinowitzHParam}[1]{\mathcal{A}^{H}_{#1}}
\newcommand{\RabinowitzHrs}[0]{\RabinowitzHParam{r,s}}
\newcommand{\CCrit}[0]{\mathscr{C}^{\tau(\alpha)}}
\newcommand{\moduliH}[1]{{\moduli}^H_{#1}}

\newcommand{\Fsection}[0]{\mathcal{F}}
\newcommand{\Bcal}[0]{\mathcal{B}}
\newcommand{\Ecal}[0]{\mathcal{E}}
\newcommand{\Ucal}[0]{\mathcal{U}}
\newcommand{\Ocal}[0]{\mathcal{O}}
\newcommand{\BcalHat}[0]{\widehat{\Bcal}}
\newcommand{\EcalHat}[0]{\widehat{\Ecal}}
\newcommand{\ev}[0]{\mathrm{ev}}

\newcommand{\gaux}[0]{g_{\mathrm{aux}}}

\numberwithin{equation}{section}

\begin{document}

\begin{abstract}
We consider the perturbation of a Zoll contact form on some prescribed domain of the underlying manifold by multiplying it with a positive function which is constant of value 1 outside the domain. For every sufficiently $C^0$-small such perturbation we find a periodic Reeb orbit of the perturbed contact form intersecting the domain. As an application, starting from a Zoll Riemannian manifold, we demonstrate that for every exact magnetic field, with $C^0$-small magnetic potential vanishing outside some given domain, there exists a periodic magnetic geodesic intersecting this domain.

The theorem is a rather direct consequence of a result which is of independent interest: For Hamiltonians sufficiently $C^0$-close to a defining Hamiltonian, we show existence of a gradient flow line of the corresponding Rabinowitz action functional with a constraint on the position of the cylinder component at $(0,0) \in \R \times \bbS^1$. This relies on a homotopy stretching argument for the Rabinowitz action functional, in the course of which we have to derive some delicate estimates to ensure compactness of the appearing moduli spaces.
\end{abstract}

\maketitle

\section{Introduction}
\label{sec: Introduction}

\subsection{Motivation}
\label{subsec: Introduction - Motivation}

Suppose we are given a contact manifold $(\Sigma,\alpha)$ and an open domain $D \subseteq \Sigma$ with smooth boundary for which there exists a periodic Reeb orbit intersecting $D$. Let $f : \Sigma \rightarrow \R$ be a function which vanishes on $\Sigma \backslash D \,$. We are interested in finding periodic Reeb orbits of the perturbed contact form $e^f \alpha \,$. Since the Reeb vector fields of $\alpha$ and $e^f \alpha$ agree on $\Sigma \backslash D \,$, a nontrivial such periodic Reeb orbit should intersect $D$. We explicitly allow $D = \Sigma \,$, in which case the condition on $f$ to vanish on $\Sigma \backslash D$ becomes void and we are merely looking for a periodic Reeb orbit of $e^f \alpha \,$.
\\

\noindent The relevance of this question lies in the fact that Reeb orbits of $(\Sigma,e^f \alpha)$ are conjugated via $f \times \id_\Sigma : \Sigma \rightarrow \R \times \Sigma$ to 
integral curves of the characteristic line bundle on the graph $\Sigma_f = \mathrm{Gr}(f) \subseteq \R \times \Sigma$ in the symplectization $(\R_h \times \Sigma, \d (e^h \alpha)) \,$. If $\Sigma \subseteq W$ is of contact type in an ambient symplectic manifold $(W,\omega) \,$, then a neighborhood of $\Sigma$ in $W$ is symplectomorphic to a neighborhood of $\{0\} \times \Sigma$ in $\R \times \Sigma $ via the flow of a transverse Liouville vector field. In this sense, periodic Reeb orbits of $e^f \alpha$ intersecting $D$ correspond to closed characteristics on the hypersurface $\Sigma_f \subseteq W$ whose projection to $\Sigma$ via the Liouville flow intersects $D$.

This is especially interesting in the case of the cotangent bundle,
\[
(W,\omega) = (T^*M, \omegacan = \d \lambdacan) \comma \text{ where } \lambdacan := {\textstyle \sum_{i}  p_i \, \d q_i }\, ,
\]
as the flow of the fiberwise radial Liouville vector field $X_{\mathrm{rad}} = \sum_i p_i \, \partial_{p_i}$ preserves cotangent fibers. If $\Sigma \subseteq T^*M$ is a fiberwise starshaped hypersurface with induced contact form $\alpha := \lambdacan |_\Sigma$ and if $B \subseteq M$ is an open domain with smooth boundary, we may take $D := T^*B \cap \Sigma \subseteq T^*M \,$. Then periodic Reeb orbits of $(\Sigma,e^f \alpha)$ intersecting $D$ correspond one-to-one to closed characteristics on $\Sigma_f \subseteq T^*M$ which intersect $T^*B \,$.
\\

\noindent We will answer the above raised question of existence of periodic Reeb orbits of $(\Sigma,e^f \alpha)$ intersecting $D$ in the affirmative for every Zoll contact manifold $(\Sigma,\alpha) \,$, that is the Reeb flow of $\alpha$ is everywhere periodic of the same minimal period, and every sufficiently $C^0$-small $f$ which satisfies a given bound on its derivative, see Theorem \ref{mthm: rescaling Zoll contact form}. The situation in the Zoll case is depicted schematically in Figure \ref{fig: Perturbed Contact Form}.
\\

\noindent As an application in the spirit of the above line of thought, we show that for every magnetic potential on a Zoll Riemannian manifold $(M,g)$, which vanishes outside a domain $B \subseteq M$, there exists a periodic magnetic geodesic intersecting $B$. More precisely, let $(M,g)$ be a closed Zoll Riemannian manifold, \ie all geodesics on $M$ are closed and of the same length. Given a 1-form $\theta \in \Omega^1(M) \,$, we define the skew-symmetric bundle endomorphism $F = F^{(g,\theta)} : TM \rightarrow TM$ by $g (F \, \cdot\, , \,\cdot\,) = \d \theta \,$. We call $F$ the \textit{Lorentz force} of the magnetic field $\d \theta$ with \textit{magnetic potential} $\theta \,$. A \textit{magnetic geodesic} $\gamma : (t_0,t_1) \rightarrow M$ is a curve satisfying 
\begin{equation}
\label{eq: introduction - magnetic geodesic eq}
\frac{\D}{\d t} \dot{\gamma} = F_\gamma(\dot{\gamma}) \, , 
\end{equation}
where $\frac{\D}{\d t}$ denotes covariant differentiation with respect to the metric $g$. Equation \eqref{eq: introduction - magnetic geodesic eq} describes the motion in $M$ of a charged particle under the influence of a magnetic field with Lorentz force $F$. The energy of a magnetic geodesic $\gamma$ is defined by $E(\gamma) := \tfrac{1}{2} \, |\dot{\gamma}|^2$ and is constant along $\gamma\,$. 
We then show the following.

\begin{MainThm}
\label{mthm: magnetic geodesics}
Let $(M,g)$ be a closed Zoll Riemannian manifold of dimension $\dim(M) \geq 2 \,$. For every $E_0 > 0$ 
and every $C > 0$ there exists $\delta = \delta(E_0, C) > 0$ with the following significance. For every non-empty open subset $B \subseteq M$ with smooth boundary and every magnetic potential $\theta \in \Omega^1(M)$ with
\begin{enumerate}[label=\arabic*)]
    \item $\lVert \theta \rVert_{C^0} \leq \delta$ and $\lVert \theta \rVert_{C^1} \leq C \,$, and
    \item $\theta = 0$ on $M \backslash B \,$,
\end{enumerate}
there exists a periodic magnetic geodesic $\gamma \,$, \ie a solution of $\frac{\D}{\d t} \dot{\gamma} = F^{(g,\theta)}_\gamma(\dot{\gamma}) \,$,
of energy $E(\gamma) = E_0$ intersecting $B \,$.
\end{MainThm}

\noindent In fact, in Section \ref{sec: Electromagnetic Hamiltonians} we will show a generalization of this result (Theorem \ref{mthm: electromagnetic Hamiltonians - periodic solution}), in which we also allow for a mechanic potential and distortion of the metric.

\begin{figure}
    \centering
    \includegraphics[width=0.8\linewidth]{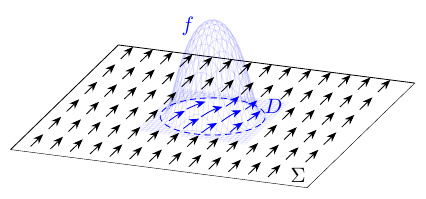}
    \caption{In this schematic sample, the open domain $D$ is the interior of the ellipse and $f : \Sigma \rightarrow \R$ has support $\overline{D}$. The Reeb vector fields of the Zoll contact form $\alpha$ respectively the perturbed contact form {\color{blue}$e^f \alpha$} are  colored black respectively {\color{blue}blue}.}
    \label{fig: Perturbed Contact Form}
\end{figure}

\subsection{Main results}
\label{subsec: Introduction - Main Results}

A contact form $\alpha \in \Omega^1(\Sigma)$ on a manifold $\Sigma$ of dimension $2N-1 \geq 3$ is a 1-form with $\alpha \wedge (\d \alpha)^{N-1} \not= 0$ pointwise everywhere. The vector field $R_\alpha$ on $\Sigma$ uniquely determined by
\[
\iota_{R_\alpha} \d \alpha = 0 \,\,\, \text{ and } \,\,\, \alpha(R_\alpha) =1 
\]
is called the \textit{Reeb vector field}, its induced flow the \textit{Reeb flow}.

\begin{definition}[Zoll contact manifold]
\label{def: Zoll contact mfd}
The pair $(\Sigma,\alpha)$ is called \textit{Zoll contact manifold} and $\alpha$ \textit{Zoll contact form} if $\Sigma$ is a closed (\ie compact, without boundary) connected manifold of dimension $\dim(\Sigma) \geq 3$ and $\alpha \in \Omega^1(\Sigma)$ is a contact form on $\Sigma$ for which all its Reeb orbits are periodic of the same minimal period.
\end{definition}

\noindent We need to introduce some further definitions to state our main theorem. 

\begin{definition}[$C^1$-bounded]
\label{def: C^1-bounded subset}
Let $M$ be a compact smooth manifold (possibly with boundary). The space of smooth functions $C^\infty(M) = \Cinfty(M,\R)$ with the $C^1$-topology is a normable topological vector space. A subset $\mathscr{B} \subseteq \Cinfty(M)$ will be called \textit{$C^1$-bounded} if $\mathscr{B}$ is a bounded subset with respect to any norm inducing the $C^1$-topology on $\Cinfty(M) \,$. This is independent of the norm used.
\end{definition}

\begin{remark}
\label{rmk: C^1-bounded subset}
Fixing a Riemannian metric on $M$, a subset $\mathscr{B} \subseteq \Cinfty(M)$ is $C^1$-bounded iff there exists $A \geq 0$ so that $\lVert f \rVert_\infty + \lVert \nabla f \rVert_\infty \leq A$ for every $f \in \mathscr{B} \,$.
\end{remark}

\noindent Given a topological space $X$ and a subset $S$, we write $\overline{S}$ (respectively $\mathrm{int}(S) ,\, \partial S$) for its closure (respectively interior, boundary). 

\begin{definition}[Tame boundary]
\label{def: tame boundary}
A subset $S \subseteq X$ of a topological space $X$ has \textit{tame boundary} if the following equivalent conditions hold:
\begin{enumerate}[label=(\roman*)]
    \item\label{it: tame boundary - i} $\partial S \cap \mathrm{int}(\overline{S}) = \varnothing \,$,
    \item\label{it: tame boundary - ii} $X \backslash S \subseteq \overline{\mathrm{int}(X \backslash S)} \,$.
\end{enumerate}
\end{definition}

\begin{remark}
\label{rmk: tame boundary}
Since $\mathrm{int}(\overline{S}) \subseteq \overline{S} = \mathrm{int}(S) \sqcup \partial S$ and $\mathrm{int} (\overline{S}) =  X \backslash  \overline{\mathrm{int}(X \backslash S)} \,$, indeed
\begin{align*}
    \text{\ref{it: tame boundary - i}} \,\, \Longleftrightarrow \,\, \mathrm{int}(\overline{S}) \subseteq \mathrm{int}(S) \,\, \Longleftrightarrow \,\, \mathrm{int}( \overline{S}  ) \subseteq S \,\, \Longleftrightarrow\,\,    X \backslash  \overline{\mathrm{int}(X \backslash S)}  \subseteq S \,\, \Longleftrightarrow \,\,  \text{\ref{it: tame boundary - ii}} \,\, .
\end{align*}
\end{remark}

\begin{example}
\label{example: tame boundary}
Let $M$ be a smooth manifold without boundary.
\begin{enumerate}[label=(\alph*)]
    \item Every closed subset of $M$ (in particular $M$ itself) has tame boundary.
    \item Let $U \subseteq M$ be an open subset with smooth boundary, \ie $\overline{U}$ is a codimension-$0$ submanifold of $M$ with (possibly empty) boundary and $\mathrm{int}(\overline{U}) = U \,$. Then $U$ has tame boundary.
    \item $M \backslash \{q\}$ does not have tame boundary, for every $q \in M$.
\end{enumerate}
\end{example}

\noindent Lastly, given a continuous curve $\gamma : I \rightarrow X$ from an interval $I \subseteq \R$ (or $I = \bbS^1$) to a topological space $X$ and a subset $S \subseteq X \,$, we say that \textit{$\gamma$ intersects $S$} if $\gamma^{-1}(S) \not= \varnothing \,$.

We are now ready to state our main result.

\begin{MainThm}
\label{mthm: rescaling Zoll contact form}
Let $(\Sigma,\alpha)$ be a Zoll contact manifold and $\mathscr{B} \subseteq \Cinfty(\Sigma)$ be a $C^1$-bounded subset. Then there exists a constant $\delta = \delta(\mathscr{B}) > 0$ with the following significance. For every non-empty subset $D \subseteq \Sigma$ with tame boundary and every $f \in \mathscr{B}$ satisfying
\begin{enumerate}[label=\arabic*)]
    \item\label{it: mthm rescaling Zoll contact form - C^0 bound} $\lVert f \rVert_\infty := \max_\Sigma |f| \leq \delta$ and
    \item\label{it: mthm rescaling Zoll contact form - f=0 outside D} $f \equiv 0$ on $\Sigma \backslash D \,$,
\end{enumerate}
there exists a periodic Reeb orbit of $(\Sigma, \, e^f \alpha)$ intersecting $D$.
\end{MainThm}

\begin{remark}
\label{rmk: mthm rescaling Zoll contact form}
The following remarks are in order.
\begin{enumerate}[label=(\roman*)]
    \item\label{it: mthm rescaling Zoll contact form - D=Sigma} For $D := \Sigma \,$, condition \ref{it: mthm rescaling Zoll contact form - f=0 outside D} is void. We obtain an existence result for Reeb orbits of $(\Sigma, \,e^f \alpha) \,$, for every $f \in \mathscr{B}$ with $\lVert f \rVert_\infty \leq \delta \,$.
    \item\label{it: mthm rescaling Zoll contact form - period} The proof of Theorem \ref{mthm: rescaling Zoll contact form} even reveals the following: Let $\tau(\alpha) > 0$ denote the common minimal Reeb period of $(\Sigma,\alpha) \,$. Given $\eps > 0$ and a $C^1$-bounded subset $\mathscr{B} \subseteq \Cinfty(\Sigma) \,$, there exists $\delta = \delta(\eps, \mathscr{B}) > 0$ with the significance stated in the theorem and so that additionally the Reeb orbit of $(\Sigma,e^f \alpha)$ intersecting $D$ has (not necessarily minimal) period $\tau$ with $|\tau - \tau(\alpha)| \leq \eps \,$.
    \item\label{it: mthm rescaling Zoll contact form - Reeb vf e^f alpha} Setting $\xi := \ker (\alpha) \,$, the Reeb vector field of $e^f \alpha$ is explicitly given by
    \[
    R_{e^f \alpha} = e^{-f} \big( R_\alpha + (\d \alpha|_{\xi})^{-1}(\d f|_{\xi})\big) \,\, ,
    \]
    where we use the isomorphism $\d \alpha|_\xi : \xi \overset{\cong}{\longrightarrow} \xi^* \,$. We see that $R_{e^f \alpha}$ depends on the derivative of $f$, so it is remarkable that Theorem \ref{mthm: rescaling Zoll contact form} only requires $f$ to be $C^0$-small (for prescribed $C^1$-bounds) instead of $C^1$-small.
\end{enumerate}
\end{remark}

\noindent As a byproduct of the proof of Theorem \ref{mthm: rescaling Zoll contact form}, we also obtain a mild multiplicity result, for which we now recall the appropriate terminology:
Two globally defined integral curves $\gamma_1 , \gamma_2: \R \rightarrow M$ of an autonomous vector field on a manifold $M$ are called 
\textit{geometrically equal} if any of the following equivalent conditions holds:
\begin{enumerate}[label=(\roman*)]
    \item $\gamma_1(\R) = \gamma_2(\R) \,$,
    \item $\gamma_1(\R) \cap \gamma_2(\R) \not= \varnothing \,$,
    \item $\gamma_2 = \gamma_1(\,\cdot\, + s) $ for some $s \in \R \,$.
\end{enumerate}
If $\gamma_1$ and $\gamma_2$ are not geometrically equal, we call them \textit{geometrically distinct}.

\begin{MainThm}
\label{mthm: multiplicity result}
Let $(\Sigma,\alpha)$ be a Zoll contact manifold. For given $C^1$-bounded subset $\mathscr{B}\subseteq \Cinfty(\Sigma)$ and $\eps > 0$ there exists a constant $\delta = \delta(\eps,\mathscr{B}) > 0$ with the following significance. For every $f \in \mathscr{B}$ with $\lVert f \rVert_\infty \leq \delta$ one of the following cases holds.
\begin{enumerate}[label=(\alph*)]
    \item\label{it: mthm multiplicity - two periodic orbits} $(\Sigma,e^f \alpha)$ has at least two geometrically distinct periodic Reeb orbits.
    \item\label{it: mthm multiplicity - one short orbit} $(\Sigma,e^f \alpha)$ has a periodic Reeb orbit of period at most $\eps $.
\end{enumerate}
\end{MainThm}

\noindent We do not claim any optimality regarding the number of periodic Reeb orbits of the perturbed contact form in the above theorem.

\subsection{Proof strategy}
\label{subsec: Introduction - Proof Strategy}

Theorems \ref{mthm: rescaling Zoll contact form} and \ref{mthm: multiplicity result} are rather direct consequences of Theorem \ref{mthm: Rabinowitz gradient flow line} in Section \ref{sec: Rabinowitz AF existence crit pts}, which is the actual main contribution of this paper. Briefly summarized, it asserts existence of gradient flow lines of the Rabinowitz action functional for every Hamiltonian sufficiently $C^0$-close to a defining Hamiltonian for the given Zoll contact manifold. Moreover, one can control the position of the cylinder component of the gradient flow line at $(s,t)=(0,0) \in \R \times \bbS^1 \,$. The proof of Theorem \ref{mthm: Rabinowitz gradient flow line} constitutes the bulk of this work and relies on a homotopy stretching argument for Rabinowitz-Floer gradient flow lines. The key step is showing that certain moduli spaces of gradient flow lines are non-empty (Proposition \ref{prop: moduli^H_[0,R](z) non-empty}). This requires an abstract perturbation result for Fredholm sections.

\subsection{Organization of the paper}
\label{subsec: Introduction - Organization}

\noindent In Section \ref{sec: Symplectization} we recall some well-known facts about the symplectization of a contact manifold.

Section \ref{sec: Electromagnetic Hamiltonians} contains a brief repetition of the interplay between Lagrangian and Hamiltonian dynamics and the proof of Theorem \ref{mthm: electromagnetic Hamiltonians - periodic solution} (the generalization of Theorem \ref{mthm: magnetic geodesics}). We try to be comprehensive and give the proof details in length to illustrate how one may generally apply Theorem \ref{mthm: rescaling Zoll contact form}.

In Section \ref{sec: Rabinowitz AF existence crit pts} we introduce the Rabinowitz action functional, state the actual main result (Theorem \ref{mthm: Rabinowitz gradient flow line}) and use it to derive both Theorems \ref{mthm: rescaling Zoll contact form} and \ref{mthm: multiplicity result} from it.

The rest of the paper is devoted to the proof of Theorem \ref{mthm: Rabinowitz gradient flow line}.

Section \ref{sec: Proof MThm - Preparation} is of preparatory nature. We define an interpolation functional and the corresponding moduli space of gradient flow lines. We also prove that the gradient flow lines in this moduli space satisfy uniform bounds on the Lagrange multiplier and energy and that the cylinder components lie in a common compact domain of the symplectization. These are the crucial bounds for the homotopy stretching argument to work.

The proof of Theorem \ref{mthm: Rabinowitz gradient flow line} is presented in Section \ref{sec: Proof MThm}, except for the key proposition (Proposition \ref{prop: moduli^H_[0,R](z) non-empty}), whose proof is deferred to Section \ref{sec: Non-Empty Moduli Spaces}.

Section \ref{sec: Outlook} contains a short outlook regarding the question on how to possibly adapt the arguments for more general (not necessarily Zoll) contact manifolds.

Two appendices are added. Appendix \ref{app sec: Cieliebak-Frauenfelder estimate} contains the proof of the lemma on which the uniform Lagrange multiplier bound rests. Although it is a straightforward adaption of a similar lemma in \cite{Cieliebak_Frauenfelder}, we decided to include its proof due to its high relevance for the overall argument and to illustrate why we need to work with $C^1$-data (see Definition \ref{def: C^1 datum}) in the first place.

In Appendix \ref{app sec: Banach Mfd Bcal} we define an atlas for the Banach manifold $\Bcal$ appearing in Section \ref{sec: Non-Empty Moduli Spaces} and carry out some local computations in $\Bcal \,$.

\subsection{Acknowledgments}
\label{subsec: Introduction - Acknowledgments}

I am indebted to my supervisor P.~Albers to bring to my awareness the homotopy stretching argument to find critical points for perturbations of a Morse-Bott functional. I would also like to thank A.~Abbondandolo, L.~Dahinden, U.~Frauenfelder and F.~Ruscelli for inspiring comments and further discussion. This research was funded by Deutsche Forschungsgemeinschaft through the Research Training Group RTG 2229 (281869850).

\subsection{AI disclosure statement}
\label{subsec: Introduction - AI Statement}

This work was completed without use of artificial intelligence.

\section{The Symplectization}
\label{sec: Symplectization}

\noindent We briefly introduce the terminology we use in the context of the symplectization of a contact manifold. We follow standard conventions.

\subsection{Liouville vector fields}
\label{subsec: Liouville vector fields}

Let $(W,\omega = \d \lambda)$ be an exact symplectic manifold. The \textit{Liouville vector field $X_\lambda$ corresponding to $\lambda$} is the vector field on $W$ defined by $\iota_{X_\lambda} \omega = \lambda \,$. For a given hypersurface $\Sigma \subseteq W \,$, the restriction $\lambda|_{\Sigma} \in \Omega^1(\Sigma)$ is a contact form on $\Sigma$ iff $X_\lambda$ is transverse to $\Sigma \,$.

Now suppose $\Sigma$ is a closed (\ie compact, without boundary) hypersurface and $X= X_\lambda$ is transverse to $\Sigma \,$. Then $\alpha := \lambda|_\Sigma \in \Omega^1(\Sigma)$ is a contact form and for some $0 < h_0 \leq \infty$ the flow $(\phi_X^t)_{t \in \R}$ of $X$ induces an embedding
\begin{equation}
\label{eq: Liouville vfields - Liouville flow}
\Phi : (-h_0,h_0) \times \Sigma \hookrightarrow W \comma \quad (h,x) \mapsto \phi_X^h(x) \,\, .
\end{equation}
The map $\Phi$ pulls back $\lambda$ to the $1$-form $\Phi^* \lambda = e^h \alpha$ on $(-h_0,h_0) \times \Sigma \,$, where $h$ denotes the coordinate in $(-h_0,h_0) \,$. Hence $\Phi^* \omega = \d (e^h \alpha) \,$.

\subsection{Symplectization}
\label{subsec: Symplectization}

The \textit{symplectization} of a contact manifold $(\Sigma,\alpha)$ is the symplectic manifold
\[
(S \Sigma := \R \times \Sigma , \, \d (e^h \alpha)) \,\, ,
\]
where $h$ denotes the $\R$-coordinate. Notice that the coordinate vector field $\partial_h$ is the Liouville vector field for the distinguished primitive $e^h \alpha \,$, that is $X_{e^h \alpha} = \partial_h \,$. It is transverse to the hypersurface $\{0\} \times \Sigma \subseteq S \Sigma$ and the induced contact form $(e^h \alpha)|_{\{0\} \times \Sigma}$ is just $\alpha \,$. We will identify $(\Sigma,\alpha)$ with the hypersurface $\{0\} \times \Sigma \subseteq S \Sigma \,$. 

By the previous discussion, if $\Sigma$ is a closed hypersurface in an exact symplectic manifold $(W, \omega=\d \lambda)$ to which $X_\lambda$ is transverse, the Liouville flow \eqref{eq: Liouville vfields - Liouville flow} induces an exact symplectic diffeomorphism between an open neighborhood of $\Sigma$ in $(S \Sigma,e^h \alpha) \,$, where $\alpha := \lambda|_\Sigma \,$, and an open neighborhood of $\Sigma$ in $(W,\lambda) \,$.

\begin{example}
\label{example: Symplectization in cotangent bundle}
Let $M$ be a closed $n$-manifold. The Liouville vector field of the canonical $1$-form $\lambdacan = \sum_{i=1}^{n} p_i \, \d q_i$ on the cotangent bundle $(T^*M , \omegacan = \d \lambdacan)$ is the fiberwise radial vector field $X_{\mathrm{rad}} = \sum_{i=1}^n p_i \, \partial_{p_i} \,$. Its flow is given by $\phi^t_{X_{\mathrm{rad}}}(q,p) = (q, e^t p)$ and in particular complete. Let $\Sigma \subseteq T^* M$ be a \textit{fiberwise starshaped} hypersurface, \ie $X_{\mathrm{rad}}$ is transverse to $\Sigma$ and in each fiber $T_q^*M$ every ray emanating from the origin intersects $\Sigma$ in precisely one point. Then the fiberwise radial Liouville flow \eqref{eq: Liouville vfields - Liouville flow} yields a globally defined exact symplectic embedding
\begin{equation}
\label{eq: embedding symplectization in cotangent bdl}
\Phi : \, (S \Sigma , e^h \alpha) \hookrightarrow (T^*M, \lambdacan ) \comma \,\,\, \Phi(h,(q,p)) = (q, e^h p) \comma \quad \alpha := \lambdacan|_\Sigma \,\, ,
\end{equation}
onto $\Phi(S \Sigma) = T^*M \backslash O_{T^*M} \,$, the cotangent bundle minus the zero section.
\end{example}

\noindent Given a function $f \in \Cinfty(\Sigma) $ on a closed contact manifold $(\Sigma,\alpha) $, its graph $\mathrm{Gr}(f) \subseteq \R \times \Sigma$ is a hypersurface in the symplectization which is transverse to the Liouville vector field $\partial_h $ (see also Figure \ref{fig: graph symplectization}), so $(e^h \alpha)|_{\mathrm{Gr}(f)} $ is a contact form on $\mathrm{Gr}(f)$.
The next lemma, whose proof is immediate, expresses the idea that one can either work on the graph of a function $f$ with the well-behaved contact form $(e^h \alpha)|_{\mathrm{Gr}(f)}$ or instead work directly on $\Sigma$ with the perturbed contact form $e^f \alpha \,$.

\begin{lemma}
\label{lem: Reeb orbits Gr(f) <--> e^f alpha Reeb orbits}
The map $f \times \id_\Sigma : (\Sigma , e^f \alpha) \overset{\cong}{\longrightarrow} (\mathrm{Gr}(f) , (e^h \alpha)|_{\mathrm{Gr}(f)})$ is a strict contactomorphism whose inverse is the restriction of the projection $\pi_\Sigma : \R \times \Sigma \rightarrow \Sigma$ to $\mathrm{Gr}(f) \,$. In particular, $f \times \id_\Sigma$ conjugates the Reeb flows of $(\Sigma,e^f \alpha)$ and $(\mathrm{Gr}(f), (e^h \alpha)|_{\mathrm{Gr}(f)}) \,$.
\end{lemma}

\begin{figure}
    \centering
    \includegraphics[width=0.5\linewidth]{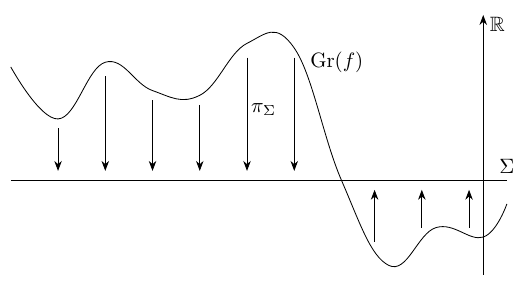}
    \caption{Graph of function $f$ in the symplectization $S \Sigma = \R \times \Sigma \,$. We denote by $\pi_\Sigma : \R \times \Sigma \rightarrow \Sigma$ the projection.}
    \label{fig: graph symplectization}
\end{figure}

\subsection{Level sets of Hamiltonians}
\label{subsec: Level Sets of Hamiltonians}

Let $(W,\omega = \d \lambda)$ be an exact symplectic manifold and $H \in \Cinfty(W)$ an autonomous Hamiltonian. Its \textit{Hamiltonian vector field} $X_H$ is defined by $\iota_{X_H} \omega = - \d H \,$. The Hamiltonian vector field and the Liouville vector field are related by
\[
\lambda(X_H) = \omega( X_\lambda , X_H) = \d H (X_\lambda) \, .
\]
If $\d H(X_\lambda) \not= 0$ on the level set $\Sigma := H^{-1}(c) \,$, then $(\Sigma,\alpha := \lambda|_{\Sigma})$ is a contact hypersurface. The \textit{characteristic line bundle} $\ker(\omega|_\Sigma) \subseteq T \Sigma$ is spanned both by the Reeb vector field $R_\alpha$ as well as the Hamiltonian vector field $X_H|_\Sigma $ (which is tangent to $\Sigma$), so that $R_\alpha$ and $X_H|_\Sigma$ are collinear. More precisely
\begin{equation}
\label{eq: X_H = lambda(X_H) Reeb}
X_H = \lambda(X_H) \, R_\alpha =  \d H(X_\lambda) \, R_\alpha \quad \text{on } \Sigma \, . 
\end{equation}
In particular, $H$-Hamiltonian flow lines on $\Sigma$ are reparametrizations of Reeb orbits.

\begin{example}
\label{example: Level Set Hamiltonian H = h - f(x)}
Given $f \in \Cinfty(\Sigma) \,$, the Hamiltonian $H_f \in \Cinfty(S \Sigma) \comma H_f(h,x) := h - f(x) \,$, has regular zero level set $H_f^{-1}(0) = \mathrm{Gr}(f)$ and its Hamiltonian vector field agrees with the Reeb vector field on $\mathrm{Gr}(f)$ because $\d H(\partial_h) \equiv 1 \,$. By Lemma \ref{lem: Reeb orbits Gr(f) <--> e^f alpha Reeb orbits}, the $H_f\,$-Hamiltonian flow lines are conjugated to the Reeb orbits of $(\Sigma,e^f \alpha) \,$.
\end{example}

\section{Proof of Theorem \ref{mthm: magnetic geodesics}}
\label{sec: Electromagnetic Hamiltonians}

\noindent We will illustrate Theorem \ref{mthm: rescaling Zoll contact form} by proving a generalization of Theorem \ref{mthm: magnetic geodesics}.

\subsection{Electromagnetic Hamiltonians}
\label{subsec: Electromagnetic Hamiltonians}

\noindent The following discussion is a brief summary of what we need below. For a more complete overview regarding electromagnetic Lagrangians and the interplay between the Lagrangian and Hamiltonian formalism we refer the reader to \cite{Lagrangian_Minimizers}, \cite{Ana_Cannas_Silva}, \cite{Contreras_Macarini_Paternain} and \cite{Abbondandolo}.

Let $(M,g)$ be a closed, connected Riemannian manifold of dimension $n \geq 2$. To a 1-form $\theta \in \Omega^1(M) \,$, called \textit{magnetic potential}, and a function $V \in \Cinfty(M) \,$, called \textit{mechanic potential}, we associate the electromagnetic Lagrangian
\[
L = L_{(g,\theta,V)}  : TM \rightarrow \R \comma \quad L(q,v) := \tfrac{1}{2} g_q(v,v) - \theta_q(v) - V(q) \,\, .
\]
Define a bundle endomorphism $F = F^{(g,\theta)} : TM \rightarrow TM \,$, called \textit{Lorentz force}, via
\[
g_q (F_q(v_1), v_2) = \d \theta_q (v_1,v_2) \qquad \Forall v_1,v_2 \in T_q M \,\, .
\]
A curve $\gamma : (t_0,t_1) \rightarrow M$ solves the Euler-Lagrange equations for $L$ if and only if
\begin{equation}
\label{eq: Euler-Lagrange electromagnetic Lagrangian}
\frac{\D \dot{\gamma} }{\d t} = F_\gamma(\dot{\gamma}) - \nabla V(\gamma) \,\, ,
\end{equation}
where both the covariant derivative $\frac{\D}{\d t}$ and the gradient $\nabla V$ are understood with respect to the metric $g$. 
A curve $\gamma$ solving \eqref{eq: Euler-Lagrange electromagnetic Lagrangian} describes the trajectory of a charged particle moving in $M$ under the influence of the conservative force $-\nabla V$ and the Lorentz force $F$ of the exact magnetic field $\d \theta \,$.

The energy function $E : TM \rightarrow \R$ associated to $L$ is given by $E(q,v) = \tfrac{1}{2} |v|^2 + V(q) \,$. For an Euler-Lagrange solution $\gamma \,$, the function $E(\gamma,\dot{\gamma})$ is constant and we define the \textit{energy $E(\gamma) = E_{(g,V)}(\gamma) \in \R$ of $\gamma$} to be
\begin{equation*}
E(\gamma) := E(\gamma(t),\dot{\gamma}(t)) = \tfrac{1}{2}  |\dot{\gamma}(t) |^2 + V(\gamma(t)) \comma \quad \text{ for any } t \in (t_0,t_1) \, .
\end{equation*}
Let $g^*$ be the bundle metric on $T^*M$ dual to $g$, defined by
\[
g^* = g(\flat_g^{-1} \cdot , \flat_g^{-1}\cdot\,) \comma \text{ where } \, \flat_g : TM \overset{\cong}{\longrightarrow} T^*M \comma \, \flat_g(q,v) = (q, g_q(v,\,\cdot\,)) \, .
\]
The Fenchel dual of the Lagrangian $L$ is the electromagnetic Hamiltonian 
\begin{equation}
\label{eq: electromagnetic Hamiltonian}
H = H_{(g,\theta,V)}  : T^*M \rightarrow \R \comma \quad H(q,p) = \tfrac{1}{2} | p + \theta_q |_{g^*}^2 + V(q) \,\, .
\end{equation}
The composition of $H$ with the Legendre transform of $L$ agrees with the energy function $E : TM \rightarrow \R \,$. Denote by $\omegacan = \d \lambdacan = \sum_{i=1}^{n} \d p_i \wedge \d q_i$ the canonical symplectic form on $T^*M \,$. Since the Legendre transform of $L$ conjugates the Euler-Lagrange flow and the Hamiltonian flow of $H$ with respect to $\omegacan \,$, see \cite[Prop.~1-4.1]{Lagrangian_Minimizers}, there is a one-to-one correspondence between Hamiltonian flow lines on $T^*M$ contained in $H^{-1}(E_0)$ and solutions of \eqref{eq: Euler-Lagrange electromagnetic Lagrangian} of energy $E_0 \,$, given by projecting a Hamiltonian flow line to $M$ via the basepoint projection $\pi_{T^*M} : T^*M \rightarrow M$. Moreover, a Hamiltonian flow line is periodic if and only if its projection to $M$ is periodic; and if this is the case, the minimal periods coincide.

\subsection{Results}
\label{subsec: Electromagnetic Hamiltonians - Results}

\noindent We first remind the reader of the definition of a (Riemannian) Zoll manifold.

\begin{definition}
\label{def: Riemannian Zoll mfd}
A closed, connected Riemannian manifold $(M,g)$ is called \textit{Zoll manifold} if all its geodesics are closed with same Riemannian length. 
\end{definition}

\begin{remark}
\label{rmk: Riemannian Zoll mfd}
Equivalently, all unit-speed geodesics are periodic with same minimal period. We emphasize that we do not require the geodesics to be simple.
\end{remark}

\begin{example}
\label{example: Riemannian Zoll mfd}
The standard example of a Zoll manifold is the $n$-sphere $(\bbS^n, g_{\bbS^n})$ with the round metric $g_{\bbS^n} \,$, all whose geodesics are great circles of length $2 \pi \,$. However, as Zoll \cite{Zoll} discovered,
there are many more Zoll metrics on $\bbS^2 \,$. Indeed, in cylindrical coordinates $(z,\varphi) \in (-1,1) \times \R/2\pi \Z \,$, every metric of the form
\[
g = \frac{(1+f(z))^2}{1 - z^2} \,\d z^2 + (1-z^2) \,\d \varphi^2 \,\, ,
\]
where $f : [-1,1] \rightarrow (-1,1)$ is an odd smooth function with $f(\pm 1) = 0 \,$, is Zoll, see \cite{Zoll_Surface}.
\end{example}

\noindent Given a Zoll manifold $(M,g_0)$ we are now interested in how the dynamics of the geodesic equation $\frac{\D}{\d t} \dot{\gamma} = 0$ changes if we introduce a $C^0$-small perturbation 
\[
(g,\theta,V) \in \mathrm{Met}(M) \times \Omega^1(M) \times \Cinfty(M) 
\]
and consider the solutions of the corresponding Euler-Lagrange equation \eqref{eq: Euler-Lagrange electromagnetic Lagrangian}. That is, we distort the metric and turn on an exact magnetic field and a conservative force field.

Generalizing Definition \ref{def: C^1-bounded subset}, given a smooth vector bundle $E$ over the closed base manifold $M$, a subset of the space of smooth sections $\Gamma(E) $ is called \textit{$C^1$-bounded} if it is bounded in any norm inducing the $C^1$-topology on $\Gamma(E) \,$. In particular, we may speak of $C^1$-bounded subsets of the space of smooth metrics $\mathrm{Met}(M)$ (which is a subset of $\Gamma(\mathrm{Sym}^2(T^*M)$) respectively of the space of smooth 1-forms $\Omega^1(M)$.

Here is the announced generalization of Theorem \ref{mthm: magnetic geodesics}.

{
\theoremstyle{theorem}
\newtheorem{MThmPrime}{Theorem}
\renewcommand{\theMThmPrime}{\Alph{MThmPrime}'}
\begin{MThmPrime}
\label{mthm: electromagnetic Hamiltonians - periodic solution}
Let $(M,g_0)$ be a Zoll manifold of dimension $\dim(M) \geq 2 \,$. For every $E_0 > 0$ and every $C^1$-bounded subset
\[
\mathscr{B}_{C^1} \subseteq \mathrm{Met}(M) \times \Omega^1(M) \times \Cinfty(M)
\]
there exists a $C^0$-neighborhood $\mathscr{B}_{C^0} \subseteq \mathrm{Met}(M) \times \Omega^1(M) \times \Cinfty(M)$ of $(g_0,0,0)$
with the following significance. For every subset $\varnothing \not= B \subseteq M$ with tame boundary and every $(g,\theta,V) \in \mathscr{B}_{C^0} \cap \mathscr{B}_{C^1}$ with
\begin{align}
\label{eq: electromagnetic Hamiltonians - no perturbation on M-B}
g= g_0 \comma \theta = 0 \comma V = 0 \quad \text{ on } M \backslash B \, ,
\end{align}
there exists a periodic solution $\gamma$ to
\begin{equation}
\label{eq: Euler-Lagrange electromagnetic Lagrangian (g,theta,V)}
\frac{\D^g}{\d t} \dot{\gamma} = (F^{(g,\theta)})_\gamma(\dot \gamma) - \nabla^g V(\gamma) 
\end{equation}
of energy $E_{(g,V)}(\gamma) = E_0$ which intersects $B \,$.
\end{MThmPrime}
}

\noindent We emphasize that we only require $(g, \theta,V)$ to be $C^0$-close to the unperturbed system $(g_0, \theta_0 = 0, V_0 = 0)$ and still obtain periodic solutions of \eqref{eq: Euler-Lagrange electromagnetic Lagrangian (g,theta,V)}, although this differential equation depends heavily on the first derivatives of $g , \theta$ and $V$.
The proof of the theorem is postponed to Subsection \ref{subsec: Electromagnetic Hamiltonians - proof mthm}.

\begin{remark}
\label{rmk: electromagnetic Hamiltonians - periodic solution}
For $B := M \,$, condition \eqref{eq: electromagnetic Hamiltonians - no perturbation on M-B} is void and we obtain existence of a solution of \eqref{eq: Euler-Lagrange electromagnetic Lagrangian (g,theta,V)} of energy $E_0$ for every $(g,\theta,V) \in \mathscr{B}_{C^0} \cap \mathscr{B}_{C^1} \,$. This result is not new and fits into the existing literature as we will now briefly explain.

The \textit{Ma\~n\'e critical value} of a Tonelli Lagrangian $L : TM \rightarrow \R$ is defined by
\[
c(L) := \inf \Bigset{\kappa \in \R}{\int_0^T \big( L(\gamma(t), \dot{\gamma}(t)) + \kappa \, \big)  \d t \geq 0 \quad \Forall T > 0 \comma \gamma \in \Cinfty(\R/T\Z,M)} \,\, .
\]
It is well-known, see \cite[Thm. 6.1]{Abbondandolo}, that periodic Euler-Lagrange solutions of $L$ exist on every energy level above the Ma\~n\'e critical value $c(L) \,$.

Now for given $E_0 > 0 \,$, if $(g-g_0,\theta,V)$ is sufficiently $C^0$-small, then the Ma\~n\'e critical value of the electromagnetic Lagrangian $L_{(g,\theta,V)}$ is smaller than $E_0 \,$, that is
\begin{equation}
  \label{eq: Electromagnetic Hamiltonians - c(L) < E_0}  
  c(L_{(g,\theta,V)}) < E_0 \,\, ,
\end{equation}
so that a periodic solution to \eqref{eq: Euler-Lagrange electromagnetic Lagrangian (g,theta,V)} of energy $E_0$ exists. To see that \eqref{eq: Electromagnetic Hamiltonians - c(L) < E_0} holds, for arbitrary $(q,v) \in TM$ we estimate
\begin{align*}
    L_{(g,\theta,V)}(q,v) &\geq \tfrac{1}{2} |v|^2_g - \lVert \theta \rVert_{\infty,g^*} \, |v|_{g} - \lVert V \rVert_\infty 
    =  \tfrac{1}{2} \big( |v|_{g} - \lVert \theta \rVert_{\infty,g^*} \big)^2 - \tfrac{1}{2} \lVert \theta\rVert_{\infty,g^*}^2 - \lVert V \rVert_\infty \\
    &\geq - \tfrac{1}{2} \lVert \theta\rVert_{\infty,g^*}^2 - \lVert V \rVert_\infty  \,\, ,
\end{align*}
showing that $c(L_{(g,\theta,V)}) \leq \tfrac{1}{2} \lVert \theta\rVert_{\infty,g^*}^2 + \lVert V \rVert_\infty \,$. The right-hand side is arbitrarily small, provided $(g-g_0, \theta,V)$ is sufficiently $C^0$-small. This proves \eqref{eq: Electromagnetic Hamiltonians - c(L) < E_0}.

In the Lagrangian approach, no bound on the first derivatives is required at all. This indicates that Theorem \ref{mthm: electromagnetic Hamiltonians - periodic solution} is not optimal in this regard.
\end{remark}

\subsection{Two general lemmata to apply Theorem \ref{mthm: rescaling Zoll contact form}}
\label{subsec: Electromagnetic Hamiltonians - general lemmata}

In this subsection, we prove two lemmata which are frequently useful when one wants to apply Theorem \ref{mthm: rescaling Zoll contact form}.

\begin{lemma}
\label{lem: preimage tame boundary}
Let $\pi : X \rightarrow Y$ be a continuous, open map between topological spaces $X$ and $Y$. If $B \subseteq Y$ has tame boundary, then $\pi^{-1}(B) \subseteq X $ has tame boundary.
\end{lemma}

\begin{proof}
    Since $\pi$ is continuous and open, the operations of taking the interior (respectively closure/boundary) and taking the preimage commute. Thus
    \[
    \partial (\pi^{-1}(B)) \cap \interior(\overline{\pi^{-1}(B)}) = \pi^{-1} (\partial B \cap \interior(\overline{B}) ) = \varnothing \,\, .
    \]
\end{proof}

\noindent The next lemma is concerned with the relation between $f \in \Cinfty(\Sigma)$ and $H \in \Cinfty(\R \times \Sigma)$ if $\mathrm{Gr}(f) = H^{-1}(0) \,$. As usual, we denote by $\partial_h$ the coordinate vector field on $\R \times \Sigma$ in $\R$-direction.

\begin{lemma}
\label{lem: Gr(f) C^0 small <--> f C^0-small}
Let $\Sigma$ be a closed manifold and let $ c , h_0 > 0$ be positive constants.
\begin{enumerate}[label=(\alph*)]
    \item\label{it: level set is graph} If $H \in \Cinfty(\R \times \Sigma)$ satisfies $\d H(\partial_h) \geq c $ on $[-h_0,h_0] \times \Sigma$
 and $|H(0,z)| < h_0 c$ for every $z \in \Sigma \,$, then
    \begin{equation}
    \label{eq: H-level set = Gr(f)}
    H^{-1}(0) \cap ((-h_0,h_0) \times \Sigma) = \mathrm{Gr}(f)
    \end{equation}
    is the graph of a uniquely determined function $f \in \Cinfty(\Sigma) \,$.
    \item\label{it: level set is graph - C^1 bound f <--> C^1-bound H}
    For every $C^1$-bounded subset $\mathscr{B} \subseteq \Cinfty([-h_0,h_0] \times \R)$ there exists a $C^1$-bounded subset $\mathscr{B}_{\Sigma} \subseteq \Cinfty(\Sigma)$ with the following significance. Given $H \in \mathscr{B}$ and $f \in \Cinfty(\Sigma)$ so that $\d H (\partial_h) \geq c$ on $[-h_0,h_0] \times \Sigma$ and \eqref{eq: H-level set = Gr(f)} holds, then $f \in \mathscr{B}_\Sigma \,$.
\end{enumerate}
\end{lemma}

\begin{proof}
For part \ref{it: level set is graph} notice that for every fixed $z \in \Sigma$ the function $(-h_0,h_0) \rightarrow \R \comma h \mapsto H(h,z) \,$, is strictly increasing with a unique zero $h_z \in (-h_0,h_0)$ since $H(0,z) \in (- h_0 c , h_0 c)$ and $\tfrac{\d}{\d h} H(h,z) \geq c \,$. Thus the submanifold $H^{-1}(0) \cap ((-h_0,h_0) \times \Sigma)$ intersects $\R \times \{z\}$ transversally in precisely one point, for every $z \in \Sigma$, so that it must be the graph of a smooth function, see \cite[Thm.~6.32]{Lee}.

For part \ref{it: level set is graph - C^1 bound f <--> C^1-bound H}, we fix a Riemannian metric on $\Sigma$ which allows us to speak of the norm of the differential of a function on $\Sigma$ respectively on $[-h_0,h_0] \times \Sigma \,$. Given a constant $D_1 > 0$ we must find $D_2 > 0$ so that, for every $H\in \Cinfty([-h_0,h_0] \times \Sigma)$ and every $f \in \Cinfty(\Sigma)$ satisfying \eqref{eq: H-level set = Gr(f)}, it holds
\[
\lVert \d H \rVert_\infty \leq D_1 \text{ and } \d H(\partial_h) \geq c \,\, \Longrightarrow \,\, \lVert \d f \rVert_\infty \leq D_2 \,\, .
\]
(Note that the $C^0$-norm of $f$ is bounded by $h_0$ due to \eqref{eq: H-level set = Gr(f)}.) We claim that $D_2 := c^{-1}D_1$ is a suitable constant. Indeed, given $z \in \Sigma$ and $v \in T_z \Sigma$ with $|v| \leq 1$, we know that $\d f_z (v) \, \partial_h + v \in T_{(f(z),z)} \mathrm{Gr}(f) = \ker \d H_{(f(z),z)} \,$, hence
\[
0 = \d f_z (v) \, \d H_{(f(z),z)}(\partial_h) + \d H_{(f(z),z)}(v)
\]
and therefore
\begin{align*}
    |\d f_z (v)| = \big| - \frac{1}{\d H_{(f(z),z)} (\partial_h)} \, \d H_{(f(z),z)} (v)\big| \leq \frac{ \lVert \d H \rVert_\infty }{c} \leq D_2 \,\, . 
\end{align*}
\end{proof}

\subsection{Proof of Theorem \ref{mthm: electromagnetic Hamiltonians - periodic solution}}
\label{subsec: Electromagnetic Hamiltonians - proof mthm}

We fix a Zoll Riemannian manifold $(M,g_0)$ of dimension $n \geq 2$ as well as $E_0 > 0$ and let
\begin{align*}
    \Sigma := \set{(q,p) \in T^*M}{ \tfrac{1}{2} |p|^2_{g_0^*} = E_0}
\end{align*}
denote the $g_0$-cotangent sphere bundle of radius $\sqrt{2E_0} \,$, endowed with the contact form $\alpha := \lambdacan|_{\Sigma} \,$. It is well-known, see \cite[Thm.~1.5.2]{Geiges_Contact_Topology}, that the Reeb orbits of $(\Sigma,\alpha)$ are reparametrizations with factor $2E_0$ of the Hamiltonian orbits for the kinetic Hamiltonian $(q,p) \mapsto \tfrac{1}{2} |p|_{g_0^*}^2 \,$. The latter project to the geodesics of $(M,g_0)$ of energy $E_0 \,$. Therefore, since $(M,g_0)$ is Zoll Riemannian, $(\Sigma,\alpha)$ is Zoll contact.

 Justified by Example \ref{example: Symplectization in cotangent bundle}, we identify the symplectization $S \Sigma$ with the cotangent bundle minus the zero section, $S \Sigma \cong T^*M \backslash O_{T^*M} \,$, via the exact symplectomorphism displayed in \eqref{eq: embedding symplectization in cotangent bdl}. In particular we identify $\Sigma \subseteq T^*M$ with $\{0\} \times \Sigma \subseteq S \Sigma \,$. Via the identification \eqref{eq: embedding symplectization in cotangent bdl}, we may consider the electromagnetic Hamiltonian $H_{(g,\theta,V)} \in \Cinfty(T^*M)$ from \eqref{eq: electromagnetic Hamiltonian} as being defined on $S \Sigma \,$.

\begin{lemma}
\label{lem: Electromagnetic Hamiltonians - g mapsto g^*}
Given a metric $g$ on $M$, denote by $\flat_g : TM \overset{\cong}{\longrightarrow} T^*M \comma (q,v) \mapsto g_q(v,\,\cdot\,) \,$, the induced bundle isomorphism and by $g^*$ the bundle metric on $T^*M$ dual to $g$, defined via $g^* = g(\flat_g^{-1} \cdot \, , \flat_g^{-1} \cdot \,) $. Then the map 
\begin{align*}
    \mathrm{Met}(M) = \mathrm{BundleMet}(TM) \longrightarrow \mathrm{BundleMet}(T^*M) \comma \quad g \mapsto g^* \,\, ,
\end{align*}
is continuous in the respective $C^0$-topologies.
\end{lemma}
\begin{proof}
   The above map is the pushforward map of sections induced by the continuous fiber-preserving map between open subsets of the bundles $\mathrm{Sym}^2(T^*M)$ and $\mathrm{Sym}^2(TM) $ over the compact base $M$, sending an inner product $g$ on $T_q M$ to the inner product $g^* = g(\flat_g^{-1}\cdot\,,\flat_g^{-1} \cdot)$ on $T^*_q M \,$.
\end{proof}

\begin{lemma}
\label{lem: Electromagnetic Hamiltonians - key lemma}
Abbreviate $\mathfrak{M} := \mathrm{Met}(M) \times \Omega^1(M) \times \Cinfty(M) \,$.
\begin{enumerate}[label=(\alph*)]
    \item\label{it: Electromagnetic Hamiltonians key lemma - dH(partial_h)}
    There exist $c = c(E_0) > 0$ and a $C^0$-neighborhood $\mathscr{B}^{C^0}_M \subseteq \mathfrak{M}$ of $(g_0,0,0)$ with
    \[
     \d H_{(g,\theta,V)}(\partial_h ) \geq c \,\,\, \text{on } [-1,1] \times \Sigma \comma \quad \Forall (g,\theta,V) \in \mathscr{B}^{C^0}_M \,\, .
    \]
    \item\label{it: Electromagnetic Hamiltonians key lemma - C^0 small} For every $\delta > 0$ there exists a $C^0$-neighborhood $\mathscr{B}^{C^0}_M(\delta) \subseteq \mathfrak{M}$ of $(g_0,0,0)$ so that 
    \[
    \sup_{ \Sigma} |H_{(g,\theta,V)} - E_0| < \delta \qquad \Forall (g,\theta,V) \in \mathscr{B}^{C^0}_M(\delta) \,\, .
    \]
    \item\label{it: Electromagnetic Hamiltonians key lemma - C^1}
    Given a $C^1$-bounded subset $\mathscr{B}^{C^1}_M \subseteq \mathfrak{M}$, there exist a $C^0$-neighborhood $\mathscr{B}_{M}^{C^0} \subseteq \mathfrak{M}$ of $(g_0,0,0)$ and a $C^1$-bounded subset $\mathscr{B}^{C^1}_{S \Sigma} \subseteq C^\infty([-1,1] \times \Sigma)$ with
    \[
    H_{(g,\theta,V)}|_{[-1,1] \times \Sigma} \in \mathscr{B}_{S \Sigma}^{C^1} \qquad \Forall (g,\theta,V) \in \mathscr{B}^{C^1}_M \cap \mathscr{B}^{C^0}_M \, .
    \]
\end{enumerate}
\end{lemma}

\begin{proof}
    For notational ease let us stipulate that, by default, the norm of any co- or contravariant tensor is understood with respect to our fixed Zoll metric $g_0 \,$.
    
    For \ref{it: Electromagnetic Hamiltonians key lemma - dH(partial_h)}, we abridge $H := H_{(g,\theta,V)} \,$. At $(h_0,(q,p)) \in [-1,1] \times \Sigma \,$, which corresponds to the point $(q,e^{h_0}p) \in T^* M \,$, we compute 
    \begin{align*}
        \d H(\partial_h) \big|_{(h_0,(q,p))} &= \frac{\d}{\d h} \Big|_{h_0}  H(q, e^{h} p) =  \frac{\d}{\d h} \Big|_{h_0} \Big( \tfrac{1}{2} g^*(e^h p + \theta_q , e^h p +\theta_q) + V(q) \Big) \\
        &= e^{2 h_0} g^*(p,p) + e^{h_0} g^*(p,\theta_q) \,\, .
    \end{align*}
   This can be estimated from below by
    \begin{align*}
        &\hspace*{5mm}e^{2 h_0} g^*(p,p) + e^{h_0} g^*(p,\theta_q) \\
        &=e^{2h_0} \Big( g_0^*(p,p) + (g^* - g_0^*)(p,p) \Big) 
        \hspace*{4mm}+ e^{h_0} \Big( g_0^*(p,\theta_q) + (g^*-g_0^*)(p,\theta_q)\Big) \\
        &\geq e^{ 2 h_0} \big( |p|^2 - \lVert g^*-g_0^* \rVert_\infty \, |p|^2 \big) - e^{h_0} \big( \lVert \theta \rVert_\infty \, |p| + \lVert g^*-g_0^* \rVert_\infty \, \lVert \theta \rVert_\infty \, |p| \big) \\
        &\geq  2 E_0 \, e^{-2} \, ( 1 - \lVert g^*- g^*_0 \rVert_\infty) -  \sqrt{2E_0}  \, e \, (\lVert \theta \rVert_\infty + \lVert g^* - g_0^* \rVert_\infty \, \lVert \theta \rVert_\infty) \,\, ,
    \end{align*}
    using that $h_0 \in [-1,1]$ and $\tfrac{1}{2} |p|^2 = \tfrac{1}{2} g_0^*(p,p) = E_0$ since $(q,p) \in \Sigma \,$.
    The last term can be bounded from below by $c(E_0) := E_0 \, e^{-2} \,$, provided $(g - g_0, \theta,V)$ is sufficiently $C^0$-small (taking also Lemma \ref{lem: Electromagnetic Hamiltonians - g mapsto g^*} into account). This shows part \ref{it: Electromagnetic Hamiltonians key lemma - dH(partial_h)}.

    To prove part \ref{it: Electromagnetic Hamiltonians key lemma - C^0 small}, for arbitrary $(q,p) \in \Sigma$ we estimate
    \begin{align*}
       &\hspace*{6mm}| H_{(g,\theta,V)}(q,p) - E_0 | = \big| \tfrac{1}{2} g^*(p + \theta_q,p + \theta_q) +V(q) - \tfrac{1}{2} g_0^*(p,p) \big| \\
       &\leq \big| \tfrac{1}{2}(g^*-g_0^*)(p+\theta_q,p+\theta_q) \big| + \big| \tfrac{1}{2} g_0^*(p+\theta_q, p + \theta_q) - \tfrac{1}{2} g_0^*(p,p) \big| + \lVert V \rVert_\infty \\
       &\leq \lVert g^* - g_0^* \rVert_\infty \, \big( \tfrac{1}{2} |p|^2 + \lVert \theta \rVert_\infty \, |p| + \tfrac{1}{2} \lVert \theta \rVert_\infty^2 \big) + \lVert \theta \rVert_\infty \, |p| + \tfrac{1}{2} \lVert \theta \rVert_\infty^2 + \lVert V \rVert_\infty \\
       &=  \lVert g^* - g_0^* \rVert_\infty \big(  E_0 + \lVert \theta \rVert_\infty \sqrt{2E_0} + \tfrac{1}{2} \lVert \theta \rVert_\infty^2  \big) + \lVert \theta \rVert_\infty \, \sqrt{2E_0} + \tfrac{1}{2} \lVert \theta \rVert_\infty^2 + \lVert V \rVert_\infty \,\, .
    \end{align*}
    The right-hand side can be made arbitrarily small, provided $(g-g_0,\theta,V)$ is sufficiently $C^0$-small. This shows part \ref{it: Electromagnetic Hamiltonians key lemma - C^0 small}.
    
    In canonical coordinates on $T^*M$ induced by local coordinates on $M$ we have
    \begin{equation}
    \label{eq: Electromagnetic Hamiltonians key lemma - H coordinate representation}
    H_{(g,\theta,V)} (q,p) = \tfrac{1}{2} \sum_{i,j=1}^{n} g^{ij}(q) (p_i + \theta_i(q)) \, (p_j + \theta_j(q)) + V(q)
    \end{equation}
    with inverse matrix $(g^{ij})_{ij} := (g_{ij})_{ij}^{-1} \,$. Since $\partial_{q_k} g^{ij} = - \sum_{r,s=1}^n g^{i r} \, (\partial_{q_k} g_{rs} ) \, g^{sj} \,$, we can bound the $C^1$-norm of $(g^{ij})_{ij}$ in terms of the $C^0$-norm of $(g^{ij})_{ij}$ and the $C^1$-norm of $(g_{ij})_{ij} \,$. Then \eqref{eq: Electromagnetic Hamiltonians key lemma - H coordinate representation} shows that the $C^1$-norm of $H_{(g,\theta,V)}$ on $[-1,1] \times \Sigma \subseteq T^*M$ can be bounded in terms of the $C^1$-norms of $(g,\theta, V)$ and the $C^0$-norm of $g^*$ (which is bounded uniformly for sufficiently $C^0$-small $g-g_0$). This concludes part \ref{it: Electromagnetic Hamiltonians key lemma - C^1}.
\end{proof}

\noindent We can now finish the proof of Theorem \ref{mthm: electromagnetic Hamiltonians - periodic solution}.

\begin{proof}[Proof of Theorem \ref{mthm: electromagnetic Hamiltonians - periodic solution}]
 Let $\mathscr{B}^{C^1}_{M} \subseteq \mathrm{Met}(M) \times \Omega^1(M) \times \Cinfty(M)$ be an arbitrary given $C^1$-bounded subset. For this $\mathscr{B}_M^{C^1}$ choose a $C^0$-neighborhood $\mathscr{B}_{M}^{C^0} \subseteq \mathrm{Met}(M) \times \Omega^1(M) \times \Cinfty(M)$ of $(g_0,0,0)$ and a $C^1$-bounded subset $\mathscr{B}^{C^1}_{S\Sigma } \subseteq \Cinfty([-1,1] \times \Sigma)$ as in Lemma \ref{lem: Electromagnetic Hamiltonians - key lemma} \ref{it: Electromagnetic Hamiltonians key lemma - C^1}. Making $\mathscr{B}_{M}^{C^0}$ smaller if necessary, we can choose $c = c(E_0) > 0$ so that Lemma \ref{lem: Electromagnetic Hamiltonians - key lemma} \ref{it: Electromagnetic Hamiltonians key lemma - dH(partial_h)} holds with $\mathscr{B}_M^{C^0}$ and $c$.

 By Lemma \ref{lem: Gr(f) C^0 small <--> f C^0-small} \ref{it: level set is graph - C^1 bound f <--> C^1-bound H}, we can choose a $C^1$-bounded subset $\mathscr{B}_{\Sigma}^{C^1} \subseteq \Cinfty(\Sigma)$ so that for every $H \in \Cinfty(S \Sigma)$ with $\d H (\partial_h ) \geq c$ on $[-1,1] \times \Sigma$ and every $f \in \Cinfty(\Sigma)$ 
 with
 \[
 H^{-1}(E_0) \cap ((-1,1) \times \Sigma ) = \mathrm{Gr}(f)
 \]
 it holds
 \begin{align}
 \label{eq: Electromagnetic Hamiltonians - proof thm C^1-bounded subsets}
 H|_{[-1,1] \times \Sigma} \in \mathscr{B}^{C^1}_{S \Sigma}  \,\, \Longrightarrow \,\, f \in \mathscr{B}^{C^1}_{\Sigma} \,\, .
 \end{align}
 Now choose a constant $\delta = \delta(\mathscr{B}^{C^1}_\Sigma) > 0$ as in Theorem \ref{mthm: rescaling Zoll contact form}. Choose a $C^0$-neighborhood $\mathscr{B}^{C^0}_M(\delta c) \subseteq \mathscr{B}^{C^0}_M$ of $(g_0,0,0)$ as in Lemma \ref{lem: Electromagnetic Hamiltonians - key lemma} \ref{it: Electromagnetic Hamiltonians key lemma - C^0 small}. We contend that $\mathscr{B}^{C^0}_M(\delta c)$ is the $C^0$-neighborhood of $(g_0,0,0)$ having the significance stated in Theorem \ref{mthm: electromagnetic Hamiltonians - periodic solution}.

 To this end, let $\varnothing \not=  B \subseteq M$ be a subset with tame boundary and assume $(g,\theta,V) \in \mathscr{B}^{C^1}_M \cap \mathscr{B}^{C^0}_M(\delta c)$ satisfies \eqref{eq: electromagnetic Hamiltonians - no perturbation on M-B}. We must show that there exists a periodic solution $\gamma$ of \eqref{eq: Euler-Lagrange electromagnetic Lagrangian (g,theta,V)} of energy $E_{(g,V)}(\gamma) = E_0$ which intersects $B$.
 
 We abbreviate $H := H_{(g,\theta,V)} \,$. Then 
 \begin{equation}
 \label{eq: electromagnetic Hamiltonians - proof thm dH partial_h bound}
  \d H(\partial_h) \geq c \qquad \text{on } [-1,1] \times \Sigma
 \end{equation}
 by choice of $c$ and $\mathscr{B}^{C^0}_M \ni (g,\theta,V) \,$. Moreover $\sup_\Sigma |H - E_0| < \delta c$ by choice of $\mathscr{B}^{C^0}_M (\delta c) \ni (g,\theta,V) \,$. By Lemma \ref{lem: Gr(f) C^0 small <--> f C^0-small} \ref{it: level set is graph} hence
 \begin{equation}
 \label{eq: electromagnetic Hamiltonians - proof thm graph f}
H^{-1}(E_0) \cap ((-1,1) \times \Sigma) \overset{\text{\eqref{eq: electromagnetic Hamiltonians - proof thm dH partial_h bound}}}{=} H^{-1}(E_0) \cap ((-\delta,\delta) \times \Sigma) = \mathrm{Gr}(f)
 \end{equation}
 is the graph of some $f \in \Cinfty(\Sigma) \,$. We make the following observations.
 \begin{enumerate}[label=(\roman*)]
     \item Clearly $\lVert f \rVert_\infty \leq \delta $ by \eqref{eq: electromagnetic Hamiltonians - proof thm graph f}.
     \item Writing $\pi : \Sigma \rightarrow M$ for the restriction of the basepoint projection $\pi_{T^*M} : T^*M \rightarrow M \,$, notice that the non-empty set 
     \[
     D:= \pi^{-1}(B) \subseteq \Sigma
     \]
     has tame boundary by Lemma \ref{lem: preimage tame boundary}.
     \item Because $(g-g_0,\theta,V)$ vanishes on $M \backslash B $ by assumption \eqref{eq: electromagnetic Hamiltonians - no perturbation on M-B}, it holds
\[
H = H_{(g,\theta,V)} = \tfrac{1}{2} |\cdot|_{g_0^*}^2 \qquad  \text{on } (\pi_{T^*M})^{-1}(M\backslash B) \,\, ,
\]
so that $\Sigma \backslash D = \pi^{-1}(M \backslash B) \subseteq H^{-1}(E_0) \,$. In the symplectization, this reads $\{0\} \times (\Sigma \backslash D) \subseteq H^{-1}(E_0)$ and
consequently $f$ vanishes on $\Sigma \backslash D \,$, using the defining equation \eqref{eq: electromagnetic Hamiltonians - proof thm graph f} for $f$. See also Figure \ref{fig: Graph Unit Cotangent Bundle}.
    \item Since $(g,\theta,V) \in \mathscr{B}_M^{C^1} \cap \mathscr{B}_M^{C^0} \,$, we have
    $H|_{[-1,1] \times \Sigma} \in \mathscr{B}_{S \Sigma}^{C^1}$ by choice of $\mathscr{B}_{S\Sigma}^{C^1} \,$. By choice of $\mathscr{B}_\Sigma^{C^1}$ and
    \eqref{eq: Electromagnetic Hamiltonians - proof thm C^1-bounded subsets}, \eqref{eq: electromagnetic Hamiltonians - proof thm dH partial_h bound} and \eqref{eq: electromagnetic Hamiltonians - proof thm graph f}, thus $f \in \mathscr{B}_\Sigma^{C^1}\,$.
 \end{enumerate}
Thus, having chosen $\delta = \delta(\mathscr{B}_\Sigma^{C^1}) $ as in Theorem \ref{mthm: rescaling Zoll contact form}, there exists a periodic Reeb orbit $\overline{\gamma}$ of $(\Sigma,e^f \alpha)$ intersecting $D \,$. By Lemma \ref{lem: Reeb orbits Gr(f) <--> e^f alpha Reeb orbits}, hence $(f \times \id_{\Sigma}) \circ \overline{\gamma}$ is a closed characteristic on $\mathrm{Gr}(f) \subseteq H^{-1}(E_0) $ intersecting $\R \times D$ and thus a reparametrization of a periodic $H$-Hamiltonian flow line $\overline{\gamma}_1$ in $H^{-1}(E_0)$ intersecting 
\[
\R \times D = \R \times \pi^{-1}(B) = \pi_{T^*M}^{-1}(B) \backslash O_{T^*M} \,\, ,
\]
where the last identity is of course understood via the identification \eqref{eq: embedding symplectization in cotangent bdl}.

Finally, from the discussion in Subsection \ref{subsec: Electromagnetic Hamiltonians}, the basepoint projection $\gamma := \pi_{T^*M } \circ \overline{\gamma}_1$ is a periodic solution of \eqref{eq: Euler-Lagrange electromagnetic Lagrangian (g,theta,V)} of energy $E_0$ intersecting $B \,$.
\end{proof}

\begin{figure}
    \centering
    \includegraphics[width=0.8\linewidth]{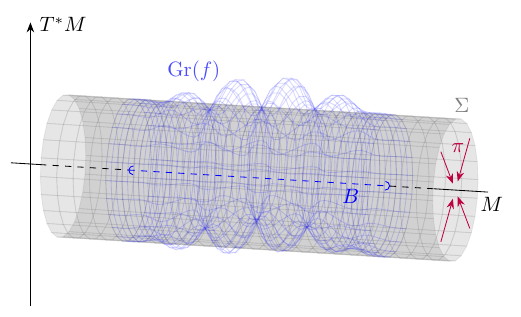}
    \caption{The cotangent bundle $T^*M$ minus the zero section $O_{T^*M} \cong M$ is identified with the symplectization of the cotangent sphere bundle $\Sigma = \set{(q,p) \in T^*M}{\tfrac{1}{2} |p|^2_{g_0^*} = E_0} \,$. Over $M \backslash B $ the graph $\mathrm{Gr}(f)=(H_{(g,\theta,V)})^{-1}(E_0) \cap ((-\delta,\delta) \times \Sigma)$ agrees with $\Sigma \,$.}
    \label{fig: Graph Unit Cotangent Bundle}
\end{figure}

\section{An Existence Theorem for Rabinowitz-Floer Gradient Flow Lines of Perturbed Hamiltonians}
\label{sec: Rabinowitz AF existence crit pts}

\noindent In Subsection \ref{subsec: Gradient Flow Lines for Perturbed Hamiltonians} we state the actual main result of this work, Theorem \ref{mthm: Rabinowitz gradient flow line}, and use it to derive Theorems \ref{mthm: rescaling Zoll contact form} and \ref{mthm: multiplicity result} from it in Subsections \ref{subsec: Proof mthm rescaling Zoll contact form} and \ref{subsec: Proof of mthm multiplicity result}. To this end, we first need to introduce some additional terminology in Subsections \ref{subsec: Defining Hamiltonians and C^1-Data} and \ref{subsec: Rabinowitz AF}.

\subsection{Defining Hamiltonians and $C^1$-data}
\label{subsec: Defining Hamiltonians and C^1-Data}

Let $(\Sigma,\alpha)$ be a closed, connected contact manifold of dimension $\dim(\Sigma) \geq 3 \,$. We work on the symplectization $(S \Sigma = \R \times \Sigma, \omega = \d \lambda )$ with primitive $\lambda = e^h \alpha $ and corresponding Liouville vector field $\partial_h \,$, where $h : \R \times \Sigma \rightarrow \R$ denotes the projection. Recall from Subsection \ref{subsec: Level Sets of Hamiltonians} that the Hamiltonian vector field of a Hamiltonian $H : S\Sigma \rightarrow \R$ is defined by $\iota_{X_H} \omega = - \d H$ and that $\d H(\partial_h) = \lambda(X_H) \,$. Recall moreover that we identify $\Sigma $ with $\{0\} \times \Sigma \subseteq S \Sigma \,$.

Motivated by \cite{Cieliebak_Frauenfelder}, we now introduce the notion of a defining Hamiltonian.

\begin{definition}[Defining Hamiltonian]
\label{def: defining Hamiltonian}
Let $H : S \Sigma \rightarrow \R$ be a smooth function.
\begin{enumerate}[label=(\alph*)]
    \item $H$ is called \textit{almost-defining Hamiltonian} if 
    \[
    H^{-1}(0) = \Sigma \,\text{ and }\, \d H(\partial_h) \equiv 1 \text{ on }\Sigma \, .
    \]
    \item $H$ is called \textit{defining Hamiltonian} if it is an almost-defining Hamiltonian and $\d H$ has compact support.
\end{enumerate}
\end{definition}

\begin{remark}
\label{rmk: almost-defining Hamiltonian}
Suppose $H$ is an almost-defining Hamiltonian. Then:
\begin{enumerate}[label=(\roman*)]
    \item $0$ is a regular value of $H$ with regular level set $\Sigma$.
    \item $H > 0$ on $(0,\infty) \times \Sigma$ and $H < 0$ on $(-\infty,0) \times \Sigma \,$.
    \item $X_H = R_\alpha$ along $\Sigma$ by \eqref{eq: X_H = lambda(X_H) Reeb}.
\end{enumerate}
The converse also holds:
\begin{enumerate}[label=(\roman*),resume]
    \item If $H \in \Cinfty(S\Sigma)$ satisfies $H^{-1}(0) = \Sigma$ and $X_H = R_\alpha$ along $\Sigma$, then $H$ is an almost-defining Hamiltonian.
\end{enumerate}
\end{remark}

\begin{example}
\label{example: almost-defining Hamiltonian}
The projection $h : S \Sigma = \R \times \Sigma \rightarrow \R$ is an almost-defining Hamiltonian. If $\beta : \R \rightarrow [-1,1]$ is a monotone increasing function with $\beta(\pm t) = \pm 1 $ for $t \geq 1$ and $\beta(t) = t$ for $t$ around zero, then $\beta \circ h$ is a defining Hamiltonian.
\end{example}

\noindent Next we define the notion of a $C^1$-datum. For this recall from Definition \ref{def: C^1-bounded subset} that a $C^1$-bounded subset is just a subset of the space of smooth functions that is bounded with respect to any $C^1$-norm.

\begin{definition}[$C^1$-datum]
\label{def: C^1 datum} \hfill
\begin{enumerate}[label=(\alph*)]
    \item\label{it: C^1-datum} A \textit{$C^1$-datum} is a tuple $\Dscr = (U_\Sigma , \, c , \, \BCOne)$ consisting of
    \begin{enumerate}[label=\arabic*)]
        \item an open neighborhood $U_\Sigma \subseteq S \Sigma$ of $\Sigma$ whose closure $\overline{U_\Sigma}$ is a compact smooth manifold with boundary,
        \item a positive constant $c > 0$ and
        \item a $C^1$-bounded subset $\BCOne \subseteq \Cinfty(\overline{U_\Sigma}) \,$.
    \end{enumerate}
    \item\label{it: complying with C^1-datum} A function $H \in \Cinfty(S\Sigma)$ \textit{complies with the $C^1$-datum} $\Dscr =(U_\Sigma , \, c , \, \BCOne)$ if its restriction to $\overline{U_\Sigma}$ lies in $\BCOne \,$, that is $H|_{\overline{U_\Sigma}} \in \BCOne \,$, and 
    \[
    \d H (\partial_h) = \lambda(X_H) \geq c > 0 \quad \text{ on } U_\Sigma \,\, .
    \]
\end{enumerate}
\end{definition}

\noindent Intuitively, one should think of a small neighborhood $U_\Sigma =(-\eps,\eps) \times \Sigma \,$, a small constant $c > 0$ and of $\BCOne$ being large but bounded. A Hamiltonian $H$ complies with a $C^1$-datum if its derivative close to $\Sigma$ is bounded by a large constant and its partial derivative in $\R$-direction is bounded below by some small positive constant.

\begin{remark}
\label{rmk: C^1-datum partial order}
We can define a partial order on the set of $C^1$-data by declaring for $C^1$-data $\Dscr = (U_\Sigma, c , \BCOne)$ and $\Dscr' = (U_\Sigma', c', \BCOne')$ that
\begin{align*}
   \Dscr \preceq \Dscr' \,\,\, :\Longleftrightarrow \,\,\, U_\Sigma \supseteq U'_{\Sigma} \comma c \geq c' \text{ and } \, \forall f \in \BCOne : \, f|_{\overline{U'_\Sigma}} \in \BCOne' \,\, .
\end{align*}
If $\mathscr{D} \preceq \mathscr{D}'$ and $H$ complies with $\mathscr{D} \,$, then $H$ also complies with $\mathscr{D}' \,$. 
Notice moreover that, $\Sigma$ being compact, the subset of $C^1$-data of the form
\begin{equation*}
\Bigset{\Dscr = \big(  (-\eps,\eps) \times \Sigma ,\,\eps ,\, \BCOne \big) }{ \eps  > 0 \comma \BCOne \subseteq \Cinfty([-\eps,\eps] \times \Sigma) \text{ } C^1\text{-bounded} }\, ,
\end{equation*}
is cofinal in the set of $C^1$-data. 
\end{remark}

\subsection{The Rabinowitz action functional}
\label{subsec: Rabinowitz AF}

Denote by $\loopspace := \Cinfty(\bbS^1,S\Sigma)$ the free loop space of $S \Sigma \,$. The Rabinowitz action functional, rediscovered by Cieliebak-Frauenfelder \cite{Cieliebak_Frauenfelder}, for an autonomous Hamiltonian $H : S \Sigma \rightarrow \R$ is the functional
\begin{equation}
\label{eq: Rabinowitz action functional def}
\RabinowitzH : \, \loopspace \times \R \rightarrow \R \comma \quad \RabinowitzH(v,\tau) := \int_{\bbS^1} v^* \lambda - \tau \int_0^1 H(v(t)) \, \d t \,\, .
\end{equation}
An $\omega$-compatible almost complex structure $J$ on $S \Sigma$ induces an $L^2$-metric $g_J$ on $\loopspace \times \R$ via
\[
(g_J)_{(v,\tau)} \big( (\Hat{v}_1, \Hat{\tau}_1) ,  (\Hat{v}_2, \Hat{\tau}_2) \big) = \int_0^1 \omega_v (\Hat{v}_1 , J(v) \, \Hat{v}_2) \, \d t + \Hat{\tau}_1 \cdot \Hat{\tau}_2 \,\, 
\]
where $(\Hat{v}_j, \Hat{\tau}_j) \in T_{(v,\tau)} (\loopspace \times \R) = \Gamma(v^*T(S \Sigma)) \oplus \R   \,$, \ie $\Hat{v}_j$ are vector fields along the loop $v$.
The gradient of $\RabinowitzH$ with respect to this metric is given by
\begin{equation}
\label{eq: gradient(Rabinowitz)}
 \nabla^{g_J} \RabinowitzH (v,\tau) = { {- J(v) (\partial_t v - \tau X_H(v))} \choose {-\int_0^1  H(v) \, \d t} } \,\, ,
\end{equation}
so that the critical points are
\begin{equation}
 \label{eq: Crit(Rabinowitz)}
 (v,\tau) \in \Crit(\RabinowitzH) \,\, \Longleftrightarrow \,\, \begin{cases}
     \partial_t v = \tau X_H (v) \\
     v \subseteq H^{-1}(0) \,\, ,
 \end{cases}
\end{equation}
(using that $H$ is constant along its Hamiltonian flow lines) and the $(\nabla^{g_J} \RabinowitzH)$-flow lines $w = (u,\eta) : \R \rightarrow \loopspace \times \R$ are the solutions of the \textit{Rabinowitz-Floer equation}
\begin{equation}
\label{eq: Rabinowitz-Floer eq}
\begin{cases}
  0 =  \partial_s u + J(u) (\partial_t u - \eta \, X_H(u) ) \\
0 = \partial_s \eta (s) + \int_0^1 H(u(s,t)) \, \d t  \,\, ,
\end{cases}
\end{equation}
where we regard $u : \R \rightarrow \loopspace$ as cylinder $u \in \Cinfty(\R \times \bbS^1,S \Sigma)\,$.

\begin{remark}
\label{rmk: Rabinowitz action value defining Hamiltonian}
We record some immediate consequences of \eqref{eq: Crit(Rabinowitz)}.
\begin{enumerate}[label=(\roman*)]
    \item\label{it: Rabinowitz action functional - action value} The action value of a critical point $(v,\tau) \in \Crit(\RabinowitzH)$ is 
    \[
    \RabinowitzH(v,\tau) = \tau \int_0^1 \lambda_v(X_H(v)) \, \d t \,\, .
    \] 
    In particular, $\RabinowitzH(v,\tau) = \tau$ if $H$ is an almost-defining Hamiltonian.
    \item\label{it: Rabinowitz action functional - constant loops} $(v,0) \in \loopspace \times \R$ is a critical point iff $v$ is a constant loop with $v \subseteq H^{-1}(0) \,$.
    \item\label{it: Rabinowitz action functional - reparametrized Hamiltonian flow lines} For $\tau \not= 0 \,$, the critical points $(v,\tau) \in \Crit(\RabinowitzH)$ are in one-to-one correspondence with $\tau$-periodic $H$-Hamiltonian flow lines contained in $H^{-1}(0) \,$.
    \item\label{it: Rabinowitz action functional - spectrum Zoll case} Let $(\Sigma,\alpha)$ be Zoll contact with minimal Reeb period $\tau(\alpha) > 0$ and $H$ an almost-defining Hamiltonian. By \ref{it: Rabinowitz action functional - action value} and \ref{it: Rabinowitz action functional - reparametrized Hamiltonian flow lines}, $\RabinowitzH$ has \textit{action spectrum}
    \[
     \RabinowitzH(\Crit(\RabinowitzH)) = \tau(\alpha) \Z \,\, .
    \]
    \item\label{it: Rabinowitz action functional - critical point tower} A critical point $(v,\tau) \in \Crit(\RabinowitzH) \,$, with $v$ being an embedded loop, gives rise to the \textit{critical point tower}
    \[
    \bigsqcup_{k \in \Z} \set{(v(t + k \, \cdot\,) \comma k \tau)}{t \in \bbS^1 = \R/\Z} \subseteq \Crit(\RabinowitzH) \,\, ,
    \]
    see \cite[Rmk.~5.1.9]{Weber}. Two critical points $(v_1,\tau_1) , \, (v_2,\tau_2) \,$, for which $\tau_j X_H(v_j(0)) \not= 0 \comma j =1,2 \,$, belong to the same critical point tower iff their corresponding Hamiltonian flow lines $v_j(\frac{\cdot}{\tau_j})$ are geometrically equal.
\end{enumerate}
\end{remark}

\noindent Define the \textit{energy} $E_J(w) \in [0,\infty]$ of a $(\nabla^{g_J} \RabinowitzH)$-flow line $w  \in \Cinfty(\R,\loopspace \times \R)$ by
\begin{align*}
E_J(w) := \int_{-\infty}^{\infty} |\partial_s w(s)|_{g_J}^2 \, \d s = \int_{-\infty}^{\infty} \frac{\d}{\d s} \RabinowitzH(w(s)) \, \d s = \lim_{s \rightarrow \infty} \Big( \RabinowitzH(w(s)) - \RabinowitzH(w(-s)) \Big) \, .
\end{align*}
We will work with almost complex structures on $S \Sigma$ of SFT-type (\textit{SFT} stands for \textit{Symplectic Field Theory}):

\begin{definition}[SFT-type]
\label{def: SFT-type almost complex structure}
An almost complex structure $J$ on $S \Sigma = \R \times \Sigma$ is of \textit{SFT-type} if the following conditions hold:
\begin{enumerate}[label=\arabic*)]
    \item $J$ is invariant under the translation $(h,z) \mapsto (h + h_0, z) \,$, for every $h_0 \in \R \,$.
    \item $\xi = \ker(\alpha)$ is $J$-invariant and $J|_\xi$ is $\d \alpha|_\xi$-compatible.
    \item $J \partial_h = R_\alpha \,$.
\end{enumerate}
\end{definition}

\begin{remark}
\label{rmk: SFT-type almost complex structure}
An almost complex structure of SFT-type is $(\omega = \d(e^h \alpha)$)-compatible. Moreover, SFT-type almost complex structures on $S \Sigma$ are in bijection with compatible complex structures on the symplectic vector bundle $(\xi, \d \alpha|_\xi ) $ with base $\Sigma \,$. In particular, there exist SFT-type almost complex structures.
\end{remark}

\subsection{A theorem about gradient flow lines for perturbed Hamiltonians}
\label{subsec: Gradient Flow Lines for Perturbed Hamiltonians}

Let $(\Sigma,\alpha)$ be a Zoll contact manifold whose Reeb flow has minimal period $\tau(\alpha) > 0 \,$.
For every constant $h_0 > 0$ we define the class of functions
\[
\Hcal^{h_0} := \set{H \in \Cinfty(S \Sigma)}{ \d H = 0 \,\text{ on } S \Sigma \backslash ( (-h_0,h_0) \times \Sigma )} \,\, .
\]
Modifying Example \ref{example: almost-defining Hamiltonian} one sees that $\Hcal^{h_0}$ contains defining Hamiltonians for every $h_0 >0 \,$. For positive constants $h_0 , \delta > 0 \,$, a $C^1$-datum $\Dscr$ (see Definition \ref{def: C^1 datum}) and a defining Hamiltonian $H_0 \in \Hcal^{h_0}$ we consider the class of admissible perturbations of $H_0$, defined by
\begin{equation}
\label{eq: def Hclassparam}
\Hclassparam := \set{H \in \Hcal^{h_0}}{H \text{ complies with } \Dscr \comma \, \lVert H - H_0 \rVert_{\infty} \leq \delta} \,\, ,
\end{equation}
with sup-norm
\[
\lVert H - H_0 \rVert_\infty := \sup_{S\Sigma} |H - H_0| = \max_{[-h_0,h_0] \times \Sigma} |H - H_0| < \infty \,\, .
\]
With the partial order on the set of $C^1$-data introduced in Remark \ref{rmk: C^1-datum partial order}, it obviously holds
\begin{align}
\label{eq: Hclassparam inclusion for varying parameters}
   h_0 \leq h_0' \comma \Dscr \preceq \Dscr' \text{ and } \delta \leq \delta' \,\,\, \Longrightarrow \,\,\, \Hclassparam \subseteq \Hcal^{h_0'}(H_0, \, \Dscr' , \, \delta') \,\, . 
\end{align}
The next theorem lies at the heart of this work.

\begin{MainThm}
\label{mthm: Rabinowitz gradient flow line}
Suppose $(\Sigma,\alpha)$ is a Zoll contact manifold with minimal Reeb period $\tau(\alpha) >0 \,$. Let $J$ be an almost complex structure of SFT-type on the symplectization $S \Sigma $ and let $H_0 \in \Hcal^{h_0}$ be a defining Hamiltonian. 
For every $\eps > 0 $ and every $C^1$-datum $\Dscr$ there exists a constant $\delta = \delta(J,H_0, h_0,\eps,\Dscr) > 0$ with the following significance: 
\\
For every $H \in \Hclassparam$ and every $z \in \Sigma$ there exist a $(\nabla^{g_J} \RabinowitzH)$-flow line $w = (u,\eta) \in \Cinfty(\R,\loopspace\times\R)$, a sequence $(s_n)_{n \in \N}$ of positive numbers tending to infinity and critical points $(v^-,\tau^-), \, (v^+,\tau^+) \in \Crit(\RabinowitzH)$ with
\begin{enumerate}[label=\arabic*), itemsep=0.75ex]
     \item\label{it: mthm Rabinowitz gradient flow line - tau not= 0} $\tau^-, \tau^+ \not= 0 \,$,
    \item\label{it: mthm Rabinowitz gradient flow line - convergence} $w(\pm s_n) \overset{n \rightarrow \infty}{\longrightarrow} (v^\pm, \tau^\pm)$ in $\Cinfty(\bbS^1,S\Sigma) \times \R \,$,
    \item\label{it: mthm Rabinowitz gradient flow line - action estimate} $\RabinowitzH(w(s)) , \, \RabinowitzHZero(w(s)) \in [\tau(\alpha) - \eps, \, \tau(\alpha) + \eps]$ for every $s \in \R \,$,
    \item\label{it: mthm Rabinowitz gradient flow line - u(0,0)} $u(0,0) \in \R \times \{z\} \subseteq S \Sigma \,$.
\end{enumerate}
\end{MainThm}
\noindent In point \ref{it: mthm Rabinowitz gradient flow line - u(0,0)} we evaluate the cylinder $u \in \Cinfty(\R \times \bbS^1,S \Sigma)$ at $(0,0) \in \R \times \bbS^1 \,$.

\begin{remark}
\label{rmk: Rabinowitz gradient flow line}
For a $(\nabla^{g_J} \RabinowitzH)$-flow line $w=(u,\eta)$ as in the theorem, observe that
\[
0 \leq E_J(w) = \lim_{n \rightarrow \infty} \Big( \RabinowitzH(w(s_n)) - \RabinowitzH(w(-s_n)) \Big) = \RabinowitzH(v^+,\tau^+) - \RabinowitzH(v^-,\tau^-) \,\, .
\]
If $E_J(w) > 0 \,$, then $\RabinowitzH(v^-,\tau^-) < \RabinowitzH(v^+,\tau^+) \,$. Otherwise $E_J(w) = 0 \,$, so that $w(s) = (v^-,\tau^-) = (v^+,\tau^+)$ for every $s \in \R \,$. Consequently, by property \ref{it: mthm Rabinowitz gradient flow line - u(0,0)}, the critical point $(v^+,\tau^+)$ satisfies $v^+(0) \in \R \times \{z\} \,$.
\end{remark}

\noindent We will prove Theorem \ref{mthm: Rabinowitz gradient flow line} in Sections \ref{sec: Proof MThm - Preparation}\,-\,\ref{sec: Non-Empty Moduli Spaces}. The rest of the present section is devoted to deriving Theorems \ref{mthm: rescaling Zoll contact form} and \ref{mthm: multiplicity result} from it.

\subsection{Cutting off away from $\Sigma$}
\label{subsec: Cutting off away from Sigma}

To overcome the additional assumption in Theorem \ref{mthm: Rabinowitz gradient flow line} that both the almost-defining Hamiltonian $H_0$ and the perturbed Hamiltonian $H$ have compactly supported differential, we cut them off uniformly.

\begin{proposition}
\label{prop: cut-off}
Given an almost-defining Hamiltonian $H_0 \in \Cinfty(S\Sigma)\,$, a constant $h_0 > 0$ and a $C^1$-datum
\begin{equation}
\label{eq: C^1 datum U_Sigma = product}
\Dscr = \big(  (-h_1, h_1) \times \Sigma \comma c \comma \BCOne \big) \quad \text{ with } h_1 \in (0,h_0) \,\, .
\end{equation}
Then there exist a defining Hamiltonian $\hat{H}_0 \in \Hcal^{h_0}$ and for every $\delta > 0$ a map
\begin{align*}
 \mathsf{Cut}_{\delta} : \,\, \Bigset{H \in \Cinfty(S\Sigma)}{ \begin{array}{c}
    |H - H_0| \leq \delta \text{ on } [-h_0,h_0] \times \Sigma \\
      \text{ and } H \text{ complies with } \Dscr 
 \end{array}} \longrightarrow \Hcal^{h_0}(\hat{H}_0,  \Dscr,  \delta)
\end{align*}
with the following property:\\
For every $H$ in the domain of $\mathsf{Cut}_\delta \,$, setting $\hat{H} := \mathsf{Cut}_\delta(H) \,$, it holds
\begin{enumerate}[label=\arabic*)]
    \item\label{it: cut-off - Hhat = H} $\hat{H}_0 = H_0$ and $\hat{H} = H$ on $[-h_1,h_1] \times \Sigma \,$,
    \item\label{it: cut-off - equal to 1 near infty} $\hat{H}_0 (h,z) = \hat{H} (h,z) = \pm 1$ for $\pm h \geq h_0 \comma z \in \Sigma \,$,
    \item\label{it: cut-off - Hhat - Hhat_0 < H - H_0} $|\hat{H} - \hat{H}_0| \leq |H - H_0|$ on $S \Sigma \,$.
\end{enumerate}
For $\delta$ sufficiently small (depending on $H_0 , \, h_0, \, h_1$), it moreover holds
\begin{enumerate}[label=\arabic*),resume]
    \item\label{it: cut-off - zero set Hhat} $\hat{H}^{-1}(0) \subseteq  (-h_1,h_1) \times \Sigma \,$.
\end{enumerate}
\end{proposition}

\begin{proof}
    Fix a monotone function $\beta \in \Cinfty(\R_{\geq 0}, \, [0,1])$ with $\beta(h) = 0$ for $h \leq h_1$ and $\beta(h) = 1$ for $h \geq h_0 \,$. Moreover fix a monotone function $\rho \in \Cinfty(\R,[-1,1])$ with $\rho(h) =1$ for $h \geq h_0 \,$, with $\rho(h) > 0$ for $h > 0$ and $\rho(-h) = - \rho(h) \,$. Define
    \[
    \hat{H}_0(h,z) := (1- \beta(|h|)) \, H_0(h,z) + \beta(|h|) \, \rho(h) \quad \text{ for } (h,z) \in \R \times \Sigma = S \Sigma  
    \]
    and observe that $\hat{H}_0^{-1}(0) = \Sigma$ since both $H_0(h,z)$ and $\rho(h)$ are positive for $h > 0 $ respectively negative for $h < 0 \,$, see also Figure \ref{fig: Cut delta function}. Similarly, we set
    \begin{equation}
     \label{eq: cutting of away from Sigma - def hatH}   
      \hat{H}(h,z) := \mathsf{Cut}_\delta(H)(h,z) := (1- \beta(|h|)) \, H(h,z) + \beta(|h|) \, \rho(h) \,\, ,
    \end{equation}
    which is independent of $\delta$ (though the domain of $\mathsf{Cut}_\delta$ depends on $\delta$).
    
    Properties \ref{it: cut-off - Hhat = H}\,-\,\ref{it: cut-off - Hhat - Hhat_0 < H - H_0} are readily verified, which also shows that $\hat{H}_0$ is a defining Hamiltonian and that $\hat{H}$ indeed lies in $\Hcal^{h_0}(\hat{H}_0,  \Dscr,  \delta )\,$. 
    
    Lastly, we show property \ref{it: cut-off - zero set Hhat}. Choose a constant $\delta_0  > 0$ so that $|H_0(h,z)| \geq \delta_0$ for every $(h,z) \in \R \times \Sigma$ with $h_1 \leq |h| \leq h_0 \,$. Then, for every $\delta < \delta_0 \,$, any Hamiltonian $H$ with $|H - H_0| \leq \delta$ on $[-h_0,h_0] \times \Sigma$ is strictly positive on $[h_1,h_0] \times \Sigma$ and strictly negative on $[-h_0,-h_1] \times \Sigma$ and this carries over to $\hat{H}$ in view of \eqref{eq: cutting of away from Sigma - def hatH}. Hence property \ref{it: cut-off - zero set Hhat} holds for every $\delta < \delta_0 \,$.
\end{proof}

\begin{figure}
    \centering
    \includegraphics[width=0.7\linewidth]{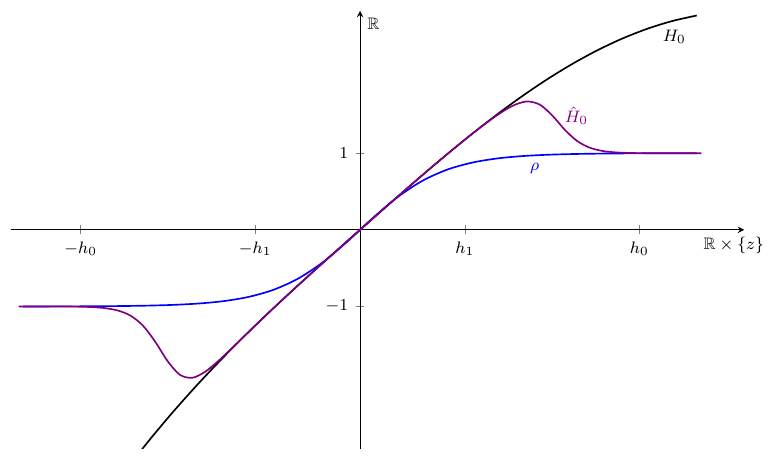}
    \caption{Cross section of the graph of $\hat{H}_0$ on the slice $\R \times \{z\} \subseteq \R \times \Sigma \comma z \in \Sigma \,$.}
    \label{fig: Cut delta function}
\end{figure}

\noindent The cutoff allows us to infer the following corollary from Theorem \ref{mthm: Rabinowitz gradient flow line}.

\begin{corollary}
\label{cor: Rabinowitz gradient flow line}
Suppose $(\Sigma,\alpha)$ is a Zoll contact manifold with minimal Reeb period $\tau(\alpha) \,$. Let $H_0$ be an almost-defining Hamiltonian. Then for every $h_0 >0 \,$, every $\eps > 0$ and every $C^1$-datum $\Dscr = (U_\Sigma,\, c, \, \BCOne)$ there exists a constant $\delta = \delta(H_0,h_0, \eps,\Dscr) > 0$ with the following significance:\\
For every $H \in \Cinfty(S \Sigma)$ satisfying
\begin{enumerate}[label=\arabic*)]
    \item\label{it: cor Rabinowitz crit pts - complying} $H$ complies with $\Dscr \,$,
    \item\label{it: cor Rabinowitz crit pts - H-H_0} $|H - H_0| \leq \delta$ on $[-h_0,h_0] \times \Sigma \,$,
\end{enumerate}
one of the following cases holds:
\begin{enumerate}[label=(\alph*)]
    \item\label{it: cor Rabinowitz crit pts - two crit pts} There exist critical points $(v^-,\tau^-) , \, (v^+,\tau^+) \in \Crit(\RabinowitzH)$ with $\tau^\pm \not= 0 \comma v^\pm \subseteq U_\Sigma \,,$ and
    \[
    \RabinowitzH(v^\pm,\tau^\pm) ,\, \RabinowitzHZero(v^\pm,\tau^\pm) \in [\tau(\alpha) - \eps, \, \tau(\alpha) + \eps] 
    \]
    as well as $ \RabinowitzH(v^-,\tau^-) < \RabinowitzH(v^+,\tau^+) \,$.
    \item\label{it: cor Rabinowitz crit pts - one crit pt} For every $z \in \Sigma$ there exists a critical point $(v,\tau) \in \Crit(\RabinowitzH)$ with $\tau \not= 0 \comma v \subseteq U_\Sigma$ and
    \[
    \RabinowitzH(v,\tau) , \, \RabinowitzHZero(v,\tau) \in [\tau(\alpha) - \eps, \, \tau(\alpha) + \eps] 
    \]
    as well as $v(0) \in \R \times \{z\} \,$.
\end{enumerate}
\end{corollary}

\begin{proof}
    It suffices to show the corollary for a $C^1$-datum $\Dscr$ of the form in \eqref{eq: C^1 datum U_Sigma = product} because such $C^1$-data are cofinal in the set of $C^1$-data (Remark \ref{rmk: C^1-datum partial order}). So assume $U_\Sigma = (-h_1,h_1) \times \Sigma$ with $h_1 < h_0 \,$. Fix a defining Hamiltonian $\hat{H}_0$ and a family of cutoff functions $(\mathsf{Cut}_\delta)_{\delta > 0}$ as in the preceding Proposition \ref{prop: cut-off}. Additionally, fix an almost complex structure $J$ on $S\Sigma$ of SFT-type. Now choose a positive constant $\delta = \delta(J, \hat{H}_0, h_0, \eps, \Dscr)$ as in Theorem \ref{mthm: Rabinowitz gradient flow line}. 
    Choosing $\delta$ smaller if necessary (preserving the property stated in Theorem \ref{mthm: Rabinowitz gradient flow line}), we can additionally assume that property \ref{it: cut-off - zero set Hhat} in Proposition \ref{prop: cut-off} holds for $\delta \,$. Let us verify that this $\delta$ has the desired property.

    Given $H \in \Cinfty(S \Sigma)$ which complies with $\Dscr$ and with $|H - H_0| \leq \delta$ on $[-h_0,h_0] \times \Sigma \,$. Set $\hat{H} := \mathsf{Cut}_\delta(H) \in \Hcal^{h_0}(\hat{H}_0, \Dscr, \delta) \,$. Suppose case \ref{it: cor Rabinowitz crit pts - two crit pts} does not hold. We must show case \ref{it: cor Rabinowitz crit pts - one crit pt}. To this end, let $z \in\Sigma$ be given. Because $\delta$ has the significance stated in Theorem \ref{mthm: Rabinowitz gradient flow line}, we find a $(\nabla^{g_J} \Rabinowitz{\hat{H}})$-flow line $w \,$, a sequence $(s_n)_n$ tending to infinity and critical points $(v^\pm,\tau^\pm) \in \Crit(\Rabinowitz{\hat{H}})$ satisfying \ref{it: mthm Rabinowitz gradient flow line - tau not= 0}\,-\,\ref{it: mthm Rabinowitz gradient flow line - u(0,0)} in Theorem \ref{mthm: Rabinowitz gradient flow line}, with $H_0$ replaced by $\hat{H}_0$ and $H$ replaced by $\hat{H} \,$. Observe that $v^\pm \subseteq \hat{H}^{-1}(0) \subseteq U_\Sigma$ by the critical point equation \eqref{eq: Crit(Rabinowitz)} and Proposition \ref{prop: cut-off} \ref{it: cut-off - zero set Hhat}. Since $\hat{H}_0$ and $H_0$ respectively $\hat{H} $ and $H$ agree on $U_\Sigma$ by Proposition \ref{prop: cut-off} \ref{it: cut-off - Hhat = H}, the points $(v^\pm,\tau^\pm)$ are also critical for $\RabinowitzH$ and their action values agree, that is
    \[
    \Rabinowitz{\hat{H}}(v^\pm,\tau^\pm) = \RabinowitzH(v^\pm,\tau^\pm) \quad \text{ and } \quad \Rabinowitz{\hat{H}_0}(v^\pm,\tau^\pm)  = \Rabinowitz{H_0}(v^\pm,\tau^\pm) \,\, .
    \]
    Taking Remark \ref{rmk: Rabinowitz gradient flow line} into account, the gradient flow line $w$ must in fact have energy zero, otherwise we would obtain different action values, \ie $\Rabinowitz{\hat{H}}(v^-,\tau^-) < \Rabinowitz{\hat{H}}(v^+,\tau^+) \,$, contradicting our assumption that case \ref{it: cor Rabinowitz crit pts - two crit pts} does not hold. Hence $w$ is $s$-independent and $(v,\tau) := w(0) = (v^\pm,\tau^\pm)$ is a critical point of $\RabinowitzH$ with $v(0) \in \R \times \{z\} \,$. As $z$ was arbitrary, this shows case \ref{it: cor Rabinowitz crit pts - one crit pt}.
\end{proof}

\subsection{Proof of Theorem \ref{mthm: rescaling Zoll contact form}}
\label{subsec: Proof mthm rescaling Zoll contact form}

Corollary \ref{cor: Rabinowitz gradient flow line} is all we need to prove Theorem \ref{mthm: rescaling Zoll contact form}.

Let $(\Sigma,\alpha)$ be a Zoll contact manifold with minimal Reeb period $\tau(\alpha) > 0 \,$. To every function $f \in \Cinfty(\Sigma)$ we associate the smooth Hamiltonian
\[
H_f : S \Sigma \rightarrow \R \comma \,\, H_f(h,x) := h - f(x) \text{ for } (h,x) \in \R \times \Sigma \, .
\]
As discussed in Example \ref{example: Level Set Hamiltonian H = h - f(x)}, on $H_f^{-1}(0) = \mathrm{Gr}(f)$ the Hamiltonian vector field $X_{H_f}$ agrees with the Reeb vector field of $\lambda|_{\mathrm{Gr}(f)} = (e^h \alpha)|_{\mathrm{Gr}(f)} \,$. The $H_f\,$-Hamiltonian flow lines on $H_f^{-1}(0)$ are conjugated via the projection $\pi_\Sigma : S \Sigma \rightarrow \Sigma$ to Reeb orbits of $(\Sigma,e^f \alpha) \,$, see Lemma \ref{lem: Reeb orbits Gr(f) <--> e^f alpha Reeb orbits}. The Hamiltonian $H_0$ associated to the zero function $0 \in \Cinfty(\Sigma)$ is just the projection to the $\R$-factor and is an almost-defining Hamiltonian.

\begin{lemma}
\label{lem: C^1-bounded subset Cinfty(Sigma) --> C^1-bounded subset Cinfty(S Sigma)}
For every $C^1$-bounded subset $\mathscr{B}_\Sigma \subseteq \Cinfty(\Sigma)$ there exists a $C^1$-bounded subset $\mathscr{B}_{S \Sigma} \subseteq \Cinfty([-1,1] \times \Sigma)$ with $H_f|_{[-1,1] \times \Sigma} \in \mathscr{B}_{S \Sigma}$ for every $f \in \mathscr{B}_\Sigma\,$.
\end{lemma}

\begin{proof}
    One clearly can bound $\d H_f = \d h - \d f$ in terms of $\d f \,$.
\end{proof}

\noindent Let us now prove Theorem \ref{mthm: rescaling Zoll contact form} and its minor upgrade regarding the period, see Remark \ref{rmk: mthm rescaling Zoll contact form} \ref{it: mthm rescaling Zoll contact form - period}.

\begin{proof}[Proof of Theorem \ref{mthm: rescaling Zoll contact form}]
Suppose we are given a $C^1$-bounded subset $\mathscr{B}_{\Sigma} \subseteq \Cinfty(\Sigma)$ and $\eps > 0 \,$. We must find a constant $\delta >0$ so that for every subset $\varnothing \not= D \subseteq \Sigma$ with tame boundary and every $f \in \mathscr{B}_\Sigma$ with $\lVert f \rVert_\infty \leq \delta$ vanishing on $\Sigma \backslash D$ there exists a periodic Reeb orbit of $(\Sigma, e^f \alpha)$ intersecting $D$ of period $\tau \in [\tau(\alpha) - \eps, \tau(\alpha)+\eps] \,$. It obviously suffices to show this for sufficiently small $\eps \,$, so without loss of generality assume $\eps \in (0, \tau(\alpha)) \,$.

Because $H_0$ is an almost-defining Hamiltonian and $\eps < \tau(\alpha) \,$, it follows from Remark \ref{rmk: Rabinowitz action value defining Hamiltonian} \ref{it: Rabinowitz action functional - action value} and \ref{it: Rabinowitz action functional - spectrum Zoll case} that
 \begin{align}
    \label{eq: proof rescaling Zoll contact form - specgap}
    \Forall (v,\tau) \in \Crit(\RabinowitzHZero): \,\, | \RabinowitzHZero(v,\tau) - \tau(\alpha) | \leq \eps  \,\, \Longrightarrow \,\, \RabinowitzHZero(v,\tau) = \tau = \tau(\alpha) \, .
    \end{align}
For the given $C^1$-bounded subset $\mathscr{B}_\Sigma$ choose a $C^1$-bounded subset $\mathscr{B}_{S \Sigma} \subseteq \Cinfty([-1,1] \times \Sigma)$ as in Lemma \ref{lem: C^1-bounded subset Cinfty(Sigma) --> C^1-bounded subset Cinfty(S Sigma)} above. Now define the $C^1$-datum
\[
\mathscr{D} := ( U_\Sigma := (-1,1) \times \Sigma, \, c := 1 , \, \mathscr{B}_{S \Sigma})
\]
and choose a constant $\delta = \delta(H_0 , h_0 := 1 , \eps, \mathscr{D}) > 0$ as in Corollary \ref{cor: Rabinowitz gradient flow line}. We will verify that this $\delta$ has the claimed significance.

To this end, suppose we are given $\varnothing \not= D \subseteq \Sigma$ with tame boundary as well as $f \in \mathscr{B}_\Sigma$ with $\lVert f \rVert_\infty \leq \delta$ and vanishing on $\Sigma \backslash D \,$.\footnote{
The reader should check that the subsequent argument works in particular for $D = \Sigma \,$.
} 
Then $H_f$ complies with $\mathscr{D}$ due to our choice of $\mathscr{B}_{S\Sigma}$ and since $\d H_f (\partial_h) \equiv 1 \,$. Observe moreover that $H_0 - H_f = f \circ \pi_\Sigma \,$, so that $\sup_{S \Sigma} |H_f - H_0| = \lVert f \rVert_\infty \leq \delta \,$. Having chosen $\delta$ with the significance stated in Corollary \ref{cor: Rabinowitz gradient flow line}, we conclude that case \ref{it: cor Rabinowitz crit pts - two crit pts} or case \ref{it: cor Rabinowitz crit pts - one crit pt} therein must hold with $H = H_f$ and $\eps\,$. 
\begin{claimI}
\label{claim: proof mthm rescaling Zoll - functionals equal}
    For $(v,\tau) \in \loopspace \times \R$ with $v \subseteq \R \times (\Sigma \backslash D)$ it holds
    \[
    \Rabinowitz{H_f}(v,\tau) = \RabinowitzHZero(v,\tau) \,\, \text{ and } \,\, \d \Rabinowitz{H_f} = \d \RabinowitzHZero(v,\tau) \,\, .
    \]
\end{claimI}
\begin{proofClaim}
    Since $f$ vanishes on $\Sigma \backslash D$, clearly $H_f = H_0$ on $\R \times (\Sigma \backslash D)  \,$. In particular $H_f$ and $H_0$ agree along the loop $v$, so, in view of the definition of the Rabinowitz action functional in \eqref{eq: Rabinowitz action functional def}, we have $\Rabinowitz{H_f}(v,\tau) = \RabinowitzHZero(v,\tau) \,$. Again because $f$ vanishes on $\Sigma \backslash D \,$, its differential vanishes on $\mathrm{int}(\Sigma \backslash D)$ and by continuity also on the closure $\overline{\mathrm{int}(\Sigma \backslash D)} \,$. The latter contains $\Sigma \backslash D$ since $D$ has tame boundary, so $\d f$ vanishes on $\Sigma \backslash D \,$. Consequently $\d H_f = \d H_0$ on $\R \times (\Sigma \backslash D) \,$, so that $X_{H_f}(v) = X_{H_0}(v) \,$. In view of formula \eqref{eq: gradient(Rabinowitz)} for the gradient, we conclude that $\d \Rabinowitz{H_f}(v,\tau) = \d \RabinowitzHZero(v,\tau) \,$.
\end{proofClaim}
\noindent We now assume by contradiction that there does not exist a periodic Reeb orbit of $(\Sigma,e^f \alpha)$ intersecting $D$ of period $\tau \in [\tau(\alpha)-\eps,\tau(\alpha)+ \eps] \,$.
\begin{claimI}
\label{claim: proof mthm rescaling Zoll - critical sets equal}
Under this assumption, for every critical point $(v,\tau) \in \Crit(\Rabinowitz{H_f})$ with $|\Rabinowitz{H_f}(v,\tau) - \tau(\alpha)| \leq \eps$ it holds $v \subseteq \R \times (\Sigma \backslash D)$ and $(v,\tau) \in \Crit(\RabinowitzHZero)$ as well as 
\[
\tau = \Rabinowitz{H_f}(v,\tau) = \RabinowitzHZero(v,\tau) = \tau(\alpha) \, .
\]
\end{claimI}
\begin{proofClaim}
    Given $(v,\tau) \in \Crit(\Rabinowitz{H_f})$ with $|\Rabinowitz{H_f}(v,\tau)-\tau(\alpha)| \leq \eps \,$. Notice that $\Rabinowitz{H_f}(v,\tau) = \tau$ by Remark \ref{rmk: Rabinowitz action value defining Hamiltonian} \ref{it: Rabinowitz action functional - action value} and since $ \lambda(X_{H_f}) = \d H_f(\partial_h) = 1 \,$. In particular $\tau = \Rabinowitz{H_f}(v,\tau) \geq \tau(\alpha) - \eps > 0 \,$. 
    
    We contend that the loop $v$ is contained in $\R \times (\Sigma \backslash D)$: Otherwise, by the critical point equation \eqref{eq: Crit(Rabinowitz)}, the reparametrized loop $v(\frac{\cdot}{\tau}) \subseteq H_f^{-1}(0)$ would be $\tau$-periodic $H_f\,$-Hamiltonian flow line intersecting $\R \times D$ and its projection $\pi_{\Sigma} \circ v(\frac{\cdot}{\tau})$ would be a $\tau$-periodic Reeb orbit of $(\Sigma,e^f \alpha)$ intersecting $D$ with $|\tau - \tau(\alpha)|\leq \eps \,$, contradicting our assumption that such do not exist. This shows $v \subseteq \R \times (\Sigma \backslash D) \,$.
    
    By Claim \ref{claim: proof mthm rescaling Zoll - functionals equal}, thus $(v,\tau)$ is also a critical point of $\RabinowitzHZero$ and moreover $\Rabinowitz{H_f}(v,\tau) = \RabinowitzHZero(v,\tau) \,$. Lastly, it follows from \eqref{eq: proof rescaling Zoll contact form - specgap} that $\RabinowitzHZero(v,\tau) = \tau(\alpha) \,$.
\end{proofClaim}
\noindent Claim \ref{claim: proof mthm rescaling Zoll - critical sets equal} shows that every critical point $(v,\tau) \in \Crit(\Rabinowitz{H_f})$ with $\Rabinowitz{H_f}(v,\tau) \in [\tau(\alpha)-\eps,\tau(\alpha)+\eps]$ has the same action value $\Rabinowitz{H_f}(v,\tau) =\tau(\alpha) \,$. This rules out case \ref{it: cor Rabinowitz crit pts - two crit pts} in Corollary \ref{cor: Rabinowitz gradient flow line}, so case \ref{it: cor Rabinowitz crit pts - one crit pt} must hold. For any choice of $z \in D \not= \varnothing \,$, combining case \ref{it: cor Rabinowitz crit pts - one crit pt} with Claim \ref{claim: proof mthm rescaling Zoll - critical sets equal}, we find $(v,\tau) \in \Crit(\Rabinowitz{H_f})$ with $v(0) \in \R \times \{z\} \subseteq \R \times D$ and $v \subseteq \R \times (\Sigma \backslash D) \,$, an obvious contradiction.

This contradiction shows that our assumption was false and there does indeed exist a periodic Reeb orbit of $(\Sigma,e^f \alpha)$ intersecting $D$ with period $\tau \in [\tau(\alpha)-\eps,\tau(\alpha)+\eps] \,$. This concludes the proof of Theorem \ref{mthm: rescaling Zoll contact form}.
\end{proof}

\subsection{Proof of Theorem \ref{mthm: multiplicity result}}
\label{subsec: Proof of mthm multiplicity result}

To show Theorem \ref{mthm: multiplicity result}, we invoke Corollary \ref{cor: Rabinowitz gradient flow line} again, adopting also the notation from the previous subsection. 

We recall from Remark \ref{rmk: Rabinowitz action value defining Hamiltonian} the notion of the critical point tower. Let $H : S \Sigma \rightarrow \R$ be a Hamiltonian. We call a critical point $(v,\tau) \in \Crit(\RabinowitzH)$ a \textit{prime loop} if $v \in \loopspace$ is an embedded loop. Notice that this implies that $\tau \not= 0$ and that $X_H(v)$ vanishes nowhere. The prime loop $(v,\tau)$ induces an injection
\[
i_{(v,\tau)} : \Z \times \bbS^1 \rightarrow \Crit(\RabinowitzH) \comma \quad (k,t) \mapsto (v(t + k\,\cdot\,),\, k \tau) \,\, .
\]
Its image $i_{(v,\tau)}(\Z \times \bbS^1)$ is called the \textit{critical point tower} of $(v,\tau)$ and $i_{(v,\tau)}(\{k\} \times \bbS^1)$ the \textit{$k$-th level of the critical point tower}.

\begin{lemma}
\label{lem: action k-th level crit pt tower}
Let $H \in \Cinfty(S \Sigma)$ be a Hamiltonian and $(v,\tau) \in \Crit(\RabinowitzH)$ be a prime loop. The $\RabinowitzH$-action value on the $k$-th level of the critical point tower of $(v,\tau)$ is $k \RabinowitzH(v,\tau) \,$.
\end{lemma}

\begin{proof}
   This follows readily from Remark \ref{rmk: Rabinowitz action value defining Hamiltonian} \ref{it: Rabinowitz action functional - action value} and change of variables. 
\end{proof}

\noindent The lemma allows us to compare action values of critical points whose Hamiltonian flow lines are geometrically equal.

\begin{lemma}
\label{lem: action difference for crit pts not geometrically distinct}
Suppose the Hamiltonian $H \in \Cinfty(S \Sigma)$ complies with the $C^1$-datum $\COneDatum$. Let $(v^-,\tau^-) , \, (v^+,\tau^+) \in \Crit(\RabinowitzH)$ be critical points with $v^\pm \subseteq U_\Sigma$ and $\tau^\pm \not= 0 \,$, having different action value
\[
\Delta := \RabinowitzH(v^+,\tau^+) - \RabinowitzH(v^-,\tau^-) > 0 \,\, .
\]
If the Hamiltonian flow lines corresponding to $(v^-,\tau^-)$ and $(v^+,\tau^+)$ are geometrically equal, then $(v^-,\tau^-)$ and $(v^+,\tau^+)$ lie in the critical point tower of some prime loop $(v,\tau) \in \Crit(\RabinowitzH)$ with $0 < \tau\leq \frac{\Delta}{c} \,$.
\end{lemma}

\begin{proof}
    Notice first that, since $H$ complies with $\Dscr \,$, its Hamiltonian vector field $X_H$ does not vanish on $U_\Sigma \,$, so that $v^\pm$ are not constant. By assumption, $(v^\pm,\tau^\pm)$ lie in the same critical point tower of some prime loop $(v,\tau)$ with $\tau > 0 \,$, see Remark \ref{rmk: Rabinowitz action value defining Hamiltonian} \ref{it: Rabinowitz action functional - critical point tower}. Since $H$ complies with $\Dscr \,$, we can estimate
    \[
    \RabinowitzH(v,\tau) = \tau \int_0^1 \lambda(X_H) \circ v \,\, \d t \geq \tau c > 0 \,\, .
    \]
    We conclude with Lemma \ref{lem: action k-th level crit pt tower} that $\Delta = k \RabinowitzH(v,\tau) \geq k \tau c$ for some $k \in \N_{\geq 1} \,$.
\end{proof}

\noindent The estimate in the previous lemma allows us to infer Theorem \ref{mthm: multiplicity result}.

\begin{proof}[Proof of Theorem \ref{mthm: multiplicity result}]
    Let $(\Sigma,\alpha)$ be a Zoll contact manifold with minimal Reeb period $\tau(\alpha) \,$. We are given a $C^1$-bounded subset $\mathscr{B}_\Sigma \subseteq \Cinfty(\Sigma)$ and $\eps > 0 \,$. Choose a $C^1$-bounded subset $\mathscr{B}_{S \Sigma} \subseteq \Cinfty([-1,1] \times \Sigma)$ as in Lemma \ref{lem: C^1-bounded subset Cinfty(Sigma) --> C^1-bounded subset Cinfty(S Sigma)} and consider the $C^1$-datum 
    \[
    \Dscr = ( U_\Sigma := (-1,1) \times \Sigma \comma c := 1 \comma \mathscr{B}_{S \Sigma}) \,\, .
    \]
    Recall that $H_0 = h$ is the coordinate function.
    We now choose a constant $\delta = \delta(H_0, \, h_0 := 1 , \, \tfrac{\eps}{2} \comma \Dscr) > 0$ as in Corollary \ref{cor: Rabinowitz gradient flow line} and claim that $\delta$ has the property stated in Theorem \ref{mthm: multiplicity result}.

    To this end, let $f \in \mathscr{B}_\Sigma$ with $\lVert f\rVert_\infty \leq \delta $ be given. We suppose that $e^f \alpha$ does not have two geometrically distinct periodic Reeb orbits and have to show that it has a periodic Reeb orbit of period at most $\eps \,$. We have already seen in the proof of Theorem \ref{mthm: rescaling Zoll contact form} that $H_f$ complies with $\Dscr$ and that $\sup_{S\Sigma}| H_f - H_0 | = \lVert f \rVert_\infty \leq \delta \,$. So, by choice of $\delta \,$, case \ref{it: cor Rabinowitz crit pts - two crit pts} in Corollary \ref{cor: Rabinowitz gradient flow line} must hold with $H = H_f \,$. Indeed, otherwise case \ref{it: cor Rabinowitz crit pts - one crit pt} in the corollary would hold, so that (by Lemma \ref{lem: Reeb orbits Gr(f) <--> e^f alpha Reeb orbits}) there is a periodic Reeb orbit of $e^f \alpha$ through every point of $\Sigma \,$, contradicting our assumption that $e^f \alpha$ does not have two geometrically distinct periodic Reeb orbits (this uses $\dim \Sigma > 1$). So, since case \ref{it: cor Rabinowitz crit pts - two crit pts} holds, we find two critical points $(v^\pm,\tau^\pm) \in \Crit(\Rabinowitz{H_f})$ with $\tau^\pm \not= 0 \comma v_j^\pm \subseteq U_\Sigma$ and
    \[
    0 < \Rabinowitz{H_f}(v^+,\tau^+) - \Rabinowitz{H_f}(v^-,\tau^-) \leq  \eps \,\, .
    \]
    Because $(\Sigma,e^f \alpha)$ does not have two geometrically distinct periodic Reeb orbits, the $H_f\,$-Hamiltonian flow lines $v^-(\frac{\cdot}{\tau^-}) \comma v^+(\frac{\cdot}{\tau^+})$ are geometrically equal (using Lemma \ref{lem: Reeb orbits Gr(f) <--> e^f alpha Reeb orbits}). Now Lemma \ref{lem: action difference for crit pts not geometrically distinct} implies that $(v^\pm,\tau^\pm)$ lie in the critical point tower of some prime loop $(v,\tau) \in \Crit(\Rabinowitz{H_f})$ with $0 < \tau \leq \frac{\eps}{c} = \eps \,$. By Lemma \ref{lem: Reeb orbits Gr(f) <--> e^f alpha Reeb orbits}, this prime loop corresponds to a periodic Reeb orbit of $e^f \alpha$ of minimal period $\tau \leq \eps \,$.
\end{proof}

\section{Proof of Theorem \ref{mthm: Rabinowitz gradient flow line}: Preparation}
\label{sec: Proof MThm - Preparation}

\noindent Ultimately, the proof of Theorem \ref{mthm: Rabinowitz gradient flow line} will rely on a homotopy stretching argument. In this section, we pursue the first steps in this direction.

For the present section, we fix once and for all a Zoll contact manifold $(\Sigma,\alpha)$, whose minimal Reeb period we denote by $\tau(\alpha) > 0 \,$, a positive constant $h_0 > 0$ and a defining Hamiltonian $H_0 \in \Hcal^{h_0} \subseteq \Cinfty(S\Sigma)$. We moreover fix an almost complex structure $J$ on the symplectization $S \Sigma \,$, compatible with $\d (e^h \alpha) \,$. We do not yet require $J$ to be of SFT-type in this section. For convenience we abbreviate
\[
(W ,\lambda, \omega) := (S\Sigma, \, e^h \alpha , \, \d (e^h \alpha)) \,\, .
\]

\subsection{The critical submanifold $\CCrit$}
\label{subsec: Preparation - Critical Submanifold Cscr}

Since $(\Sigma,\alpha)$ is a Zoll contact manifold and $H_0$ is a defining Hamiltonian, the critical set of the Rabinowitz action functional $\RabinowitzHZero$ is a disjoint union of copies of $\Sigma$, as the next lemma explicates.

The \textit{spectral gap} of a connected component $\mathscr{C} \subseteq \Crit(\RabinowitzHZero)$ by definition is
\begin{equation}
\label{eq: specgap def}
\mathrm{specgap}(\RabinowitzHZero,\mathscr{C}) := \inf\set{|\RabinowitzHZero (v,\tau) - \RabinowitzHZero(\mathscr{C})|}{(v,\tau) \in \Crit(\RabinowitzHZero) \backslash \mathscr{C}} \,\, ,
\end{equation}
where $\RabinowitzHZero(\mathscr{C})$ is the common value of $\RabinowitzHZero$ on $\mathscr{C} \,$.

\begin{lemma}
\label{lem: Crit(RabinowitzHZero)}
There is a natural diffeomorphism $\Crit(\RabinowitzHZero) \cong \Sigma \times \Z$ given by
\begin{equation}
\label{eq: Crit(RabinowitzHzero) -> Sigma times Z}
\Crit(\RabinowitzHZero) \overset{\cong}{\longrightarrow} \Sigma \times \Z \comma \quad (v,\tau) \mapsto \left(v(0), \tfrac{\tau}{\tau(\alpha)} \right) \,\, .
\end{equation}
Moreover, $\RabinowitzHZero$ is Morse-Bott along $\Crit(\RabinowitzHZero)$, that is 
\[
T_{(v,\tau)} \Crit(\RabinowitzHZero) = \ker \mathrm{Hess} (\RabinowitzHZero) (v,\tau ) \qquad \Forall (v,\tau) \in \Crit(\RabinowitzHZero) \,\, .
\]
The common value of $\RabinowitzHZero$ on the component of $\Crit(\RabinowitzHZero)$ corresponding to $\Sigma \times \{k\}$ is $k \tau(\alpha) \,$. In particular, each component of $\Crit(\RabinowitzHZero)$ has spectral gap $\tau(\alpha)\,$.
\end{lemma}

\begin{proof}
As $R_\alpha = X_{H_0}$ along $\Sigma = H_0^{-1}(0)$ since $H_0$ is a defining Hamiltonian, critical points of $\RabinowitzHZero$ are $1$-periodic constant reparametrizations of closed Reeb orbits, see \eqref{eq: Crit(Rabinowitz)}.
Denote by $H^1(\bbS^1,W)$ the Hilbert manifold of absolutely continuous loops with square integrable weak derivatives and by $(\phi_{R_\alpha}^t)_{t \in \R} $ the Reeb flow. Notice that $\phi_{R_\alpha}^{t} = \phi_{R_\alpha}^{t + \tau(\alpha)} \,$. The inverse of \eqref{eq: Crit(RabinowitzHzero) -> Sigma times Z} is given explicitly by
\begin{equation}
\label{eq: Sigma times Z -> Crit(RabinowitzHzero) embedding}
i : \Sigma \times \Z \hookrightarrow H^1(\bbS^1, W) \times \R \comma \quad (z,k) \mapsto \big( \big[ t \mapsto  \phi_{R_\alpha}^{k\tau(\alpha)t}(z) \big] , \, k\tau(\alpha) \big) \,\, .
\end{equation}
This is a smooth embedding onto the submanifold $\Crit(\RabinowitzHZero) \subseteq H^1(\bbS^1,W) \times \R \,$.

That $\RabinowitzHZero$ is Morse-Bott is asserted in \cite[\S3.2]{Cieliebak_Frauenfelder} and \cite[Lemma 20]{FauckThesis}. It relies on $H_0$ being a defining Hamiltonian and on the Zoll assumption, which implies assumption (A) in \cite{Cieliebak_Frauenfelder} respectively (MB) in \cite{FauckThesis}.

The assertion about the action value holds by Remark \ref{rmk: Rabinowitz action value defining Hamiltonian} \ref{it: Rabinowitz action functional - action value}.
\end{proof}

\begin{remark}
\label{rmk: topology on Crit(RabinowitzHzero)}
The map \eqref{eq: Sigma times Z -> Crit(RabinowitzHzero) embedding} is also a topological embedding $i : \Sigma \times \Z \hookrightarrow H^k(\bbS^1,W) \times \R$ for every $k \geq 1$ as well as a topological embedding $i : \Sigma \times \Z \hookrightarrow \loopspace \times \R = \Cinfty(\bbS^1,W) \times \R \,$. This follows from compactness of $\Sigma$ and the fact that continuous maps from compact spaces to Hausdorff spaces are topological embeddings. Hence the topology on $\Crit(\RabinowitzHZero)$ is independent of the ambient space of maps.
\end{remark}

Subsequently, we will denote by $\CCrit$ the connected component of $\Crit(\RabinowitzHZero)$ of loops with Lagrange multiplier $\tau(\alpha) \,$. Explicitly, this means
\begin{align*}
\CCrit &= \set{(v,\tau) \in \Crit(\RabinowitzHZero)}{\tau= \tau(\alpha)}      \\
&= \set{(v,\tau(\alpha)) \in \loopspace \times \R}{v(0) \in \Sigma \comma v(t) = \phi_{R_\alpha}^{\tau(\alpha) t} (v(0)) \,\,\, \Forall t \in \bbS^1} \, ,
\end{align*}
where $(\phi_{R_\alpha}^t)_t$ denotes the Reeb flow. As the previous lemma shows, $\CCrit$ is a smooth manifold without boundary naturally diffeomorphic to $\Sigma$ via
\begin{equation}
\label{eq: CCrit diffeomorphism Sigma}
\CCrit \overset{\cong}{\longrightarrow} \Sigma \comma \quad (v,\tau(\alpha)) \mapsto v(0) \,\, .
\end{equation}
In particular, $\CCrit$ is connected and compact. The common value of $\RabinowitzHZero$ on $\CCrit$ is $\tau(\alpha)$ and $\CCrit$ has spectral gap $\tau(\alpha) \,$.

\subsection{The moduli space $\moduliH{r}$}
\label{subsec: Preparation - moduli space}

Let us fix once and for all a family $(\beta_r)_{r \geq 0} \subseteq \Cinfty(\R,[0,1])$ of bump functions with the following properties.
\begin{enumerate}[label=($\beta$.\arabic*)]
    \item\label{it: beta_r choice - smooth} The map $\R_{\geq 0 } \times \R \rightarrow \R \comma (r,s) \mapsto \beta_r(s) \,$, is smooth.
    \item\label{it: beta_r choice - beta_0} $\beta_0 \equiv 0$
    \item\label{it: beta_r choice - zero at infinity} $\beta_r (s) = 0$ for $s \notin (-r-1, r+1)$
    \item\label{it: beta_r choice - one on [-r,r]} $\beta_R(s) = 1$ for $s \in [-r,r]$ and $r \geq 1$
    \item\label{it: beta_r choice - derivatives} $\frac{d}{ds}\beta_r (s) \geq 0$ for $s \leq 0$ and $\frac{d}{ds}\beta_r (s) \leq 0$ for $s \geq 0 \,$.
\end{enumerate}
A possible choice of such a family is depicted in Figure \ref{fig: beta_r bump functions}.
\begin{figure}
    \centering
    \includegraphics[width=0.8\linewidth]{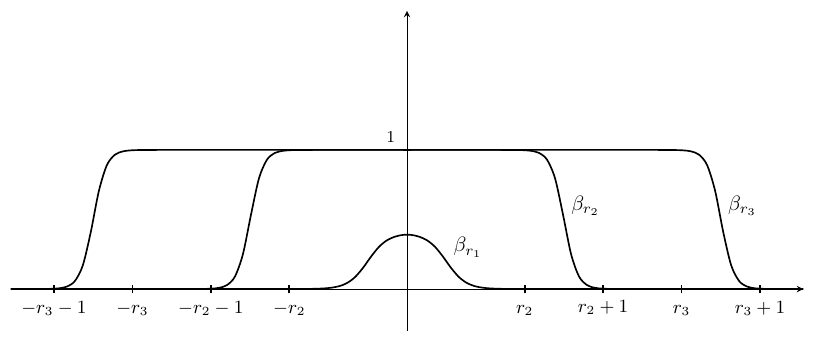}
    \caption{Bump functions $\beta_r$ for parameters $0 < r_1 < 1 < r_2 < r_2+1 < r_3 \,$. (This figure is adapted from \cite[Fig.~2]{Abstract_Perturbations}.)}
    \label{fig: beta_r bump functions}
\end{figure}
For every Hamiltonian $H \in \Cinfty(W)$ and parameters $(r,s) \in \R_{\geq 0} \times \R$ we define the functional
\begin{align}
\label{eq: def RabinowitzH_rs}
\RabinowitzHrs : \loopspace \times \R \rightarrow \R \comma \quad \RabinowitzHrs := (1 - \beta_r(s)) \RabinowitzHZero + \beta_r(s) \RabinowitzH \,\, ,
\end{align}
which interpolates between $\RabinowitzHZero$ and $\RabinowitzH \,$. Its gradient is given by
\[
\nabla \RabinowitzHrs(v,\tau) = (1-\beta_r(s)) \nabla \RabinowitzHZero(v,\tau) + \beta_r(s) \nabla \RabinowitzH(v,\tau) \,\, .
\]
(As $J$ is fixed throughout, all gradients appearing are understood to be $g_J$-gradients.) The \textit{energy} of any $w = (u,\eta) : \R \rightarrow \Cinfty(\loopspace \times \R)$, not necessarily a gradient flow line, is defined to be
\[
E(w) := \int_{-\infty}^{\infty} |\partial_s w(s)|_{g_J}^2 \, \d s \in [0,\infty] \,\, .
\]
Again we write $E(w)$ instead of $E_J(w)$ since $J$ is fixed.
For convenience, given a cylinder $u \in \Cinfty(\R \times \bbS^1,W)  \,$, we also write 
\[
u_s := u(s,\,\cdot\,) \in \loopspace 
\]
for its loop at $s \in \R \,$.

\begin{definition}[Homotopic rel $\CCrit$]
\label{def: homotopic rel C^tau(alpha)} Write $\Rbar := \R \cup \{\pm \infty\} \,$.
\begin{enumerate}[label=(\alph*)]
    \item Let $u^0,u^1\in C^0(\Rbar \times \bbS^1,W)$ be two cylinders with asymptotics $(u^{0,1}_{\pm \infty} , \tau(\alpha)) \in \CCrit \,$. We say $u^0$ and $u^1$ are \textit{homotopic rel $\CCrit$}, in symbols $u^0 \simeq u^1$ rel $\CCrit$, if there exists a continuous homotopy $(u^{\nu})_{0 \leq \nu \leq 1}$ of cylinders from $u^0$ to $u^1$ so that
\[
(u^\nu_{\pm \infty}, \tau(\alpha)) \in \CCrit \quad \Forall \nu \in [0,1] \,\, .
\]
    \item For a cylinder $u$ we write $u \simeq 0 \text{ rel } \CCrit$ if $u \simeq u_v$ rel $\CCrit \,$, with $u_v$ the $s$-independent cylinder $u_v(s,t) := v(t)$, for any choice of $(v,\tau(\alpha)) \in \CCrit \,$. Since $\CCrit \cong \Sigma$ is connected, this is independent of the choice of loop $v$.
\end{enumerate}
\end{definition}

\begin{remark}
\label{rmk: homotopic rel C^tau(alpha)}
Write the components of a cylinder as $u = (u_\R, u_\Sigma) \subseteq W = \R \times \Sigma \,$. Viewing $\bbS^2$ as $\Rbar \times \bbS^1$ with ends $\{\pm \infty\} \times \bbS^1$ collapsed (\ie as unreduced suspension of the circle), it holds $(u_{\pm\infty},\tau(\alpha)) \in \CCrit$ if and only if $(u_\R)_{\pm \infty} = 0$ and
    \[
    \Rbar \times \bbS^1 \rightarrow \Sigma \comma \quad (s,t) \mapsto \phi_{R_\alpha}^{-\tau(\alpha)t} (u_\Sigma(s,t)) 
    \]
    descends to a map $\Bar{u} : \bbS^2 \rightarrow \Sigma \,$. Now, two cylinders $u^0$ and $u^1$ are homotopic rel $\CCrit$ if and only if the descended maps $\Bar{u}^0$ and $\Bar{u}^1$ are freely homotopic in $C^0(\bbS^2,\Sigma) \,$. In particular, $u \simeq 0 \text{ rel } \CCrit$ if and only if $\Bar{u}$ is null-homotopic.
\end{remark}

\noindent We now define the moduli spaces of $s$-dependent gradient flow lines of the interpolation functionals $\RabinowitzHrs$ by setting for each $r \geq 0$
\begin{align*}
    \moduliH{r} &:= \Bigset{w = (u,\eta) \in \Cinfty(\R,\loopspace \times \R)}{\begin{array}{c}
        \partial_s w(s) = \nabla \RabinowitzHrs(w(s)) \quad \Forall s \in \R \comma \\
      \lim_{s \rightarrow \pm \infty} w(s) \in \CCrit \comma \, u \simeq 0 \text{ rel } \CCrit \,\, .
    \end{array}}
\end{align*}
The $s$-dependent gradient equation $0 = \partial_s w - \nabla \RabinowitzHrs (w)$ spelled out reads
\begin{equation}
\label{eq: s-dependent gradient eq}
\begin{cases}
    0 = \partial_s u + J(u) \big( \partial_t u - \eta(s) \big( (1-\beta_r(s)) X_{H_0}(u) + \beta_r(s) X_H(u) \big) \big) & \Forall (s,t) \in \R \times \bbS^1 \\
    0 = \partial_s \eta(s) + \int_0^1 \big( (1-\beta_r(s)) H_0(u(s,t)) + \beta_r(s) H(u(s,t)) \big) \, \d t & \Forall s \in \R \, . 
\end{cases}
\end{equation}
Notice that, by \ref{it: beta_r choice - zero at infinity} and \ref{it: beta_r choice - one on [-r,r]}, an element $w \in \moduliH{r}$ is a gradient flow line of $\RabinowitzHZero$ on $\R \backslash (-r-1,r+1) $ and, for $r \geq 1 \,$, a gradient flow line of $\RabinowitzH$ on $[-r,r] \,$. In particular $w$ has finite energy, being asymptotically an ordinary gradient flow line of $\RabinowitzHZero$ with existing limits $w(\pm \infty) \,$. 

\begin{remark}
\label{rmk: Preparation - moduli^H_r limits}
The limit $w(s) \overset{s \rightarrow \pm \infty}{\longrightarrow}w(\pm \infty)$ appearing in the definition of $\moduliH{r}$ is understood to be with respect to the $(\Cinfty(\bbS^1,W) \times \R)$-topology. However, requiring merely convergence in the $(C^0(\bbS^1,W) \times \R)$-topology and finite energy, $E(w) < \infty \,$, already implies convergence in the $(\Cinfty(\bbS^1,W) \times \R)$-topology. Indeed, for every sequence $(s_n)_n \subseteq \R$ tending to $\pm \infty \,$, a subsequence of $(w(s_n))_n$ converges in $\Cinfty(\bbS^1,W) \times \R$ to a critical point of $\RabinowitzHZero$ (see Corollary \ref{cor: Gromov compactness - single gradient flow line crit pts at ends} later), which must agree with the $(C^0(\bbS^1,W) \times \R)$-limit.
\end{remark}

\noindent We combine the various moduli spaces $\moduliH{r}$ to moduli spaces
\begin{align}
    \moduli^H &:= \set{(r,w) \in \R_{\geq 0} \times \Cinfty(\R,\loopspace \times \R) }{ w \in \moduliH{r}} \label{eq: def moduli^H}\\
    \moduliH{[0,R]} &:= \set{(r,w) \in \moduli^H}{0 \leq r \leq R} \comma \quad \text{ for } R \geq 0 \,\, . \notag
\end{align}
The next lemma, which is completely analogous to Lemmata 3.3 and 3.4 in \cite{Albers_Hein}, bounds the action value and the energy of a trajectory in the moduli space.

\begin{lemma}
\label{lem: action estimate for w in moduli}
For $(r,w)  \in \moduli^H$ let $\mathrm{im}(w) = w(\R) \subseteq \loopspace \times \R$ denote the image of $w$ and define
\[
\Delta := \sup_{\image(w)} |\RabinowitzH - \RabinowitzHZero|  \in [0,\infty] \,\, .
\]
Then $E(w) \leq 2\Delta$ and moreover
\begin{align}
    \label{eq: action estimate}
    \RabinowitzHZero(w(s)) \comma \RabinowitzHrs(w(s)) \in [\tau(\alpha) - 2\Delta,\, \tau(\alpha) + 2\Delta] \qquad \Forall s \in \R \,\, .
\end{align}
\end{lemma}

\begin{proof}
    There is nothing to show if $\Delta = \infty \,$, so we may assume $\Delta < \infty \,$.
    The gradient flow equation yields
    \begin{align*}
        \frac{\d}{\d s} \RabinowitzHrs(w(s)) &=  \d \RabinowitzHrs (w(s)) \cdot \partial_s w(s) + \frac{\partial \RabinowitzHrs}{\partial s}(w(s)) \notag\\
        &=|\partial_s w(s)|^2_{g_J} + \beta_r'(s) \, ( \RabinowitzH - \RabinowitzHZero )(w(s)) \,\, .  
    \end{align*}
    Thus, for fixed $a \leq b \,$, the energy on $[a,b]$ computes to
    \begin{align}
        0 \leq E^{b}_a (w) &= \int_a^b |\partial_s w(s)|_{g_J}^2 \, \d s \notag\\
        &=  \RabinowitzHParam{r,b}(w(b)) - \RabinowitzHParam{r,a}(w(a)) - \int_{a}^{b} \beta_r'(s) \, (\RabinowitzH - \RabinowitzHZero)(w(s)) \, \d s \, . \label{eq: action estimate - E_a^b(w) for w in moduli^H}
    \end{align}
    Due to properties \ref{it: beta_r choice - zero at infinity} and \ref{it: beta_r choice - derivatives} and since $0 \leq \beta_r(s) \leq 1 \,$, the absolute value of the second summand can be bounded by
    \begin{align}
     &\phantom{=} \,\big| \int_{a}^{b} \beta_r'(s) \, (\RabinowitzH - \RabinowitzHZero)(w(s)) \, \d s \big|  \notag\\
     &\leq \int_{-r-1}^{0} \mathbbm{1}_{[a,b]}(s) \, \beta_r'(s) \, \big|(\RabinowitzH - \RabinowitzHZero)(w(s)) \big| \, \d s \notag\\
     &\phantom{\leq} + \int_{0}^{r+1} \mathbbm{1}_{[a,b]}(s) \, (-\beta_r'(s)) \, \big|(\RabinowitzH - \RabinowitzHZero)(w(s))\big| \, \d s \notag\\
     &\leq \big( \sup_{\image(w)} |\RabinowitzH - \RabinowitzHZero| \big) \big( \int_{-r-1}^{0} \beta_r'(s) \, \d s - \int_{0}^{r+1} \beta_r'(s) \, \d s \big) \notag\\
     &= 2 \Delta  \, \beta_r(0)  \leq 2 \Delta \,\, . \label{eq: action estimate - int beta'_r (RabinowitzH - RabinowitzHZero)} 
    \end{align}
The asymptotic condition $w(\pm \infty) \in \CCrit$ implies
\begin{align}
\RabinowitzHParam{r,\pm\infty}(w(\pm\infty)) = \RabinowitzHZero(w(\pm\infty)) = \tau(\alpha) \,\, . \label{eq: action estimate - asymptotic action}
\end{align}
This, together with \eqref{eq: action estimate - E_a^b(w) for w in moduli^H} and \eqref{eq: action estimate - int beta'_r (RabinowitzH - RabinowitzHZero)}, yields $E(w) \leq 2 \Delta \,$.

Taking separately the limits $a \rightarrow - \infty$ respectively $b \rightarrow \infty$ in \eqref{eq: action estimate - E_a^b(w) for w in moduli^H} and combining this with the estimate in \eqref{eq: action estimate - int beta'_r (RabinowitzH - RabinowitzHZero)} and the asymptotic condition \eqref{eq: action estimate - asymptotic action} shows
\[
\tau(\alpha) - 2 \Delta \leq \RabinowitzHrs(w(s)) \leq \tau(\alpha) + 2 \Delta \qquad \Forall s \in \R \,\, .
\]
To conclude \eqref{eq: action estimate}, we also need to show that $\RabinowitzHZero(w(s))$ is at most $2\Delta$ apart from $\tau(\alpha) \,$. First assume $s \geq 0 \,$, so that $\beta'_r(\Bar{s}) \leq 0$ for every $\Bar{s} \geq s \,$. Rearranging \eqref{eq: action estimate - E_a^b(w) for w in moduli^H} with $a:=s$ and $b :=\infty\,$, we see
\[
\RabinowitzHrs(w(s)) -\tau(\alpha) \leq - \int_{s}^\infty \beta_r'(\Bar{s}) \, (\RabinowitzH - \RabinowitzHZero)(w(\Bar{s})) \, d \Bar{s} \leq -\Delta \int_{s}^{\infty}  \beta_r'(\Bar{s}) \, \d \Bar{s} = \beta_r(s) \, \Delta \,\, .
\]
Substituting the definition \eqref{eq: def RabinowitzH_rs} of $\RabinowitzHrs$ and rearranging yields
\[
\RabinowitzHZero(w(s)) - \tau(\alpha) \leq \beta_r(s)  \Delta  - \beta_r(s) \, (\RabinowitzH-\RabinowitzHZero)(w(s)) \leq 2 \beta_r(s) \, \Delta \leq 2 \Delta \,\, .
\]
Next we show that also $\RabinowitzHZero(w(s)) - \tau(\alpha) \geq - 2 \Delta \,$. To this end, rearrange \eqref{eq: action estimate - E_a^b(w) for w in moduli^H} with $a := -\infty$ and $b := s$ to obtain
\begin{align*}
    \RabinowitzHrs(w(s)) - \tau(\alpha) &\geq \int_{-\infty}^s \beta_r'(\Bar{s}) \, (\RabinowitzH - \RabinowitzHZero)(w (\Bar{s})) \, \d \Bar{s} \\
    &\overset{\text{\ref{it: beta_r choice - derivatives}}}{\geq} - \Delta \int_{-\infty}^0 \beta_r'(\Bar{s}) \, \d \Bar{s} + \Delta \int_0^s \beta_r'(\Bar{s}) \, \d \Bar{s} 
     = \Delta ( \beta_r(s) - 2 \beta_r(0)) \,\, ,
\end{align*}
Hence, unraveling again the definition of $\RabinowitzHrs \,$,
\begin{align*}
    \RabinowitzHZero(w(s)) - \tau(\alpha) \geq \underbrace{\beta_r(s)}_{\geq 0} \, \big(\underbrace{\Delta - (\RabinowitzH - \RabinowitzHZero)(w(s)) }_{\geq 0}  \big) - 2 \underbrace{\beta_r(0)}_{\leq 1} \, \Delta \geq - 2 \Delta \,\, .
\end{align*}
The case $s \leq 0$ works analogously, swapping the roles of $a$ and $b$ in \eqref{eq: action estimate - E_a^b(w) for w in moduli^H} in the above argument.
\end{proof}

\subsection{Uniform bound on the Lagrange multiplier}
\label{subsec: Preparation - Lagrange Multiplier Bound}

We now prove a uniform bound on the Lagrange multiplier $\eta$ of $(r,w) = (r,(u,\eta)) \in \moduli^H\,$. This has a twofold purpose: On the one hand, such a Lagrange multiplier bound is required for the compactness of the moduli spaces $\moduli^H_{[0,R]} \,$. This problem motivated the original Lagrange multiplier bound by Cieliebak-Frauenfelder \cite{Cieliebak_Frauenfelder}. On the other hand, and this is a new problem, we also need the Lagrange multiplier bound to control the energy $E(w)$ for $(r,w) \in \moduli^H \,$, uniform for every $H$ sufficiently $C^0$-close to $H_0 \,$. Let us now state our result precisely. 

\begin{proposition}
\label{prop: uniform Lagrange multiplier bound}
For every $C^1$-datum $\Dscr$ there exist constants $T^* , \delta^* > 0$ with
\begin{align*}
    \Forall H \in \Hclassdelta{\delta^*} : \,\, \Forall (r,(u,\eta)) \in \moduli^H : \quad \lVert \eta \rVert_\infty \leq T^* \,\, .
\end{align*}
\end{proposition}

\noindent Recall that we have fixed $h_0 >0$ and $H_0 \in \Hcal^{h_0}$ at the beginning of the section. The class of functions $\Hclassdelta{\delta^*}$ has been defined in \eqref{eq: def Hclassparam}. 

The proof of the proposition remains close to the proof of \cite[Cor.~3.7]{Cieliebak_Frauenfelder}. It starts with a lemma, which is the analog of \cite[Prop.~3.4]{Cieliebak_Frauenfelder}.

\begin{lemma}
\label{lem: Cieliebak-Frauenfelder}
   Given a $C^1$-datum $\Dscr$, there exist constants $C, \delta > 0$ such that for every $H \in \Hclassparam$, every $(v,\tau) \in \loopspace \times \R$ and every $(r,s) \in \R_{\geq 0} \times \R$ it holds
   \begin{align*}
      | \nabla \RabinowitzHrs (v,\tau) |_{g_J} \leq \delta \,\, \Longrightarrow \,\, |\tau| \leq C \, |\RabinowitzHrs(v,\tau)| + C \,\, . 
   \end{align*}
\end{lemma}

\noindent The proof of the lemma is entirely analogous to the one of \cite[Prop.~3.4]{Cieliebak_Frauenfelder} but, for the sake of completeness, presented in Appendix \ref{app sec: Cieliebak-Frauenfelder estimate}. It relies crucially on the fact that we restrict ourselves to Hamiltonians $H$ that comply with $\Dscr \,$.

\begin{proof}[Proof of Proposition \ref{prop: uniform Lagrange multiplier bound}]
    Let $\Dscr$ be an arbitrary $C^1$-datum which we keep fixed throughout. Choose constants $C = C(\Dscr) , \,\delta=\delta(\Dscr) > 0$ as in Lemma \ref{lem: Cieliebak-Frauenfelder}. The parabola
    \begin{align}
    \label{eq: Lagrange multiplier bound - parabola}
    P(x) :=  \frac{2}{\delta^2} \, x^2 +x \left( 2 C + \frac{2 \lVert H_0 \rVert_\infty}{\delta^2} \right)  \quad , \,\, x \in \R \, ,
    \end{align}
    goes through the origin and is positive for $x > 0 \,$. Hence we can choose $\delta^* \in (0,\delta)$ with $P(\delta^*) \in (0,1) \,$. We define
    \[
    T^* := \frac{C (\tau(\alpha) +1)}{1-P(\delta^*)} > 0 \,\, .
    \]
    We claim that $\delta^*$ and $T^*$ have the significance stated in Proposition \ref{prop: uniform Lagrange multiplier bound}. To this end, let $H \in \Hcal^{h_0}(H_0,\Dscr, \delta^*)$ and $(r,w) = (r,(u,\eta)) \in \moduli^H$ be arbitrary. We have to show that $\lVert \eta \rVert_\infty \leq T^* \,$.

    We start by observing that, because $\lVert H - H_0 \rVert_\infty \leq \delta^* \,$, we have
    \begin{align*}
        |\RabinowitzH(v,\tau) - \RabinowitzHZero(v,\tau)| &= \big| \tau \int_0^1 (H-H_0)(v) \, \d t \big| \leq \delta^* \, |\tau| \qquad \Forall (v,\tau) \in \loopspace \times \R 
    \end{align*}
    and therefore
    \begin{equation}
    \label{eq: Lagrange multiplier bound - Delta estimate}
      \sup_{\mathrm{im}(w)} |\RabinowitzH - \RabinowitzHZero| \leq \delta^* \, \lVert \eta \rVert_\infty  \,\, .
    \end{equation}
    Now let $\sigma \in \R$ be arbitrary but fixed. We define
    \begin{align*}
        \tau_\sigma := \inf \set{s \geq 0}{ | \nabla \RabinowitzHParam{r,\sigma + s} (w(\sigma +s ) ) |_{g_J} < \delta} \in [0,\infty] \,\, .
    \end{align*}
    Observe that 
    \begin{align*}
        \infty > E(w) &= \int_{-\infty}^{\infty} |\nabla \RabinowitzHrs(w(s))|_{g_J}^2 \, \d s \geq \int_{\sigma}^{\sigma + \tau_\sigma} |\nabla \RabinowitzHrs(w(s))|_{g_J}^2 \, \d s \geq \tau_\sigma \, \delta^2 \,\, ,
    \end{align*}
    so that $\tau_\sigma$ is finite and
    \begin{equation*}
    \tau_\sigma \leq \frac{E(w)}{\delta^2} \,\, .
    \end{equation*}
    Now Lemma \ref{lem: action estimate for w in moduli} and estimate \eqref{eq: Lagrange multiplier bound - Delta estimate} imply that
    \[
    E(w) \leq 2 \sup_{\mathrm{im}(w)} |\RabinowitzH - \RabinowitzHZero| \leq 2 \, \delta^* \, \lVert \eta \rVert_\infty 
    \]
    and hence
    \begin{equation}
    \label{eq: Lagrange multiplier bound - tau_sigma}
    \tau_\sigma \leq  \frac{2 \delta^*}{\delta^2} \, \lVert \eta \rVert_\infty \,\, .
    \end{equation}
   Setting $a := -\infty$ and $b := \sigma + \tau_\sigma $ in \eqref{eq: action estimate - E_a^b(w) for w in moduli^H} and using \eqref{eq: action estimate - int beta'_r (RabinowitzH - RabinowitzHZero)}, \eqref{eq: action estimate - asymptotic action} and \eqref{eq: Lagrange multiplier bound - Delta estimate}, we infer
   \begin{align*}
       \tau(\alpha) - 2 \delta^* \, \lVert \eta \rVert_\infty \leq \RabinowitzHParam{r,\sigma + \tau_\sigma} (w(\sigma + \tau_\sigma)) \,\, .
   \end{align*}
   Repeat the same argument with $a := \sigma + \tau_\sigma$ and $b := \infty \,$, to obtain 
   \[
  \RabinowitzHParam{r,\sigma + \tau_\sigma} (w(\sigma + \tau_\sigma)) \leq \tau(\alpha) + 2 \delta^* \, \lVert \eta \rVert_\infty  \,\, ,
   \]
   so that
   \begin{equation}
    \label{eq: Lagrange multiplier bound - RabinowitzHrs bound}
    |\RabinowitzHParam{r,\sigma + \tau_\sigma} (w(\sigma + \tau_\sigma)) | \leq \tau(\alpha) +  2 \delta^* \, \lVert \eta \rVert_\infty \,\, .
   \end{equation}
   Since $\delta^* < \delta \,$, we know that $H \in \Hclassdelta{\delta^*} \subseteq \Hclassdelta{\delta} \,$. Observe moreover that, by definition of $\tau_\sigma \,$, it holds
   \[
   |\nabla \RabinowitzHParam{r,\sigma + \tau_\sigma} (w(\sigma + \tau_\sigma))|_{g_J} \leq \delta \,\, .
   \]
   Having chosen $C$ and $\delta$ as in Lemma \ref{lem: Cieliebak-Frauenfelder}, we conclude that
   \begin{align}
   \label{eq: Lagrange multiplier bound - eta(sigma + tau_sigma) bound}
    |\eta(\sigma + \tau_\sigma)| \leq C \, |\RabinowitzHParam{r,\sigma + \tau_\sigma}(w(\sigma + \tau_\sigma))| + C \,\, .
   \end{align}
   The Lagrange multiplier component in the gradient equation \eqref{eq: s-dependent gradient eq} allows us to estimate
   \begin{align*}
       |\eta(\sigma)| &\leq |\eta(\sigma + \tau_\sigma)| + \int_{\sigma}^{\sigma + \tau_\sigma} |\partial_s \eta(s)| \, \d s \\
       &= |\eta(\sigma + \tau_\sigma)| + \int_{\sigma}^{\sigma + \tau_\sigma} \Big| \int_0^1 H_0(u(s,t)) \, \d t + \underbrace{\beta_r(s)}_{ 0 \leq \bullet \leq 1} \int_0^1 (H- H_0)(u(s,t)) \, \d t \Big| \, \d s \\
       &\leq |\eta(\sigma + \tau_\sigma)| + \tau_\sigma \big( \lVert H_0 \rVert_\infty + \underbrace{\lVert H - H_0 \rVert_\infty}_{\leq \delta^*} \big) \\
       &\leq C \left( \tau(\alpha) + 2 \delta^* \, \lVert \eta \rVert_\infty\right) + C + \frac{2 \delta^*}{\delta^2} \, \lVert \eta \rVert_\infty \, \big( \lVert H_0 \rVert_\infty + \delta^* \big)\,\, ,
   \end{align*}
   where for the last inequality we use the estimates \eqref{eq: Lagrange multiplier bound - tau_sigma}, \eqref{eq: Lagrange multiplier bound - RabinowitzHrs bound} and \eqref{eq: Lagrange multiplier bound - eta(sigma + tau_sigma) bound}. Notice that the right-hand side is independent of $\sigma \,$. 
    
    Since $\sigma$ was arbitrary, we have shown
    \begin{align*}
      \lVert \eta \rVert_\infty &\leq C \left( \tau(\alpha) + 2 \delta^* \, \lVert \eta \rVert_\infty\right) + C + \frac{2 \delta^*}{\delta^2} \, \lVert \eta \rVert_\infty \, \big( \lVert H_0 \rVert_\infty + \delta^* \big)  \\
      &= P(\delta^*) \, \lVert \eta \rVert_\infty + C (\tau(\alpha) + 1) \,\, ,
    \end{align*}
    where we remind the reader of the definition of the parabola $P$ in \eqref{eq: Lagrange multiplier bound - parabola}. Rearranging and dividing by $1 - P(\delta^*) \in (0,1)$ yields
    \begin{align*}
        \lVert \eta \rVert_\infty \leq \frac{C(\tau(\alpha) +1)}{1 - P(\delta^*)} = T^* \,\, .
    \end{align*}
    This finishes the proof of the proposition.
   \end{proof}

\begin{corollary}
\label{cor: uniform Lagrange multiplier bound}
For given $C^1$-datum $\Dscr$ let $T^* = T^*(\Dscr), \, \delta^* = \delta^*(\Dscr)>0$ be as in Proposition \ref{prop: uniform Lagrange multiplier bound}. Then for every $0 < \delta \leq \delta^* \,$, every $H \in \Hclassdelta{\delta}$ and every $(r,w) \in \moduli^H$ it holds
\[
\sup_{\mathrm{im}(w)} |\RabinowitzH - \RabinowitzHZero| \leq \delta \,T^*
\]
and moreover $E(w) \leq 2 \delta \,T^* \,$.
\end{corollary}

\begin{proof}
    Write $w = (u,\eta) \in \Cinfty(\R,\loopspace \times \R) \,$. For every $s \in \R$ we have
    \[
    |\RabinowitzH(w(s)) - \RabinowitzHZero(w(s))| = \Big| \eta(s) \int_0^1 (H-H_0)(u(s,t)) \, \d t \Big| \leq T^* \, \lVert H - H_0 \rVert_\infty \leq T^* \delta \,\, ,
    \]
    showing the first inequality. (This is the same computation as for the estimate in \eqref{eq: Lagrange multiplier bound - Delta estimate}.) The energy estimate now follows from Lemma \ref{lem: action estimate for w in moduli}.
\end{proof}

\subsection{Cylinders are confined to a compactum}
\label{subsec: Preparation - Cylinders in compact region}

\noindent For the desired compactness of the moduli spaces $\moduli^H_{[0,R]}$ we not only need to bound the Lagrange multiplier of its elements uniformly but also ensure that the cylinders $u = (u_\R,u_\Sigma) \subseteq \R \times \Sigma$ are contained in a common compact region. The non-trivial part here is bounding $u_\R$ from below. To this end we will borrow an argument given by Albers-Fuchs-Merry in the proof of \cite[Thm.\ 3.9]{Albers_Fuchs_Merry_2015}.

In this subsection we additionally assume that the almost complex structure $J$ on $S \Sigma$ is of SFT-type (Definition \ref{def: SFT-type almost complex structure}). Below, we will use the notion of the \textit{Hofer energy} $E_{\mathrm{Hof}}(u)$ of a $(j,J)$-holomorphic map $u : Z \rightarrow S \Sigma \,$, whose domain is a Riemann surface $(Z,j)$ (possibly with boundary, disconnected). It is defined by
\[
E_{\mathrm{Hof}}(u) := \sup_\nu  \int_Z u^* \d (\nu \alpha) \in [0,\infty] \,\, ,
\]
where the supremum runs over all $\nu \in \Cinfty(\R,[0,1])$ with $\nu' \geq 0 \,$. Note that, since $J$ is of SFT-type and $u$ is $(j,J)$-holomorphic, the integrand is an everywhere non-negatively oriented 2-form, that is
\[
u^* \d (\nu \alpha) = u^*(\nu \, \d \alpha + \nu'(h) \, \d h \wedge \alpha) \geq 0 \quad \text{ on } Z \, .
\]

\begin{proposition}
\label{prop: cylinders in compact region}
Assume additionally that $J$ is of SFT-type. For every $C^1$-datum $\Dscr$ there exist constants $\delta^*_1 > 0$ and $h_0^- < 0$ with the following significance. For every $H \in \Hclassdelta{\delta^*_1} $ and every $(r,(u,\eta)) \in \moduli^H \,$, the cylinder $u \in \Cinfty(\R \times \bbS^1,S\Sigma)$ is confined to $[h_0^-,h_0] \times \Sigma\,$, that is
\[
\image(u) \subseteq [h_0^-,h_0] \times \Sigma \subseteq S \Sigma \,\, .
\]
\end{proposition}

\begin{proof}
   Establishing the upper bound is a standard argument.

   \begin{claim}
    For every Hamiltonian $H \in \Hcal^{h_0}$ and every $(r,(u,\eta)) \in \moduli^H$ the image of the cylinder $u$ is contained in $(-\infty,h_0] \times \Sigma \,$.
   \end{claim}
   \begin{proofClaim}
        Assume by contradiction that $V := u^{-1}((h_0,\infty) \times \Sigma)$ is non-empty. Notice that $V$ is precompact in $\R \times \bbS^1$ since $u$ has asymptotics in $\{0\} \times \Sigma \subseteq (-\infty ,h_0) \times \Sigma \,$. The restriction $u|_V$ is $J$-holomorphic by the gradient equation \eqref{eq: s-dependent gradient eq} and because the Hamiltonian vector fields $X_{H_0}$ and $X_H$ vanish on $(\R \backslash (-h_0,h_0)) \times \Sigma $ as $H_0,H \in \Hcal^{h_0}$. Thus, using that $J$ is of SFT-type, the $\R$-component $u_\R = h \circ u$ is subharmonic on $V$, see \cite[Prop.~3.4]{Beginners_Symplectic_Homology}. By the maximum principle, $u_\R$ must attain its global maximum on the boundary of $V$, an obvious contradiction.
   \end{proofClaim}

   \noindent Next we deal with the bound from below.
   \begin{claim}
       Let $H \in \Hcal^{h_0}$ and $(r,w) = (r,(u,\eta)) \in \moduli^H$ be arbitrary. Let $h_{\mathrm{reg}} \leq - h_0$ be a regular value of the $\R$-component $h \circ u$ for which $Z := u^{-1}((-\infty,h_{\mathrm{reg}}] \times \Sigma)$ is non-empty. Then the Hofer energy of $u|_Z$ is bounded by
       \[
       E_{\mathrm{Hof}}(u|_Z) \leq e^{h_0} E(w) \,\, .
       \]
   \end{claim}
   \noindent Notice that $Z$ is a compact Riemann surface with boundary $\partial Z = u^{-1}(\{h_{\mathrm{reg}}\} \times \Sigma) \,$. Moreover, $u|_Z$ is $J$-holomorphic and non-constant on each component of $Z$.
   \begin{proofClaim}
   Choose an increasing sequence $(h_n)_n \subseteq [h_{\mathrm{reg}}, -h_0]$ of regular values approaching $-h_0 \,$. We write $Z_n := u^{-1}((-\infty,h_n] \times \Sigma) \,$. 
      For arbitrary $\nu \in \Cinfty(\R,[0,1])$ with $\nu' \geq 0 \,$, using Stokes' theorem twice,  we estimate
      \begin{align*}
          &\phantom{\leq} \,\int_Z u^* \d (\nu \alpha) \leq \int_{Z_n} u^* \d (\nu \alpha) = \int_{\partial Z_n}  u^* (\nu \alpha) = \nu(h_n) \, e^{-h_n} \int_{Z_n} u^*\d (e^h \alpha) \\
          &\leq \nu(h_n) \, e^{-h_n} \int_{-\infty}^{\infty} \int_0^1 \omega(\partial_s u , \, J(u) \partial_s u) \, \d t \, \d s \leq   e^{-h_n} E(w) \,\, .
      \end{align*}
    Thus $E_{\mathrm{Hof}}(u|_Z) \leq e^{-h_n} E(w)$ and the claim follows by taking the limit $n \rightarrow \infty \,$.
   \end{proofClaim}
   \noindent Now let $\Dscr$ be a given $C^1$-datum. Choose constants $T^* = T^*(\Dscr) > 0$ and $\delta^* = \delta^*(\Dscr)> 0$ as in Proposition \ref{prop: uniform Lagrange multiplier bound} and set
   \[
   \delta_1^* := \mathrm{min}\left\{\delta^* \comma \frac{e^{-h_0} \tau(\alpha)}{4 T^*}\right\} > 0 \,\, .
   \]
   Assume by contradiction that, for his $\delta_1^* \,$, there does not exist a lower bound $h_0^- < 0$ with the significance stated in the proposition. Then there exist sequences $H_n \in \Hclassdelta{\delta_1^*}$ and $(r_n,w_n) = (r_n,(u_n,\eta_n)) \in \moduli^{H_n}$ so that
   \[
   \inf_{\R \times \bbS^1} h \circ u_n \searrow -\infty \,\,\,\text{ as } n \rightarrow \infty \,\, .
   \]
    Choose a common regular value $h_{\mathrm{reg}} < - h_0$ of the $h \circ u_n$ (possible by Sard's theorem). Then, for almost every $n \in \N \,$,
    \[
    Z_n := u_n^{-1}((-\infty, h_{\mathrm{reg}} ] \times \Sigma)
    \]
    is a non-empty, compact Riemann surface with boundary of genus zero.\footnote{
    Given an orientable compact surface $S$ of genus $g$ and a compact subsurface $S' \subseteq S$ of genus $g' \,$, it holds $0 \leq g' \leq g \,$, see \cite[Thm.\ 4.3]{Axon_EndSumSurfaces}. The compact cylinder $\Rbar \times \bbS^1$ has genus zero.
    } 
    By the claim and Corollary \ref{cor: uniform Lagrange multiplier bound}, we can bound the Hofer energy by
    \[
    E_{\mathrm{Hof}}(u_n|_{Z_n}) \leq e^{h_0} E(w_n) \leq 2 \,\delta_1^* \, T^* \, e^{h_0} \leq \tfrac{1}{2} \tau(\alpha) \,\, .
    \]
    Due to this uniform bound on the Hofer energy and because $\inf_{Z_n} h \circ u_n$ tends to minus infinity, we can now apply \cite[Thm.\ 5.3]{Albers_Fuchs_Merry_2015} to the sequence of $J$-holomorphic curves $u_n|_{Z_n}$ (shifted by $h_{\mathrm{reg}}$ in the $\R$-component) and obtain existence of a periodic Reeb orbit on $(\Sigma,\alpha)$ of period at most $\tfrac{1}{2} \tau(\alpha) \,$. This is an obvious contradiction, $\tau(\alpha)$ being the minimal Reeb period.
\end{proof}

\subsection{Gromov compactness}
\label{subsec: Preparation - Gromov Compactness}

Having found uniform bounds on the energy, the Lagrange multiplier and the cylinder component, we obtain Gromov compactness for sequences in the moduli space $\moduli^H_{[0,R]} \,$. 

We record this in the next proposition in a slightly more general fashion. In fact, the proposition holds true for arbitrary exact symplectic manifolds $(W,\lambda)$ with fixed $\omega = \d \lambda$-compatible almost complex structure $J$.

\begin{convention}
\label{convention: CinftylocConvergence on loopspace times R}
Given a sequence 
\[
(w_n)_n = (u_n,\eta_n)_n \in \Cinfty(\R,\loopspace \times \R) = \Cinfty(\R \times \bbS^1,W) \times \Cinfty(\R,\R) \,\, ,
\]
we say \textit{the sequence $(w_n)_n$ converges in} $\Cinftyloc(\R,\loopspace \times \R)$ if the sequence $(u_n,\eta_n)_n$ converges in $\Cinftyloc(\R \times \bbS^1,W) \times \Cinftyloc(\R,\R) \,$.  
\end{convention}

\begin{proposition}[Gromov compactness]
\label{prop: Gromov Compactness}
    Let $H \in \Cinfty( [0,R] \times \R \times W,\R)$ be a Hamiltonian, which we also write in the form $H_{r,s}(x) :=  H(r,s,x) \,$, and suppose there exists $s_0 > 0$ so that $H$ is $s$-independent for $|s| \geq s_0 \,$, that is
    \[
    H_{r,s} = H_{r,\, \pm s_0} \qquad \text{ for } \pm s \geq s_0 \comma r \in [0,R] \,\, .
    \]
    Let $(r_n,w_n)_n = (r_n, (u_n,\eta_n))_n \subseteq [0,R] \times \Cinfty(\R,\loopspace \times \R)$ be a sequence of gradient flow lines, that is
    \begin{align*}
        \partial_s w_n(s) = \nabla \Rabinowitz{H_{r_n,s}}(w_n(s)) \qquad \Forall s \in \R \comma n \in \N \, ,
    \end{align*}
    with
    \begin{enumerate}[label=\arabic*)]
        \item\label{it: Gromov Compactness - uniform energy bound} a uniform energy bound $E(w_n) = \int_{-\infty}^{\infty} |\partial_s w_n|_{g_J}^2 \, \d s \leq E_0 \comma n \in \N \,$,
         \item\label{it: Gromov Compactness - Lagrange multiplier bound} a uniform Lagrange multiplier bound $\lVert \eta_n \rVert_\infty \leq T^* \comma n \in \N \,$,
        \item\label{it: Gromov Compactness - compact region} there exists a compact region $K \subseteq W$ containing $\mathrm{im}(u_n)$ for every $n\in \N \,$.
    \end{enumerate}
Then a subsequence of $(r_n,w_n)_n$ converges in $[0,R] \times \Cinftyloc(\R,\loopspace \times \R) \,$.
\end{proposition}

\begin{remark}
\label{rmk: Gromov compactness - limit gradient flow line}
The limit $(r,w)$ of the convergent subsequence of $(r_n,w_n)_n$ necessarily is again a gradient flow line, that is, it satisfies
    \[
    \partial_s w (s) = \nabla \Rabinowitz{H_{r,s}}(w(s)) \qquad \Forall s \in \R \,\, .
    \]
Moreover, by Fatou's lemma, its energy is bounded by $E(w) \leq E_0 \,$.
\end{remark}

\noindent Similar versions of the proposition can be found in \cite[Thm.~3.1]{Cieliebak_Frauenfelder} and in \cite[Thm.~2.9]{Albers_Frauenfelder}. Since it is well-known how to prove such a result, we just point out the key steps.

\begin{proof}
    It suffices to show that the cylinders $u_n$ have uniformly bounded derivatives. If this is accomplished, then the \ArzelaAscoli theorem allows us to extract a $C^0_{\mathrm{loc}}$-convergent subsequence of $(u_n,\eta_n)_n \,$. Bootstrapping both components of the gradient equation simultaneously and using elliptic regularity shows that a further subsequence converges in fact in $\Cinftyloc\,$.

    Now assume by contradiction that the $u_n$ do not have uniformly bounded derivatives. Then, using the assumptions \ref{it: Gromov Compactness - uniform energy bound}\,-\,\ref{it: Gromov Compactness - compact region}, one can construct a non-constant $J$-holomorphic plane of finite energy.\footnote{
    In more detail, as in \cite[Ch.~6.6]{Audin-Damian}, suitable reparametrizations $\widetilde{u}_n = u_n(s_n + \frac{\cdot}{R_n}, t_n + \frac{\cdot}{R_n})$ have uniformly bounded derivatives on the ball $B_{\eps_n R_n}(s_n,t_n) \subseteq \C \,$, where $\eps_n \rightarrow 0$ and $\eps_n R_n \rightarrow \infty \,$. A subsequence of $(\widetilde{u}_n)_n$ will converge in $C^1_{\mathrm{loc}}(\C,W)$ to a non-constant $J$-holomorphic plane, which is smooth by elliptic regularity and has energy bounded by $E_0 \,$.
    (We only get $C^1_{\mathrm{loc}}$-convergence because the Lagrange multiplier component in the gradient flow equation, $\partial_s \eta_n = -\int_0^1 H_{r_n, s}(u_n) \, \d t \,$, is not local, so that elliptic regularity only carries us up to $W^{2,p}_{\mathrm{loc}}$-convergence.)
    }
    Removing the singularity at infinity, one obtains a non-constant $J$-holomorphic sphere. However, such do not exist on an exact symplectic manifold.
\end{proof}

\begin{remark}
\label{rmk: Gromov compactness - half-open interval}
The following version of the proposition holds if the gradient flow lines are only defined on an infinite half-open interval, say $(r_n,w_n)_n \subseteq [0,R] \times \Cinfty(\R_{\geq a}, \loopspace \times \R) \,$. Suppose assumptions \ref{it: Gromov Compactness - uniform energy bound}\,-\,\ref{it: Gromov Compactness - compact region} in Proposition \ref{prop: Gromov Compactness} hold, where the energy is defined by $E(w_n) := \int_{a}^{\infty} |\partial_s w_n(s)|^2_{g_J} \, \d s \,$. Then 
\begin{align*}
    \text{a subsequence of } (r_n,w_n)_n \text{ converges in } [0,R] \times \Cinftyloc(\R_{>a} ,\loopspace \times \R) \, .
\end{align*}
Note that we only obtain $\Cinftyloc$-convergence on $\R_{> a}$ and not on $\R_{\geq a } $ because in the proof of the uniform bound for the derivatives of the cylinders $u_n$ we need to stay away from the boundary of the half-cylinder $\R_{\geq a} \times \bbS^1$ to obtain a $J$-holomorphic plane $u : \C \rightarrow W$ in the limit.
\end{remark}

\noindent We next state a version of the above proposition where there are no parameters $(r,s)$ involved anymore but the gradient flow lines are only defined on bounded intervals of increasing size. 

\begin{lemma}
\label{lem: Gromov compactness - sequence on (-R_n, R_n)}
Let $H \in \Cinfty(W,\R)$ be a Hamiltonian and $(w_n)_n  \subseteq \Cinfty(\R,\loopspace \times \R)$ be a sequence satisfying \ref{it: Gromov Compactness - uniform energy bound}\,-\,\ref{it: Gromov Compactness - compact region} in Proposition \ref{prop: Gromov Compactness}. Suppose moreover that for some sequence $(R_n)_n \subseteq \R_{> 0}$ tending to infinity it holds
\[
\partial_s w_n (s) = \nabla \Rabinowitz{H}(w_n(s)) \qquad \Forall s \in (-R_n,R_n) \comma n \in \N \,\, .
\]
Then a subsequence of $(w_n)_n$ converges in $\Cinftyloc(\R,\loopspace \times \R)$ to a gradient flow line $w$ of $\RabinowitzH $ having energy bounded by $E(w) \leq \sup_n E(w_n) < \infty\,$.
\end{lemma}

\begin{proof}
  For $N \in \N$ we abridge $I_N := [-N,N]$. We iteratively extract subsequences: If $(w_n^{(N)})_n$ is a subsequence converging in $\Cinfty(I_N, \loopspace \times \R)$, use Gromov compactness to extract a subsequence $(w_n^{(N+1)})_n$ of $(w_n^{(N)})_n$ converging in $\Cinfty(I_{N+1}, \loopspace \times \R)$. On $I_N$ the limits of $(w_n^{(N+1)})_n$ and of $(w_n^{(N)})_n $ agree. Finally, define $w$ by $w|_{I_N} = \lim_{n \rightarrow \infty} w_n^{(N)}|_{I_N} $ and take the diagonal subsequence $(w_N^{(N)})_N \,$.
\end{proof}

\noindent By a well-known argument, Gromov compactness also yields existence of critical points at the asymptotic ends of a finite-energy gradient flow line.

\begin{corollary}
\label{cor: Gromov compactness - single gradient flow line crit pts at ends}
Let $H \in \Cinfty(W,\R)$ be a Hamiltonian and $w = (u,\eta) \in \Cinfty(\R,\loopspace \times \R)$ be a gradient flow line of $\RabinowitzH$ having finite energy, with $\lVert \eta \rVert_\infty < \infty$ and so that the image of $u$ is precompact in $W$. Then for every sequence $(s_n)_n \subseteq \R$ tending to $+\infty$ there exists a subsequence $(s_{n_k})_k$ and critical points $(v^-,\tau^-) \comma (v^+,\tau^+) \in \Crit(\RabinowitzH)$ so that
\[
w(\pm s_{n_k}) \overset{k \rightarrow \infty}{\longrightarrow} (v^\pm,\tau^\pm) \quad \text{ in } \Cinfty(\bbS^1,W) \times \R \,\, .
\]
\end{corollary}

\begin{proof}
    Finite energy implies that the increasing function $\RabinowitzH \circ w$ has finite limits $(\RabinowitzH \circ w)(\pm \infty) \,$.
    Consider the sequence $(w(\,\cdot\, + s_n))_n$ of gradient flow lines of $\RabinowitzH \,$. Note that $E(w(\,\cdot\, + s_n)) = E(w) < \infty \,$. By Gromov compactness, a subsequence $(w(\,\cdot\, + s_{n_k}))_k$ converges in $\Cinftyloc(\R,\loopspace \times \R)$ to a gradient flow line $w^+\,$. Notice that
    \[
    \RabinowitzH(w^+(s)) = \lim_{k \rightarrow \infty} \RabinowitzH(w(s+s_{n_k})) = (\RabinowitzH \circ w)(+\infty) \qquad \Forall s \in \R \, ,
    \]
    so that $\RabinowitzH \circ w^+$ is constant. Consequently $w^+$ is $s$-independent and sitting at a critical point $(v^+,\tau^+) := w^+(0) \in \Crit(\RabinowitzH) \,$. Moreover, clearly $(w(s_{n_k}))_k$ converges to $w^+(0) = (v^+,\tau^+) \,$. To also obtain a critical point at minus infinity, we extract a further subsequence, considering this time $(w(\,\cdot\, - s_{n_k}))_k \,$, and repeat the above arguments.
\end{proof}

\subsection{Spectrum of the Hessian}
\label{subsec: Preparation - Spectrum Hessian}

We denote by $\nabla^2 \RabinowitzHZero(v,\tau)$ the covariant Hessian of $\RabinowitzHZero$ at $(v,\tau) \in \loopspace \times \R $ with respect to the metric $\omega(\,\cdot\, , J \, \cdot \,)$ on $W$. This is a bounded linear operator
\[
\nabla^2 \RabinowitzHZero (v,\tau) : \, H^1(\bbS^1,v^*TW) \oplus \R \longrightarrow L^2(\bbS^1,v^*TW) \oplus \R \,\, ,
\]
which we also consider as selfadjoint unbounded operator on $L^2(\bbS^1,v^*TW) \oplus \R $ with compact resolvent. From the theory of selfadjoint unbounded operators, see for example \cite[Thm.~6.3.13]{Functional_Analysis_Salamon}, we know that the spectrum of $\nabla^2 \RabinowitzHZero(v,\tau)$ is a discrete, closed, countably infinite subset of $\R$ and that it consists of eigenvalues only, each with finite multiplicity.
Because $\RabinowitzHZero$ is Morse-Bott (see Lemma \ref{lem: Crit(RabinowitzHZero)}), $0$ is an eigenvalue of $\nabla^2\RabinowitzHZero(v,\tau(\alpha))$ for $(v,\tau(\alpha)) \in \CCrit$ and the corresponding eigenspace is precisely the tangent space $T_{(v,\tau(\alpha))}\CCrit \,$. In particular, the multiplicity of the eigenvalue $0$ is the same for every critical point $(v,\tau(\alpha)) \in \CCrit \,$, namely $\dim(\CCrit) = \dim(\Sigma) \,$. Therefore
\begin{equation}
\label{eq: min spec Hessian > 0}
0 < \min\set{|\mu|}{0 \not= \mu \text{ is eigenvalue of } \nabla^2\RabinowitzHZero(v,\tau(\alpha)) \comma (v,\tau(\alpha)) \in \CCrit }  \, ,
\end{equation}
owing to the stability of the spectrum under small perturbations by symmetric bounded operators, see \cite[V.\S4.3]{Kato}, and the compactness of $\CCrit \,$.

\subsection{Exponential decay}
\label{subsec: Preparation - Exponential Decay}

Recall from Lemma \ref{lem: Crit(RabinowitzHZero)} that $\RabinowitzHZero$ is a Morse-Bott functional. It is well-known that this implies exponential decay of its gradient flow lines converging asymptotically to points in $\CCrit \subseteq \Crit(\RabinowitzHZero)$. To later prove compactness of the relevant moduli spaces, we need the following stronger version of exponential decay.

\begin{proposition}[Exponential decay]
\label{prop: Uniform Exponential Decay}
There exists a constant $\kappa > 0$ so that for every $(v,\tau(\alpha)) \in \CCrit$ there exists a $(\Cinfty \times \R)$-open neighborhood $\mathscr{U} \subseteq \loopspace \times \R$ of $(v,\tau(\alpha))$ and a $(C^\infty \times \R)$-continuous function $\Xi : \mathscr{U} \rightarrow \R_{\geq 0}$ vanishing on $\mathscr{U} \cap \CCrit$ with the following significance: For every $s_0 \in \R$ and every gradient flow line $w = (u,\eta) \in \Cinfty(\R_{\geq s_0} ,\mathscr{U})$ of $\RabinowitzHZero$ with existing limit (in the $(\Cinfty \times \R)$-topology) $w(\infty) \in \CCrit \cap \mathscr{U}$ it holds
\begin{align}
    |\partial_s u(s,t)| + |\partial_s \eta(s)| \leq \Xi(w(s_0)) \, e^{-\kappa(s-s_0)} \qquad \Forall s \geq s_0 \comma t \in \bbS^1 \,\, . \label{eq: Uniform Exponential Decay - partial_s}
\end{align}
\end{proposition}

\noindent This is proven at length for the classical action functional by the author in \cite[Thm.~C, Cor.~4.4]{Compact_Moduli_Spaces_Gradient_Floer_Lines}. The proof in the Rabinowitz-Floer case is a straightforward adaption, which we only sketch briefly.

\begin{proof}[Proof sketch]
At fixed $(v_0,\tau_0) \in \CCrit \,$, choose tubular coordinates $U \cong \bbS^1 \times \R^{2N-1}$ around the loop $v_0$ as in the subsequent Lemma \ref{lem: Analytic Set-Up - tubular nbhd}.
We show the analog of \cite[Prop.~5.6]{Compact_Moduli_Spaces_Gradient_Floer_Lines}.
Replace the $\loopspace$-dependent family of operators $\mathsf{A}_v$ in \cite[eq.~(5.4)]{Compact_Moduli_Spaces_Gradient_Floer_Lines}, where we omit the shift operator, by the $(\loopspace \times \R)$-dependent family of operators 
  \[
  \mathsf{A}_{(v,\tau)} :  H^k := H^k(\bbS^1,\R^{2N}) \oplus \R \longrightarrow H^{k-1} := H^{k-1}(\bbS^1,\R^{2N}) \oplus \R
  \]
  displayed in \cite[eq.~(16)]{FauckThesis} but with minus sign. The operator $\mathsf{A}_{(v_0,\tau_0)}$ agrees with the covariant Hessian $\nabla^2 \RabinowitzHZero(v_0,\tau_0) \,$, see \cite[Lemma~30]{FauckThesis}. This shows (Fact 1) in \cite[\S5.1]{Compact_Moduli_Spaces_Gradient_Floer_Lines}. For (Fact 2), observe that the kernel of $\mathsf{A}_{(v,\tau)}$ contains every $(X,0) \in H^k= H^k(\bbS^1,\R^{2N}) \oplus \R $ for which $X$ is a constant loop with vanishing last component, \ie $X_{2N} = 0 \,$. Moreover, every element of the kernel of $\mathsf{A}_{(v_0,\tau_0)} = \nabla^2 \RabinowitzHZero(v_0,\tau_0)$ is of this form by the Morse-Bott property of $\RabinowitzHZero \,$, see Lemma \ref{lem: Crit(RabinowitzHZero)} or \cite[Lemma~20]{FauckThesis}. (Fact 3) holds in the present setting as well due to \cite[Lemma~22]{FauckThesis}, in the proof of which also $(L^2(\bbS^1,\R^{2N}) \oplus \R)$-orthogonality is shown. With these facts established, Part I in the proof of \cite[Prop.~5.6]{Compact_Moduli_Spaces_Gradient_Floer_Lines} goes through verbatim.
  For Part II one needs to rearrange only the cylinder component of the Rabinowitz-Floer equation \eqref{eq: Rabinowitz-Floer eq}, as the image of $(X,T) \in H^k$ under the orthogonal projection $\mathsf{P} : H^k \rightarrow \ker(\mathsf{A}_{(v_0,\tau_0)})$ has Lagrange multiplier component zero by (Fact 2). This finishes the proof of the analog of \cite[Prop.~5.6]{Compact_Moduli_Spaces_Gradient_Floer_Lines}. 
  
  From this one infers the proposition as in the proof of \cite[Thm.~C]{Compact_Moduli_Spaces_Gradient_Floer_Lines}.
  Lastly, since $\CCrit$ is compact, we can choose a constant $\kappa $ independent of the base point $(v_0,\tau_0) \in \CCrit \,$. 
\end{proof}

\begin{remark}
\label{rmk: Uniform Exponential Decay}
Regarding the above proposition, we record the following.
\begin{enumerate}[label=(\roman*)]
    \item\label{it: Uniform Exponential Decay - Xi bounded} Due to the continuity of $\Xi \,$, choosing $\mathscr{U}$ smaller if necessary, we can always assume that $\Xi $ is a bounded function on $\mathscr{U} \,$.
    \item\label{it: Uniform Exponential Decay - integration} For $w = (u,\eta) \in \Cinfty(\R_{\geq s_0}, \mathscr{U})$ as in Proposition \ref{prop: Uniform Exponential Decay} we can integrate the exponential decay estimates and obtain
    \begin{align}
       d_{W} (u(s_1,t), u(\infty,t)) &\leq \int_{s_1}^{\infty} |\partial_s u(s,t)| \, \d s 
        \leq \Xi(w(s_0)) \int_{s_1}^{\infty} e^{-\kappa(s-s_0)} \, \d s \notag\\
        &\leq \tfrac{1}{\kappa} \Xi(w(s_0)) \, e^{-\kappa(s_1 - s_0)}\qquad \Forall s_1 \geq s_0 \comma t \in \bbS^1 \,\, , \label{eq: Uniform Exponential Decay - d(u, u_infty)}\\
       |\eta(s_1) - \tau(\alpha)| &= |\eta(s_1) - \eta(\infty) | \leq \int_{s_1}^{\infty} |\partial_s \eta(s)| \, \d s \notag\\
       &\leq \tfrac{1}{\kappa} \Xi(w(s_0)) \, e^{-\kappa(s_1 - s_0)} \qquad \Forall s_1 \geq s_0 \,\, . \label{eq: Uniform Exponential Decay - |eta - eta(infty)|}
    \end{align}
    So we have an analogous exponential decay assertion for the pointwise distances of $u$ respectively $\eta$ to its limits at $s = \infty \,$.
    \item\label{it: Uniform Exponential Decay - negative ends} There is an analogous version of the proposition for gradient flow lines $w \in \Cinfty(\R_{\leq s_0}, \mathscr{U})$ and converging to $w(-\infty) \in \CCrit \cap \mathscr{U} \,$.
    \item\label{it: Uniform Exponential Decay - delta < min spec Hessian} Every gradient flow line $w$ of $\RabinowitzHZero$ converging asymptotically to a point in $w(\infty ) \in \CCrit$ exhibits exponential decay for arbitrary exponential coefficient $\kappa$ smaller than the absolute value of a non-zero eigenvalue of the covariant Hessian $\nabla^2 \RabinowitzHZero(w(\infty)) \,$. It is therefore likely that every
    \begin{align*}
    \hspace*{12mm}\kappa < \min \set{|\mu|}{0 \not= \mu \text{ eigenvalue of } \nabla^2 \RabinowitzHZero(v,\tau(\alpha)) \comma (v,\tau(\alpha)) \in \CCrit }
    \end{align*}
    would work in Proposition \ref{prop: Uniform Exponential Decay} but this is not proven in \cite{Compact_Moduli_Spaces_Gradient_Floer_Lines}.
    \item\label{it: Uniform Exponential Decay - Fauck} There is a similar exponential decay result for Rabinowitz-Floer gradient flow lines in \cite[Prop.~32]{FauckThesis}. However, I believe its proof to be flawed.\footnote{
    In the proof of \cite[Prop.~32]{FauckThesis}, the family of operators $A(s)$ is derived in $t$-direction, which seems absurd.
    }
\end{enumerate}
\end{remark}

\section{Proof of Theorem \ref{mthm: Rabinowitz gradient flow line}}
\label{sec: Proof MThm}

\noindent In this section, we prove Theorem \ref{mthm: Rabinowitz gradient flow line}, based on the preparatory steps in the previous section. We state the key result (Proposition \ref{prop: moduli^H_[0,R](z) non-empty}) from which we conclude the theorem. Its rather involved proof is postponed to Section \ref{sec: Non-Empty Moduli Spaces}.
\\

\noindent Let $(\Sigma,\alpha)$ be a Zoll contact manifold with minimal Reeb period $\tau(\alpha) \,$. Let $J$ be an SFT-type almost complex structure on the symplectization $S \Sigma\,$. Suppose we are given $h_0 > 0 \comma \eps > 0 \,$, a defining Hamiltonian $H_0 \in \Hcal^{h_0} $ and a $C^1$-datum $\Dscr \,$, which we keep fixed throughout. 
Since it obviously suffices to show Theorem \ref{mthm: Rabinowitz gradient flow line} for sufficiently small $\eps$ only, without loss of generality we may assume
\begin{equation}
\label{eq: Proof MThm - eps < tau(alpha)}
    0 < \eps < \tau(\alpha) \,\, .
\end{equation}

\subsection{Choice of $\delta$}
\label{subsec: Proof MThm - delta choice}

For the given $C^1$-datum $\Dscr$ choose constants $T^* = T^*(\Dscr) >0$ and $\delta^* = \delta^*(\Dscr) > 0$ as in Proposition \ref{prop: uniform Lagrange multiplier bound}. Moreover, choose constants $\delta_1^* = \delta_1^*(\Dscr) > 0$ and $h_0^- = h_0^-(\Dscr) < 0$ as in Proposition \ref{prop: cylinders in compact region}. Now define $\delta > 0$ by
\[
\delta := \min \left\{  \delta^* \comma \delta^*_1 \comma \frac{\eps}{2 T^*} \right\} \,\, .
\]
We claim that this $\delta$ has the property stated in Theorem \ref{mthm: Rabinowitz gradient flow line}. We will verify this in the following.

\subsection{$\delta$ is as desired}
\label{subsec: Proof MThm - delta has claimed property}

We are given a Hamiltonian $H \in \Hclassparam $ and a point $z \in \Sigma \,$. We have to show that there exist a gradient flow line $w=(u,\eta)$ of $\RabinowitzH \,$, a sequence $(s_n)_n$ of positive numbers tending to infinity and critical points $(v^\pm,\tau^\pm) \in \Crit(\RabinowitzH)$ having properties \ref{it: mthm Rabinowitz gradient flow line - tau not= 0}\,-\,\ref{it: mthm Rabinowitz gradient flow line - u(0,0)} in Theorem \ref{mthm: Rabinowitz gradient flow line}.

We have the following bounds for elements in the moduli space $\moduli^H \,$, defined in equation \eqref{eq: def moduli^H}.

\begin{lemma}
\label{lem: Proof MThm - bounds on energy, Lagrange multiplier, cylinder}
For every $(r,w) = (r,(u,\eta)) \in \moduli^H$ the following bounds hold:
\begin{enumerate}[label=\arabic*)]
    \item\label{it: Proof MThm - bound on Delta} $\sup_{\mathrm{im}(w)} |\RabinowitzH - \RabinowitzHZero| \leq \tfrac{\eps}{2}  \,$,
    \item\label{it: Proof MThm - bound on energy} $E(w) \leq \eps \,$,
    \item\label{it: Proof MThm - bound on action value} $\RabinowitzHrs(w(s)) , \, \RabinowitzHZero(w(s)) \in [\tau(\alpha) - \eps, \, \tau(\alpha) + \eps]$ for every $s \in \R \,$,
    \item\label{it: Proof MThm - bound on Lagrange multiplier} $\lVert \eta \rVert_\infty \leq T^* \,$,
    \item\label{it: Proof MThm - bound on cylinder} $\mathrm{im}(u) \subseteq [h_0^-, h_0] \times \Sigma \,$.
\end{enumerate}
\end{lemma}

\begin{proof}
    Points \ref{it: Proof MThm - bound on Delta}\,-\,\ref{it: Proof MThm - bound on action value} follow from Lemma \ref{lem: action estimate for w in moduli} and Corollary \ref{cor: uniform Lagrange multiplier bound}, using that $\delta \, T^* \leq \tfrac{\eps}{2} \,$.
    
    Points \ref{it: Proof MThm - bound on Lagrange multiplier} and \ref{it: Proof MThm - bound on cylinder} hold by choice of $T^* , \delta^*$ as in Proposition \ref{prop: uniform Lagrange multiplier bound} respectively $\delta_1^* , \, h_0^-$ as in Proposition \ref{prop: cylinders in compact region} and since 
    \[
    H \in \Hclassparam \subseteq \Hclassdelta{\delta^*}\cap \Hclassdelta{\delta^*_1} 
    \]
    because $\delta \leq \min\{\delta^*,\delta^*_1\} \,$; see also \eqref{eq: Hclassparam inclusion for varying parameters} for the inclusion.
\end{proof}

\noindent For each $R \geq 0$ we now define the moduli space
\begin{align*}
    \moduli^H_{R}(z) := \Bigset{(u,\eta) \in \moduliH{R}}{u(0,0) \in \R \times \{z\}} \,\, .
\end{align*}
An element of the moduli space is depicted schematically in Figure \ref{fig: moduli space moduliH_[0,R](z) element}.

The following proposition is the key step to prove Theorem \ref{mthm: Rabinowitz gradient flow line}. Its proof will be presented in detail in Section \ref{sec: Non-Empty Moduli Spaces}.

\begin{proposition}
\label{prop: moduli^H_[0,R](z) non-empty}
For every $R \geq 0$ the moduli space $\moduliH{R}(z)$ is non-empty.
\end{proposition}

\noindent Admitting the proposition, we can finalize the proof of Theorem \ref{mthm: Rabinowitz gradient flow line}.

\begin{proof}[Finishing the proof of Theorem \ref{mthm: Rabinowitz gradient flow line}]
    For each $n \in \N$ choose $w_n =(u_n,\eta_n) \in \moduliH{n}(z) \,$. Lying in the appropriate moduli space, $w_n$ is a gradient flow line of $\RabinowitzH$ on the interval $[-n,n] \,$. Due to the uniform bounds in Lemma \ref{lem: Proof MThm - bounds on energy, Lagrange multiplier, cylinder}, we can apply Lemma \ref{lem: Gromov compactness - sequence on (-R_n, R_n)} and extract a subsequence $(w_{n_k})_k$ converging in $\Cinftyloc(\R,\loopspace \times \R)$ to a gradient flow line $w = (u,\eta)$ of $\RabinowitzH $ having energy bounded by $E(w) \leq \eps \,$. By $\Cinftyloc$-convergence and the bounds in Lemma \ref{lem: Proof MThm - bounds on energy, Lagrange multiplier, cylinder} again, the Lagrange multiplier of the limit is bounded by $\lVert \eta \rVert_\infty \leq T^*$ and 
    \begin{equation}
    \label{eq: Proof Mthm - RabinowitzH(limit w) close to tau(alpha)}
    \RabinowitzH (w(s)) = \lim_{k \rightarrow \infty }\Rabinowitz{H}_{n_k ,s}(w_{n_k}(s)) \in [\tau(\alpha) -\eps, \, \tau(\alpha)+\eps ] \qquad \Forall s \in \R \,\, .
    \end{equation}
    Similarly $\RabinowitzHZero(w(s))$ has distance from $\tau(\alpha)$ at most $\eps $ for every $s \in \R \,$.
    The image of the cylinder $u$ is contained in $[h_0^-,h_0] \times \Sigma \,$, in particular is precompact, and $u(0,0) = \lim_{k \rightarrow \infty} u_{n_k}(0,0) \in \R \times \{z\} \,$. By Corollary \ref{cor: Gromov compactness - single gradient flow line crit pts at ends}, there exist a sequence $(s_n)_{n \in \N} $ of positive numbers tending to infinity and critical points $(v^-,\tau^-) , \, (v^+,\tau^+) \in \Crit(\RabinowitzH)$ so that
    \[
    w(\pm s_n) \overset{n \rightarrow \infty}{\longrightarrow} (v^\pm,\tau^\pm) \qquad \text{in } \Cinfty(\bbS^1,S\Sigma) \times \R \,\, .
    \]
    Lastly, we show that $\tau^-,\tau^+ \not= 0 \,$. If $\tau^\pm = 0 \,$, then $v^\pm$ must be a constant loop by the critical point equation \eqref{eq: Crit(Rabinowitz)} and hence
    \[
    \RabinowitzH(v^\pm,\tau^\pm) = \int_{0}^1 \lambda_{v^\pm}(\partial_t v^\pm) \, \d t - \tau^\pm \int_0^1 H(v^\pm) \, \d t = 0 \,\, ,
    \]
    contradicting \eqref{eq: Proof Mthm - RabinowitzH(limit w) close to tau(alpha)} since $\tau(\alpha)- \eps > 0 $ by \eqref{eq: Proof MThm - eps < tau(alpha)}. Thus $\tau^\pm \not= 0 \,$. This shows that, with the previous choices, properties \ref{it: mthm Rabinowitz gradient flow line - tau not= 0}\,-\,\ref{it: mthm Rabinowitz gradient flow line - u(0,0)} in Theorem \ref{mthm: Rabinowitz gradient flow line} hold.
\end{proof}

\noindent Modulo the proof of Proposition \ref{prop: moduli^H_[0,R](z) non-empty}, the proof of Theorem \ref{mthm: Rabinowitz gradient flow line} is now complete.

\begin{figure}
    \centering
    \includegraphics[width=0.9\linewidth]{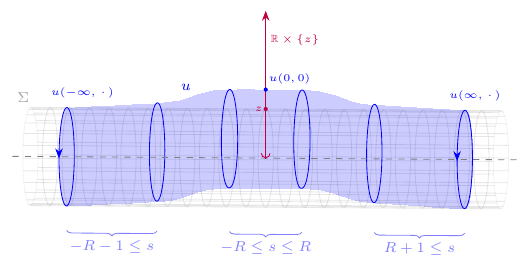}
    \caption{Cylinder $u$ of $w=(u,\eta) \in \moduliH{R}(z)$. $w$ is a gradient flow line of $\RabinowitzHZero$ for $|s| \geq R+1$ and a gradient flow line of $\RabinowitzH$ for $|s| \leq R\,$. The asymptotic orbits $u(\pm \infty,\,\cdot\,)$ are $1$-periodic integral curves of $\tau(\alpha) R_\alpha$ on $\Sigma \,$. 
    The symplectization $\R \times \Sigma$ is schematically represented by Euclidean 3-space minus the dashed line.}
    \label{fig: moduli space moduliH_[0,R](z) element}
\end{figure}

\section{Proof of Proposition \ref{prop: moduli^H_[0,R](z) non-empty}}
\label{sec: Non-Empty Moduli Spaces}

\noindent As in the previous section, we consider a Zoll contact manifold $(\Sigma,\alpha)$ with
\[
2N-1 := \dim(\Sigma) \geq 3 \,\, ,
\]
an almost complex structure $J$ on $S \Sigma$ of SFT-type, a constant $h_0 > 0$, a defining Hamiltonian $H_0 \in \Hcal^{h_0} \,$, a $C^1$-datum $\Dscr$ and $\eps \in (0,\tau(\alpha)) \,$. We fix the constants $\delta \comma T^* >0$ and $h_0^- < 0 \,$, which we have specified in Subsection \ref{subsec: Proof MThm - delta choice}. For this section we also fix $H \in \Hcal^{h_0}(H_0, \Dscr, \delta) \,$, for which in particular the bounds in Lemma \ref{lem: Proof MThm - bounds on energy, Lagrange multiplier, cylinder} hold, and a point $z \in \Sigma \,$. We aim to show Proposition \ref{prop: moduli^H_[0,R](z) non-empty}.

\subsection{Outline of the argument}
\label{subsec: Non-Empty Moduli Space - Outline}

We begin by observing that the moduli space $\moduliH{R}(z)$ for $R=0$ is a single point.

\begin{lemma}
\label{lem: moduli_0^H}
For $w = (u,\eta) \in \Cinfty(\R,\loopspace \times \R)$ it holds $w \in \moduliH{0}$ iff $w$ is constant and $w(s) = w(\pm \infty) \in \CCrit$ for any $s \in \R$. Consequently there is a natural identification $\moduliH{0} \cong \CCrit \cong \Sigma$ and, under this identification, $\moduliH{0}(z) \cong \{z\} \,$. 
\end{lemma}

\begin{proof}
    Since $\beta_0 \equiv 0$ by \ref{it: beta_r choice - beta_0}, an element $w \in \moduliH{0}$ is just a gradient flow line of $\RabinowitzHZero \,$, which starts and ends in $\CCrit \subseteq \Crit(\RabinowitzHZero) \,$. Since $\RabinowitzHZero$ is constant on $\CCrit$ with common value $\tau(\alpha)\,$, the gradient flow line $w$ has zero energy and is constant.

    The bijection $\moduliH{0} \overset{\cong}{\rightarrow} \CCrit$ is given by sending $w \in \moduliH{0}$ to $w(s)$ for any $s \in \R$ and the bijection $\CCrit \overset{\cong}{\rightarrow} \Sigma \comma (v,\tau(\alpha)) \mapsto v(0) \,$, is the one from \eqref{eq: CCrit diffeomorphism Sigma}.
\end{proof}

\noindent Let us now fix an arbitrary $R > 0 \,$. We will define a Banach bundle $\Ecal \rightarrow \Bcal$ and a section $\Fsection : [0,R] \times \Bcal \rightarrow \EcalHat$ of the pullback bundle $\EcalHat$ so that the zero set of $\Fsection_r = \Fsection(r,\,\cdot\,) : \Bcal \rightarrow \Ecal$ is precisely $\moduliH{r}$ for every $0 \leq r \leq R \,$. The section $\Fsection$ is a Fredholm section and its zero set $\moduliH{[0,R]}$ is a compact subset of $[0,R] \times \Bcal \,$. Moreover, $\Fsection_0$ is transverse to the zero section. The evaluation
\[
e : [0,R] \times \Bcal \rightarrow S \Sigma \comma \quad e(r,(u,\eta)) :=e_r(u,\eta) :=  u(0,0)
\]
restricted to $\{0\} \times \moduliH{0}$ is an embedding onto $\{0\} \times \Sigma$ and agrees with the identification from Lemma \ref{lem: moduli_0^H}. In particular $e_0|_{\moduliH{0}}$ is transverse to $\R \times \{z\}$ in $S \Sigma \,$.
Adding a suitable generic section $\sigma : [0,R] \times \Bcal \rightarrow \EcalHat $ with $\sigma(0,\,\cdot\,)= 0 \,$, the perturbed section $\Fsection + \sigma$ is transverse to the zero section and its zero set is a compact manifold $\widetilde{\moduli}_{[0,R]}^H$ with boundary $\moduliH{0} \sqcup \widetilde{\moduli}_R^H \,$. Perturbing also the evaluation map, we obtain a new smooth map $\widetilde{e} : \widetilde{\moduli}_{[0,R]}^H \rightarrow S \Sigma \,$, which is transverse to $\R \times \{z\}$ and agrees with $ e_0$ on $\moduliH{0} \,$, where transversality is already achieved. Thus $\widetilde{e}^{\,-1}(\R \times \{z\})$ is a compact smooth cobordism between $(e_0|_{\moduliH{0}})^{-1}(\R \times \{z\}) = \moduliH{0}(z)  \cong \{z\}$ and $ \widetilde{e}_R^{\,-1}(\R \times \{z\}) \,$, hence the latter is non-empty. If the perturbations are chosen sufficiently small, then this implies that also the unperturbed moduli space $(e_R|_{\moduliH{R}})^{-1}(\R \times \{z\}) = \moduliH{R}(z)$ is non-empty.

The previous considerations regarding such abstract perturbations have been carried out in detail for a general class of Banach bundles in \cite{Abstract_Perturbations}. For this reason, instead of repeating the above arguments, we content ourselves with invoking the abstract perturbation result \cite[Thm.~D]{Abstract_Perturbations} and obtain the same conclusion.
For convenience of the reader, we now repeat this theorem and the relevant definitions.

\begin{definition}
\label{def: Non-Empty Moduli Spaces - admitting smooth bump fcts}
A smooth Banach manifold $\Bcal$ \textit{admits smooth bump functions} if for every $x \in \Bcal$ and every open neighborhood $\Ucal(x) \subseteq \Bcal$ of $x$ there exists $\psi \in \Cinfty(\Bcal,[0,1])$ with $\psi(x) = 1$ and $\supp(\psi) \subseteq \Ucal(x) \,$.
\end{definition}

\begin{definition}
\label{def: Non-Empty Moduli Spaces - Fredholm section}
Let $\Fsection : \Bcal \rightarrow \Ecal$ be a smooth section of a Banach bundle $\Ecal \,$. For a zero $x \in \Fsection^{\,-1}(0) \,$, we denote by $\Dvert \Fsection(x) : T_x \Bcal \rightarrow \Ecal_x$ its vertical differential at $x$.
\begin{enumerate}[label=(\alph*)]
    \item $\Fsection$ is a \textit{Fredholm section of index $d$} if for every $x \in \Fsection^{\,-1}(0) $ the vertical differential $\Dvert \Fsection(x) : T_x \Bcal \rightarrow \Ecal_x$ is a Fredholm operator of index $d \,$.
    \item $\Fsection$ is \textit{transverse to the zero section} if for every $x \in \Fsection^{\,-1}(0)$ the vertical differential $\Dvert \Fsection(x)$ is surjective and its kernel is a complemented subspace in $T_x \Bcal \,$.
\end{enumerate}
\end{definition}

\begin{definition}[B-bundle pair, \cf Def.~2.7 and 2.8 in \cite{Abstract_Perturbations}]
\label{def: Non-Empty Moduli Spaces - B-bundle pair} \hfill
\begin{enumerate}[label=(\alph*)]
    \item A \textit{B-pair} $(\mathbbm{F},\mathbbm{F}_1)$ is a pair of Banach(able) spaces, so that $\mathbbm{F}_1 \subseteq \mathbbm{F}$ is a dense linear subspace and so that the inclusion $\mathbbm{F}_1 \hookrightarrow \mathbbm{F}$ is a compact linear operator.
    \item A \textit{B-bundle pair $(\Ecal,\Ecal^1)$ over $\Bcal$} consists of two smooth Banach bundles $\Ecal$ and $\Ecal^1$ over $\Bcal \,$, so that
    \begin{enumerate}[label=\arabic*)]
        \item for every $x \in \Bcal$ the fiber $\Ecal^1_x$ is a linear subspace of $\Ecal_x \,$,
        \item for every $x \in \Bcal$ there exists a B-pair $(\mathbbm{F},\mathbbm{F}_1)$ and a trivialization $\Phi : \Ecal|_{\mathcal{U}} \overset{\cong}{\longrightarrow} \mathcal{U} \times \mathbbm{F}$ of $\Ecal$ over an open neighborhood $\mathcal{U}$ of $x$ which restricts to a trivialization $\Phi: \Ecal^1|_{\mathcal{U}} \overset{\cong}{\longrightarrow} \mathcal{U} \times \mathbbm{F}_1 $ of $\Ecal^1\,$.
    \end{enumerate}
\end{enumerate}
\end{definition}

\noindent We now quote the perturbation theorem alluded to above.

\begin{theorem}[Thm.~D in \cite{Abstract_Perturbations}]
\label{thm: Non-Empty Moduli Space - abstract perturbation thm}
Suppose we are given
\begin{enumerate}[label=\arabic*)]
    \item a Banach manifold $\Bcal$ (smooth, without boundary) which admits smooth bump functions and a Banach bundle $\Ecal \rightarrow \Bcal$ which can be extended to a B-bundle pair $(\Ecal,\Ecal^1)$ over $\Bcal \,$,
    \item a section $\Fsection : [0,R] \times \Bcal \rightarrow \EcalHat$ of the pullback bundle $\EcalHat$ of $\Ecal$ along $[0,R] \times \Bcal \rightarrow \Bcal \,$; We assume $\Fsection$ is a Fredholm section of index $d+1$ with compact zero set $\Fsection^{\,-1}(0)$ and that $\Fsection_0 = \Fsection(0,\,\cdot\,) : \Bcal \rightarrow \Ecal$ is transverse to the zero section.
    \item a finite-dimensional manifold $M$ and a submanifold $Z \subseteq M$ (both without boundary) of codimension $d$ which is a closed subset of $M$,
    \item a smooth map $e : [0,R] \times \Bcal \rightarrow M$ so that the restriction of $e_0 = e(0,\,\cdot\,) : \Bcal \rightarrow M$ to $\Fsection_0^{\,-1}(0)$ intersects $Z$ transversally in precisely one point, that is
    \[
    e_0|_{\Fsection_0^{\,-1}(0)} \pitchfork Z \text{ and } \, \# \big( e_0^{-1}(Z) \cap \Fsection_0^{\,-1}(0) \big) = 1 \,\, .
    \]
\end{enumerate}
Then there exists $x \in \Fsection_R^{\,-1}(0)$ with $e(R,x) \in Z \,$.
\end{theorem}

\noindent Once we have verified all of the assumptions in the preceding theorem, we will apply it in Subsection \ref{subsec: Non-Empty Moduli Spaces - Concluding proof} to conclude the proof of Proposition \ref{prop: moduli^H_[0,R](z) non-empty}.

\subsection{Analytic set-up}
\label{subsec: Non-Empty Moduli Spaces - Analytic Set-Up}

The analytic set-up follows Fauck's thesis \cite{FauckThesis} closely.
We remind the reader of the notation
\[
(W, \lambda , \omega) = (S \Sigma , \, e^h \alpha , \, \d(e^h \alpha)) \quad \text{ and } \quad \loopspace = \Cinfty(\bbS^1,W) \,\, .
\]

\begin{lemma}[Tubular neighborhoods]
\label{lem: Analytic Set-Up - tubular nbhd}
For every $(v,\tau(\alpha)) \in \CCrit$ there exists a tubular neighborhood $U \subseteq W$ of $\image(v) = v(\bbS^1)$ together with smooth coordinates
\[
(\vartheta, y) = (\vartheta,y_1,\ldots , y_{2N-1})  : U \overset{\cong}{\longrightarrow} \bbS^1 \times \R^{2N-1}
\]
so that $((\vartheta,y) \circ v)(t) = (t,0)$ for every $t \in \bbS^1$ and along $U \cap \Sigma = \{y_{2N-1} = 0\}$ it holds $\tau(\alpha) R_{\alpha} = \partial_{\vartheta} \,$. In particular $U \cap \Sigma$ is invariant under the Reeb flow.
\end{lemma}

\noindent Here, as usual, we identified $\{0\} \times \Sigma \subseteq S \Sigma = W$ with $\Sigma\,$

\begin{proof}
  Because $(\Sigma,\alpha)$ is Zoll with minimal period $\tau(\alpha)$, we have a well-defined free $\bbS^1$-action on $\Sigma$ by the flow of $\tau(\alpha) R_\alpha \,$. Hence we can choose $\bbS^1$-equivariant coordinates $(\vartheta,y_1,\ldots,y_{2N-2})$ around $v$ by the slice theorem, see \cite[Thm.~23.5]{Ana_Cannas_Silva}. As last coordinate set $y_{2N-1} := h \,$, the $\R$-coordinate in the symplectization.
\end{proof}

\noindent Due to the Morse-Bott situation, we have to work with weighted Sobolev spaces:

\begin{definition}[Weighted Sobolev space]
\label{def: Analytic Set-Up - weighted Sobolev space}
For $\Omega = I \times \bbS^1$ or $\Omega = I \,$, where $I \subseteq \R$ is an unbounded interval, and every $m \in \N_{\geq 0} \comma  p \geq 1 \comma \kappa \geq 0$ we define the \textit{$\kappa$-weighted Sobolev space} by
\begin{align*}
W^{m,p}_\kappa(\Omega) &:= \set{f \in W^{m,p}_{\mathrm{loc}}(\Omega)}{ e^{\kappa |s|} \partial^\beta f \in L^p(\Omega) \,\,\, \Forall |\beta| \leq m} \\
&= \set{f \in W^{m,p}_{\mathrm{loc}}(\Omega)}{e^{\kappa \rho(s) s} f \in W^{m,p}(\Omega)}
\end{align*}
for any choice of increasing function $\rho \in \Cinfty(\R,[-1,1])$ with $\rho(s) = -1$ for $s \leq -1$ and $\rho(s) = 1$ for $s \geq  1 \,$. For $m=0$ we also write $L^p_{\kappa}(\Omega) := W^{0,p}_{\kappa}(\Omega) \,$. We endow $W^{m,p}_\kappa(\Omega)$ with the norm $\lVert f \rVert_{W^{m,p}_\kappa} := \lVert e^{\kappa \rho(s) s} f \rVert_{W^{m,p}} \,$. This norm is equivalent to the norm $f \mapsto \sum_{|\beta| \leq m} \lVert e^{\kappa |s|} \partial^{\beta} f \rVert_{L^p} \,$.
\end{definition}

\noindent We fix, once and for all, an even integer $p > 2 \,$. (For example we could set $p = 4$.) Moreover, we fix an exponential decay constant $\kappa_0 > 0$ so that
\begin{align}
\label{eq: Analytic Set-Up - choice kappa_0 spec Hessian}
\kappa_0 < \min \set{|\mu|}{0 \not= \mu \text{ eigenvalue of } \nabla^2\RabinowitzHZero(v,\tau(\alpha)) \comma (v,\tau(\alpha)) \in \CCrit}
\end{align}
 and so that additionally there exists a larger constant having the exponential decay property as in Proposition \ref{prop: Uniform Exponential Decay}, that is
\begin{align}
\label{eq: Analytic Set-Up - choice kappa_0 exponential decay}
    \Exists \kappa > \kappa_0 : \,\, \kappa \text{ has the significance stated in Proposition \ref{prop: Uniform Exponential Decay}.}
\end{align}
Such $\kappa_0$ exists by \eqref{eq: min spec Hessian > 0} and Proposition \ref{prop: Uniform Exponential Decay}.

\begin{lemma}
\label{lem: Analytic Set-Up - exponential decay}
For $s_0 \in \R \,$, an $\RabinowitzHZero$-gradient flow line $w = (u,\eta) \in \Cinfty(\R_{\geq s_0} , \loopspace \times \R)$ with limit $w(\infty) \in \CCrit$ and every tubular neighborhood $(U, (\vartheta,y))$ around $u(\infty,\cdot\,)$ as in Lemma \ref{lem: Analytic Set-Up - tubular nbhd} there exists $s_1 \geq s_0$ with $u(s,t) \in U$ for $s \geq s_1 , t \in \bbS^1$, and 
\begin{align*}
\big(\vartheta ( u \,(s,t) ) - t \comma y ( u \, (s,t) ) \big) \in W^{1,p}_{\kappa_0}(\R_{\geq s_1} \times \bbS^1 ,\R^{2N}) \,\text{ and } \,\eta - \tau(\alpha) \in W^{1,p}_{\kappa_0}(\R_{\geq s_1},\R) \, .
\end{align*}
\end{lemma}

\noindent Notice that the component $(s,t) \mapsto \vartheta ( u(s,t) ) - t \,$, which takes on values in $\bbS^1 \,$, lifts to a map with values in $\R \,$. There is an analogous statement for gradient flow lines defined on $\R_{\leq s_0}$ with limit $w(-\infty) \in \CCrit$.

\begin{proof}[Proof of Lemma \ref{lem: Analytic Set-Up - exponential decay}]
   Due to \eqref{eq: Analytic Set-Up - choice kappa_0 exponential decay} it follows immediately that $\partial_s u$ and $\partial_s \eta$ have exponential decay as in \eqref{eq: Uniform Exponential Decay - partial_s} for some exponential coefficient $\kappa > \kappa_0 \,$. From \eqref{eq: Uniform Exponential Decay - d(u, u_infty)} and \eqref{eq: Uniform Exponential Decay - |eta - eta(infty)|} we obtain exponential decay as $s \rightarrow \infty$ with decay coefficient $\kappa$ for $|\eta(s) - \tau(\alpha)|$ and for $\mathrm{dist}(u(s,t), u(\infty,t)) \,$, uniform in $t \in \bbS^1 \,$. The exponential decay bounds on $\partial_s u$ and the first equation in \eqref{eq: Rabinowitz-Floer eq} yield an exponential decay estimate for $\partial_t u - \eta X_{H_0}(u) $ with same coefficient $\kappa \,$. From what we have just proven, we know that $\eta X_{H_0}(u)$ tends to 
    \[
    \eta(\infty) X_{H_0}(u(\infty,t)) = \tau(\alpha) R_\alpha(u(\infty,t)) = \partial_{\vartheta}
    \]
    exponentially fast. Thus $\partial_t u - \partial_\vartheta$ has exponential decay with coefficient $\kappa \,$.
\end{proof}

\begin{definition}
\label{def: Analytic Set-Up - Banach mfd} We define the underlying Banach manifold $\Bcal \,$.
\begin{enumerate}[label=(\alph*)]
    \item Let $\Bcal'$ be the set consisting of those 
    \[
    w=(u,\eta) \in W^{1,p}_{\mathrm{loc}}(\R \times \bbS^1, W) \times W^{1,p}_{\mathrm{loc}}(\R,\R)
    \]
    with existing $(C^0(\bbS^1,W) \times \R)$-limits $w(\pm \infty) \in \CCrit$ and so that, for any tubular neighborhoods $(U^\pm, (\vartheta^\pm,y^\pm))$ around $u(\pm \infty,\cdot\,)$ as in Lemma \ref{lem: Analytic Set-Up - tubular nbhd}, there exists $s_0 >0$ sufficiently large with
    \begin{align*}
     \big(\vartheta^+ (u(s,t)) - t \comma y^+ (u(s,t)) \big) &\in W^{1,p}_{\kappa_0}(\R_{\geq s_0} \times \bbS^1, \R^{2N}) \, ,\\
     \big(\vartheta^-( u (s,t)) - t \comma y^- (u(s,t)) \big) &\in W^{1,p}_{\kappa_0}(\R_{\leq - s_0} \times \bbS^1, \R^{2N}) \, , \\
     \eta - \tau(\alpha) &\in W^{1,p}_{\kappa_0}(\R,\R) \, ,
    \end{align*}
    If this condition holds for some pair of tubular neighborhoods $(U^\pm , (\vartheta^\pm, y^\pm)) \,$, then it holds for every such pair of tubular neighborhoods.
    \end{enumerate}
    There is a natural topology and smooth structure on $\Bcal' \,$, making it into a smooth Banach manifold without boundary, see Appendix \ref{app subsec: Bcal' Atlas}. With this topology, $\Bcal'$ is second-countable and metrizable, see Corollary \ref{app cor: Bcal' Atlas - Bcal' second countable}.
    \begin{enumerate}[label=(\alph*),resume]
    \item Let $\Bcal$ be the connected component of $\Bcal'$ containing $\CCrit \,$, \ie containing all constant $w \in \Bcal' \,$. In other words
    \begin{align*}
        \Bcal = \set{(u,\eta) \in \Bcal'}{u \simeq 0 \text{ rel } \CCrit} \,\, ,
    \end{align*}
    using the notation established in Definition \ref{def: homotopic rel C^tau(alpha)}.
\end{enumerate}
\end{definition}

\noindent We endow $\Rbar = [-\infty,\infty]$ with a smooth structure by requiring the map
\begin{equation}
\label{eq: Rbar smooth structure}
\Rbar \overset{\cong}{\longrightarrow} [-1,1] \comma \quad s \mapsto \frac{s}{\sqrt{1+s^2}}
\end{equation}
to be a diffeomorphism. Observe that the open subset $\R \subseteq \Rbar$ inherits the standard smooth structure. We emphasize that $\Bcal' $ is contained in $ C^0(\Rbar \times \bbS^1,W) \times C^0(\Rbar,\R)$ by the Sobolev embedding theorem, as $p > 2 \,$.

\begin{lemma}
\label{lem: Analytic Set-Up - Bcal smooth bump fcts}
The Banach manifold $\Bcal'$ admits smooth bump functions. Hence, so does its connected component $\Bcal \,$.
\end{lemma}

\begin{proof}
    Locally the metrizable Banach manifold $\Bcal'$ is modeled on the Banach space
    \[
    W^{1,p}_{\kappa_0}(\R \times \bbS^1,\R^{2N}) \times \R^{2N-1} \times \R^{2N-1} \times W^{1,p}_{\kappa_0}(\R,\R)
    \]
    and it suffices to show that this Banach space admits smooth bump functions. Since the $\kappa_0$-weighted Sobolev spaces are isometric to the ordinary Sobolev spaces (the isometry given by $W^{1,p}_{\kappa_0} \rightarrow W^{1,p} \comma f \mapsto e^{\kappa_0 \rho(s)s} f$), it suffices to show that $W^{1,p}(\R \times \bbS^1,\R^{2N})$ respectively $W^{1,p}(\R,\R)$ admit smooth bump functions. This holds indeed since $p$ is an even integer, see \cite[Thm. 1]{Lp_smooth_bump_fcts}.
\end{proof}

\noindent Let us consider the smooth Banach bundle $\Ecal$ over $\Bcal \,$, whose fiber over $w = (u,\eta) \in \Bcal$ is given by
\begin{align*}
    \Ecal_{(u,\eta)} := L^p_{\kappa_0}(\R \times \bbS^1, u^*TW) \oplus L^p_{\kappa_0}(\R,\R) \,\, .
\end{align*}
We briefly explain the definition of the appearing $\kappa_0$-weighted Sobolev space $L^p_{\kappa_0}(\R \times \bbS^1,u^*TW) \,$.
Observe that, because $W$ is orientable as symplectic manifold, the $C^0$-bundle $u^*TW$ over $\Rbar \times \bbS^1$ is orientable and hence trivial. Fixing a global $C^0$-trivialization $\Phi : u^*TW \overset{\cong}{\longrightarrow} (\Rbar \times \bbS^1) \times \R^{2N} \,$, the $\kappa_0$-weighted Sobolev space along $u^*TW$ is defined by
\[
L^p_{\kappa_0}(\R \times \bbS^1,u^*TW) := \set{(s,t) \mapsto \Phi^{-1}_{(s,t)} (f(s,t))}{f \in L^p_{\kappa_0}(\R \times \bbS^1,\R^{2N})}
\]
and this is independent of the chosen trivialization $\Phi \,$.

Let $R >0$ be arbitrary but fixed. We denote by $\EcalHat$ the pullback bundle of $\Ecal$ along the projection to the second component
\[
\mathrm{pr}_2 : \, \BcalHat := [0,R] \times \Bcal \longrightarrow \Bcal \,\, .
\]
Define the smooth section $\Fsection : \BcalHat \rightarrow \EcalHat$ of $\EcalHat$ by
\[
\Fsection(r,w) := \Fsection_r(w) := \partial_s w - \nabla^{g_J} \RabinowitzHrs (w) \qquad \text{ for } (r,w) \in [0,R] \times \Bcal \, .
\]
Writing the definition out, this means
\begin{align*}
  \Fsection_r(u,\eta) = { {\partial_s u + J(u) \partial_t u - \eta J(u) \big( (1- \beta_r) X_{H_0}(u) + \beta_r X_H(u) \big) }
  \choose 
  {
   \partial_s \eta + \int_0^1 \big((1-\beta_r) H_0 + \beta_r H \big)(u) \, \d t 
    }} \,\, .
\end{align*}

\begin{lemma}
\label{lem: Analytic Set-Up - zero set Fsection}
For every $r \in [0,R]$ the zero set $\Fsection_r^{\,-1}(0) \subseteq \Bcal$ is precisely $\moduliH{r} \,$. Consequently, $\Fsection^{-1}(0) = \moduliH{[0,R]} \,$.
\end{lemma}

\begin{proof}
    Notice that $\moduliH{r} \subseteq \Bcal$ due to Lemma \ref{lem: Analytic Set-Up - exponential decay}. It is then obvious that elements in $\moduliH{r}$ are zeroes of $\Fsection_r \,$. 
    
    Conversely, if $w=(u,\eta) \in \Bcal$ is a zero of $\Fsection_r \,$, then both $u$ and $\eta$ are smooth by elliptic bootstrapping. In view of Remark \ref{rmk: Preparation - moduli^H_r limits}, to prove $w \in \moduliH{r}$ it now suffices to show that $w$ has finite energy. As $w \in \Bcal'$, there exist sequences $(s_n^{\pm})_n \subseteq \R$ tending to $\pm\infty$ so that $w(s_n^{\pm}) \convergence{n} w(\pm \infty)$ in $W^{1,2}(\bbS^1,W) \times \R \,$. Consequently $w$ has finite energy because for $|s| \geq r+1$ it is a gradient flow line of $\RabinowitzHZero$ with finite limits
    \[
    \lim_{s \rightarrow \pm \infty} \RabinowitzHZero(w(s)) = \lim_{n \rightarrow \infty} \RabinowitzHZero(w(s_n^\pm)) = \RabinowitzHZero(w(\pm \infty))
    \]
    by $(W^{1,2}(\bbS^1,W) \times \R)$-continuity of $\RabinowitzHZero \,$.
\end{proof}

\noindent Next we show that the Banach bundle $\Ecal$ can be naturally extended to a B-bundle pair (see Definition \ref{def: Non-Empty Moduli Spaces - B-bundle pair}).

\begin{lemma}
\label{lem: Analytic Set-Up - B-bundle pair}
There exists a Banach bundle $\Ecal^1$ over $\Bcal$ so that $(\Ecal,\Ecal^1)$ is a B-bundle pair.
\end{lemma}

\begin{proof}
    Given $w_0 = (u_0,\eta_0) \in \Bcal \cap \Cinfty(\Rbar,\loopspace \times \R)$ and a chart $(f_*^{-1},\Ucal)$ of $\Bcal$ centered at $w_0$ as in Appendix \ref{app subsec: Bcal' Atlas} so that $f_*^{-1}(\Ucal)$ is a convex subset of its ambient Banach space. For every $w = (u,\eta) \in \Ucal$ there is a natural choice of path $ [0,1] \rightarrow \Ucal \comma \sigma \mapsto (u_w^\sigma,\eta_w^\sigma)$ from $(u_w^0,\eta_w^0) = w_0$ to $(u_w^1,\eta_w^1) = w \,$, namely the straight line connecting $w_0$ to $w$ in the chart, that is $ (u_w^\sigma,\eta_w^\sigma) := f_* ( \sigma f_*^{-1}(w) ) \,$. For fixed $(s,t) \in \R \times \bbS^1$ consider the parallel transport
    \[
    \mathrm{P}_{w}^{(s,t)} : T_{u_0(s,t)} W \overset{\cong}{\longrightarrow} T_{u(s,t)} W 
    \]
     with respect to the metric $\omega(\,\cdot\, , J \,\cdot \,)$ along the path $[0,1] \rightarrow W \comma \sigma \mapsto u_w^\sigma(s,t) \,$.
    
     A trivialization $\Phi$ of $\Ecal$ over $\Ucal$ then is given by
     \begin{align*}
         &\Phi^{-1} : \Ucal \times \big( L^p_{\kappa_0}(\R \times \bbS^1,u_0^*TW) \oplus L^p_{\kappa_0}(\R,\R)  \big) \overset{\cong}{\longrightarrow} \Ecal|_{\Ucal} \\
         &\Phi^{-1}(w, \,( \mathfrak{u}, \mathfrak{n})) := \Phi_w^{-1}(\mathfrak{u},\mathfrak{n}) := \big( \big( (s,t) \mapsto \mathrm{P}_w^{(s,t)}(\mathfrak{u}(s,t)) \big) , \, \mathfrak{n} \big) \,\, .
     \end{align*}
    Such trivializations cover $\Bcal \,$.

    \begin{claim}
    For every $\kappa_1 > \kappa_0$ the pair
    \[
    \big( L^p_{\kappa_0}(\R \times \bbS^1,u_0^*TW) \oplus L^p_{\kappa_0}(\R,\R) \comma W^{1,p}_{\kappa_1}(\R \times \bbS^1,u_0^*TW) \oplus W^{1,p}_{\kappa_1}(\R,\R) \big)
    \]
    is a B-pair.
    \end{claim}
    \noindent This claim is for example also asserted in \cite[Lemma 5.7]{Wehrheim_Fredholm_Notions}, although therein in the language of scale theory.
    
    \begin{proofClaim}
     One readily verifies that, for every $\kappa > 0$ and $\Omega =\R \times \bbS^1$ or $\Omega = \R \,$, the pair $(L^p(\Omega), W^{1,p}_\kappa(\Omega))$ is a B-pair. Conjugating this B-pair with the isometry $L^p(\Omega) \rightarrow L^p_{\kappa_0}(\Omega) \comma \, f \mapsto e^{- \kappa_0 \rho(s)s} f \,$, gives a B-pair $(L^p_{\kappa_0}(\Omega) , W^{1,p}_{\kappa_0 + \kappa}(\Omega)) \,$.
    \end{proofClaim}

    \noindent We now fix $\kappa_1 > \kappa_0 \,$. For $w \in \Bcal$ we define $\Ecal^1_w \subseteq \Ecal_w$ by choosing a trivialization $(\Phi,\Ucal,w_0)$ as above with $w \in \Ucal$ and setting
    \[
    \Ecal^1_w := \Phi^{-1}_w \big( W^{1,p}_{\kappa_1}(\R \times \bbS^1, u_0^*TW) \oplus W^{1,p}_{\kappa_1}(\R,\R) \big) \,\, .
    \]
    Given another trivialization $(\widetilde{\Phi},\widetilde{\Ucal},\widetilde{w}_0)$ with $w \in \widetilde{\Ucal} \,$, the toplinear isomorphism
    \[
    \widetilde{\Phi}_w \circ \Phi_w^{-1} \in \mathcal{L}\big(  L^p_{\kappa_0}(u_0^*TW) \oplus L^p_{\kappa_0}(\R,\R) \comma L^p_{\kappa_0}(\widetilde{u}_0^*TW) \oplus L^p_{\kappa_0}(\R,\R) \big)
    \]
    restricts to a well-defined toplinear isomorphism
     \[
    \widetilde{\Phi}_w \circ \Phi_w^{-1} \in \mathcal{L}\big(  W^{1,p}_{\kappa_1}(u_0^*TW) \oplus W^{1,p}_{\kappa_1}(\R,\R) \comma W^{1,p}_{\kappa_1}(\widetilde{u}_0^*TW) \oplus W^{1,p}_{\kappa_1}(\R,\R) \big) \,\, .
    \]
     This shows that $\Ecal^1_w$ is independent of the trivialization chosen. Moreover the map
    \begin{align*}
        \Ucal \cap \widetilde{\Ucal} \rightarrow \mathcal{L}(W^{1,p}_{\kappa_1} \oplus W^{1,p}_{\kappa_1} \comma W^{1,p}_{\kappa_1} \oplus W^{1,p}_{\kappa_1}) \comma \quad w \mapsto \widetilde{\Phi}_w \circ \Phi_w^{-1}
    \end{align*}
    is smooth, which can be seen in the same way as one shows that the transition maps for $\Ecal $ are smooth.\footnote{
    In more detail, smoothness of the transition map follows from smoothness of the map
    \begin{align*}
        &W^{1,p}_{\kappa_0}(u_0^*TW) \times \R^{2(2N-1)} \longrightarrow \mathcal{L}(W^{1,p}_{\kappa_1}(u_0^*TW), \, W^{1,p}_{\kappa_1}(\widetilde{u}_0^*TW)) \\
        &(\xi,x) \mapsto A_*^{(\xi,x)} \comma \,\, \text{ with } \, (A_*^{(\xi,x)} \cdot \mathfrak{u}) (s,t) := A_{(s,t)}(\xi(s,t),x) \cdot \mathfrak{u}(s,t) \,\, ,
    \end{align*}
    where
    \[
    A : u_0^*TW \times \R^{2(2N-1)} \longrightarrow \mathrm{Hom}(u_0^*TW, \,\widetilde{u}_0^*TW)
    \]
    is a smooth fiber preserving map over $\Rbar \times \bbS^1$ arising from parallel transport.
    }
    Together with the claim, this proves that $(\Ecal,\Ecal^1)$ constitutes a B-bundle pair. 
\end{proof}

\subsection{Concluding the proof of Proposition \ref{prop: moduli^H_[0,R](z) non-empty}}
\label{subsec: Non-Empty Moduli Spaces - Concluding proof}

In this subsection, we are going to collect the missing pieces to conclude Proposition \ref{prop: moduli^H_[0,R](z) non-empty}. Some proofs will be outsourced to subsequent subsections.

\begin{proposition}[Compactness]
\label{prop: Non_Empty Moduli Spaces - Compactness}
The zero set $\moduliH{[0,R]}$ of $\Fsection$ is a compact subset of $\BcalHat = [0,R] \times \Bcal\,$.
\end{proposition}

\noindent The proof of Proposition \ref{prop: Non_Empty Moduli Spaces - Compactness} is deferred to Subsection \ref{subsec: Non_Empty Moduli Spaces - Compactness}. It relies on Gromov compactness (Proposition \ref{prop: Gromov Compactness}) and on the bounds in Lemma \ref{lem: Proof MThm - bounds on energy, Lagrange multiplier, cylinder}.

\begin{proposition}[Fredholm Property]
\label{prop: Non_Empty Moduli Spaces - Fredholm}
For every $(r,w) \in \moduliH{[0,R]}$ the vertical differential $\Dvert \Fsection_r(w) : T_w \Bcal \rightarrow \Ecal_w$ is a Fredholm operator of Fredholm index $\dim(\Sigma) = 2 N - 1\,$.
\end{proposition}

\noindent The proposition expresses that, for every $0 \leq r \leq R \,$, the section $\Fsection_r : \Bcal \rightarrow \Ecal$ is a Fredholm section of index $2N-1\,$. Consequently, the section $\Fsection : [0,R] \times \Bcal \rightarrow \EcalHat$ is a Fredholm section of index $2N \,$. We defer the proof of Proposition \ref{prop: Non_Empty Moduli Spaces - Fredholm} to Subsection \ref{subsec: Non_Empty Moduli Spaces - Fredholm}.

\begin{proposition}[Transversality for $r=0$]
\label{prop: Non_Empty Moduli Spaces - Transversality}
The section $\Fsection_0 : \Bcal \rightarrow \Ecal$ is transverse to the zero section.
\end{proposition}

\noindent It is proven in \cite[App.~C]{FauckThesis} that $\Fsection_0 = \partial_s - \nabla \RabinowitzHZero$ has surjective vertical differential at every $s$-independent zero and these are all the zeroes of $\Fsection_0$ by Lemma \ref{lem: moduli_0^H}, hence $\Fsection_0$ is transverse to the zero section. For the sake of completeness, we present a different proof of Proposition \ref{prop: Non_Empty Moduli Spaces - Transversality} in Subsection \ref{subsec: Non-Empty Moduli Spaces - Transversality r=0}, which does not rely on \cite[App.~C]{FauckThesis}.

Combining the index formula in Proposition \ref{prop: Non_Empty Moduli Spaces - Fredholm} with the transversality assertion in Proposition \ref{prop: Non_Empty Moduli Spaces - Transversality}, we see that $\moduliH{0} = \Fsection_0^{\,-1}(0)$ is a smooth submanifold of $\Bcal$ of dimension $\dim(\Sigma) \,$.

Now consider the evaluation map
\[
\ev : \Bcal \longrightarrow W = S \Sigma \comma \quad \ev(u,\eta) := u(0,0) \,\, . 
\]
In local coordinates one readily checks that $\ev$ is smooth.

\begin{lemma}
\label{lem: Non_Empty Moduli Spaces - ev transverse}
The restriction $\ev|_{\moduliH{0}}$ is transverse to $\R \times \{z\} \subseteq S \Sigma $ and $ \moduliH{0}(z) = (\ev|_{\moduliH{0}})^{-1}(\R \times \{z\})$ consists of a single point.
\end{lemma}

\begin{proof}
Notice that $\ev|_{\moduliH{0}}$ maps bijectively onto $\{0\} \times \Sigma \subseteq S \Sigma $ and, in fact, agrees with the identification $\moduliH{0} \overset{\cong}{\longrightarrow} \{0\} \times \Sigma$ in Lemma \ref{lem: moduli_0^H}. Its inverse $\{0\} \times \Sigma \rightarrow \moduliH{0} \subseteq \Bcal$ is given explicitly by mapping a point $(0,z_1) \in \{0\} \times \Sigma$ to the $s$-independent $w = (u,\tau(\alpha))$ with $u(s,t) := (0, \, \phi_{R_\alpha}^{\tau(\alpha)t}(z_1)) \in \{0\} \times \Sigma \,$. The inverse is thus smooth as map to $\Bcal$ and hence also as map to the submanifold $\moduliH{0} \,$. Therefore $\ev|_{\moduliH{0}}$ is a smooth embedding onto $\{0\} \times \Sigma \,$. This implies the lemma.
\end{proof}

\noindent We are now ready to give the proof of Proposition \ref{prop: moduli^H_[0,R](z) non-empty}.

\begin{proof}[Proof of Proposition \ref{prop: moduli^H_[0,R](z) non-empty}]
  By Lemma \ref{lem: Analytic Set-Up - Bcal smooth bump fcts} the Banach manifold $\Bcal$ admits smooth bump functions and from Lemma \ref{lem: Analytic Set-Up - B-bundle pair} we know that the bundle $\Ecal$ over $\Bcal$ extends to a B-bundle pair $(\Ecal,\Ecal^1)$ over $\Bcal$. By Propositions \ref{prop: Non_Empty Moduli Spaces - Compactness} and \ref{prop: Non_Empty Moduli Spaces - Fredholm}, the section $\Fsection : [0,R] \times \Bcal = \BcalHat \rightarrow \EcalHat$ is a Fredholm section of index $\dim(\Sigma) + 1$ with compact zero set $\moduliH{[0,R]} \,$. Moreover $\Fsection_0$ is transverse to the zero section by Proposition \ref{prop: Non_Empty Moduli Spaces - Transversality}. Consider the submanifold $Z:=\R \times \{z\} \subseteq S \Sigma$ (observe that $\mathrm{codim}_{S \Sigma}(Z) = \dim(\Sigma) \,$ and that $Z$ is a closed subset of $S \Sigma$) and the smooth map
  \[
  \widehat{\mathrm{ev}} : \BcalHat = [0,R] \times \Bcal \longrightarrow S \Sigma \comma \quad \widehat{\mathrm{ev}}(r,w) := \ev(w) \,\, .
   \]
   By Lemma \ref{lem: Non_Empty Moduli Spaces - ev transverse}, the restriction of $\widehat{\mathrm{ev}}$ to $\{0\} \times \Fsection_0^{\,-1}(0) = \{0\} \times \moduliH{0}$ intersects $Z$ transversally in precisely one point. Hence all premises of Theorem \ref{thm: Non-Empty Moduli Space - abstract perturbation thm} are satisfied, so that there exists $w \in \Bcal$ in the zero set of $\Fsection_R$ with $\widehat{\mathrm{ev}}(R,w) \in Z \,$. This means $w \in \moduliH{R}(z) \,$, which shows that $\moduliH{R}(z)$ is non-empty.
\end{proof}

\noindent It remains to prove the compactness assertion in Proposition \ref{prop: Non_Empty Moduli Spaces - Compactness}, Proposition \ref{prop: Non_Empty Moduli Spaces - Fredholm} on the Fredholm property and Proposition \ref{prop: Non_Empty Moduli Spaces - Transversality} on transversality for $r=0 \,$. We will accomplish this in the next three subsections.

\subsection{Compactness}
\label{subsec: Non_Empty Moduli Spaces - Compactness}

We now show Proposition \ref{prop: Non_Empty Moduli Spaces - Compactness}, that is, that the moduli space $\moduliH{[0,R]}$ is a compact subset of $\BcalHat = [0,R] \times \Bcal \,$. Essentially, compactness holds by Gromov compactness and because breaking of Floer cylinders is prevented due to the uniform energy bound, $E(w) \leq \eps$ for every $(r,w) \in \moduliH{[0,R]}$ (see Lemma \ref{lem: Proof MThm - bounds on energy, Lagrange multiplier, cylinder}), with energy bound $\eps$ being smaller than the spectral gap $\mathrm{specgap}(\RabinowitzHZero, \CCrit) = \tau(\alpha)$ (see Lemma \ref{lem: Crit(RabinowitzHZero)} and \eqref{eq: Proof MThm - eps < tau(alpha)}). Below we will give the details of this argument.

\begin{lemma}
\label{lem: Non_Empty Moduli Spaces - Compactness C^0(loop)-convergence}
For every sequence $(r_n,w_n)_{n \in \N} = (r_n,(u_n,\eta_n)) \subseteq \moduliH{[0,R]}$ there exists a subsequence $(r_{n_k},w_{n_k})_{k \in \N}$ and $(r,w) = (r,(u,\eta)) \in \moduliH{[0,R]}$ so that
\[
(r_{n_k}, u_{n_k}, \eta_{n_k}) \overset{k\rightarrow \infty}{\longrightarrow} (r,u,\eta) \, \text{ in } [0,R] \times \Cinftyloc(\R \times \bbS^1,W) \times \Cinftyloc(\R,\R)
\]
as well as $(r_{n_k},w_{n_k}) \overset{k \rightarrow \infty}{\longrightarrow} (r,w)$ in $[0,R] \times C^0(\R,\loopspace \times \R) \,$.
\end{lemma}

\noindent Here $\loopspace \times \R = \Cinfty(\bbS^1,W) \times \R$ is considered to be a metric space, whose topology is the product topology of the $\Cinfty$-topology on $\loopspace$ and the standard topology on $\R$. Convergence in $C^0(\R,\loopspace \times \R)$ means metric uniform convergence. This is independent of the metric used since $\moduliH{[0,R]} \subseteq C^0(\Rbar,\loopspace \times \R)$ and metric uniform convergence on $\Rbar$ is the same as convergence in the compact-open topology.

\begin{proof}[Proof of Lemma \ref{lem: Non_Empty Moduli Spaces - Compactness C^0(loop)-convergence}]
    We will suppress subsequences in the notation throughout. Due to the uniform bounds in Lemma \ref{lem: Proof MThm - bounds on energy, Lagrange multiplier, cylinder}, we can apply Gromov compactness (Proposition \ref{prop: Gromov Compactness}) with Hamiltonian 
    \[
    H_{r,s} := (1-\beta_r(s)) H_0 + \beta_r(s) H \,\, ,
    \]
    so that a subsequence $(r_n,w_n)_n$ converges to some $(\Bar{r},\Bar{w}) = (\Bar{r},(\Bar{u},\Bar{\eta}))$ in $[0,R] \times \Cinftyloc(\R \times \bbS^1,W) \times \Cinftyloc(\R,\R) \,$. By Remark \ref{rmk: Gromov compactness - limit gradient flow line} the limit $\Bar{w}$ has energy bounded above by $\sup_n E(w_n) \leq \eps$ and is a solution to the $s$-dependent gradient equation
    \begin{equation}
    \label{eq: Compactness - limit satisfies gradient eq}
    \partial_s \Bar{w} (s) = \nabla \RabinowitzHParam{\Bar{r},s}(\Bar{w}(s)) \qquad \Forall s \in \R \,\, .
    \end{equation}
    \begin{claim}
    $\Bar{w}$ has existing limits $\Bar{w}(\pm\infty) \in \CCrit$ in the $(\Cinfty \times \R)$-topology and, after extracting a further subsequence, $(r_n,w_n)_n$ converges to $(\Bar{r},\Bar{w})$ in $C^0(\R,\loopspace \times \R) \,$.
    \end{claim}
    \noindent Provided the claim is true, we can conclude that $(\Bar{r},\Bar{w}) \in \moduliH{[0,R]} \,$. To this end, use \eqref{eq: Compactness - limit satisfies gradient eq} and notice also that $\Bar{u} \simeq 0$ rel $\CCrit$ since $u_n \simeq 0$ rel $\CCrit$ for every $n \in \N$ and we have convergence $u_n \rightarrow \Bar{u}$ in $C^0(\Rbar \times \bbS^1,W) $ by the claim.
    The claim thus finishes the proof of the lemma.
\\

\noindent \textit{Proof of the claim.} Due to convergence $(u_n,\eta_n) \rightarrow (\Bar{u},\Bar{\eta}) $ in $\Cinftyloc(\R \times \bbS^1,W) \times \Cinftyloc(\R,\R) \,$, we are only interested in the asymptotic behavior.
        We will restrict our considerations to positive asymptotics, negative asymptotics work analogously.
        
        The claim will follow from \cite[Thm.~A]{Compact_Moduli_Spaces_Gradient_Floer_Lines} respectively its slight refinement \cite[Thm.~2.12] {Compact_Moduli_Spaces_Gradient_Floer_Lines} with
        \begin{equation}
        \label{eq: Compactness - notation translation to Compactness paper}
        (M,Z, f ,G,\zeta, E_0)_{\text{\cite{Compact_Moduli_Spaces_Gradient_Floer_Lines}}} := (\loopspace \times \R , \, \CCrit , \, \RabinowitzHZero , \, |\nabla \RabinowitzHZero|_{g_J}^2, \, \tau(\alpha), \, \eps)
        \end{equation}
        and the \textit{$(\RabinowitzHZero,|\nabla \RabinowitzHZero|_{g_J}^2)$-gradient flow line class}\footnote{see \cite[Def.~1.3]{Compact_Moduli_Spaces_Gradient_Floer_Lines}} $\Gamma$, defined for each closed interval $I_0 \subseteq \R$ by
        \begin{align*}
         \Gamma(I_0) := \set{w = (u,\eta) \in \Cinfty(I_0,\loopspace \times \R)}{\begin{array}{c}
           \partial_s w(s) = \nabla \RabinowitzHZero(w(s)) \,\,\, \Forall s \in I_0 \comma     \\
              \mathrm{im}(u) \subseteq [h_0^-,h_0] \times \Sigma \comma \lVert \eta \rVert_\infty \leq T^*
         \end{array}} \,\, .
        \end{align*}
        Here the constants $h_0,h_0^-$ and $T^*$ are the fixed ones, appearing in Lemma \ref{lem: Proof MThm - bounds on energy, Lagrange multiplier, cylinder}. Note that, because the $w_n \in \moduliH{r_n}$ are $\RabinowitzHZero$-gradient flow lines for $s \geq R+1 $ and
        again by Lemma \ref{lem: Proof MThm - bounds on energy, Lagrange multiplier, cylinder}, we have $w_n|_{[R+1,\infty)} \in \Gamma([R+1,\infty))$ for every $n$ and the uniform energy bound $E(w_n|_{[R+1,\infty)}) \leq \eps \,$. Consider the intervals $I := [R+1,\infty)$ and $I_1 := [R+2,\infty) \,$. We now verify the assumptions of \cite[Thm.~2.12]{Compact_Moduli_Spaces_Gradient_Floer_Lines}.

        \textit{General set-up:} The set-up in \cite{Compact_Moduli_Spaces_Gradient_Floer_Lines} requires $M$ to be a metrizable space, $f,G$ to be continuous functions on $M$ with $G \geq 0$ and $Z \subseteq G^{-1}(0)$ to be compact with $f|_Z \equiv \zeta \,$. With the choices in \eqref{eq: Compactness - notation translation to Compactness paper} all of this holds true: With the $(\Cinfty \times \R)$-topology, $\loopspace \times \R$ is indeed metrizable. Zeroes of $|\nabla \RabinowitzHZero|^2_{g_J}$ are just critical points of $\RabinowitzHZero \,$. We know from Lemma \ref{lem: Crit(RabinowitzHZero)} that $\RabinowitzHZero$ is constant on $\CCrit \subseteq \mathrm{Crit}(\RabinowitzHZero)$ with value $\tau(\alpha)$ and that $\CCrit \cong \Sigma$ is compact.

        The basic assumption for \cite[Thm.~2.12]{Compact_Moduli_Spaces_Gradient_Floer_Lines} now is that the energy bound $E_0 = \eps$ is strictly smaller than the \textit{spectral gap for positive/negative ends} $\mathfrak{S}^+_{f,G}(Z) \,$, defined by
        \[
         \mathfrak{S}^\pm_{f,G}(Z) := \inf\Bigset{\pm (\zeta - f(x))}{x \in G^{-1}(0)\backslash Z \text{ with } \pm (\zeta - f(x)) \geq 0}
        \]
        Remembering the definition of the spectral gap in \eqref{eq: specgap def}, with the choices in \eqref{eq: Compactness - notation translation to Compactness paper} we clearly have
        \[
       \mathrm{specgap}(\RabinowitzHZero,\CCrit) = \min\{ \mathfrak{S}^-_{f,G}(Z), \, \mathfrak{S}^+_{f,G}(Z)\} \,\, .
        \]
        By \eqref{eq: Proof MThm - eps < tau(alpha)} and Lemma \ref{lem: Crit(RabinowitzHZero)}, we indeed have 
        \[
        \eps < \tau(\alpha) = \mathrm{specgap}(\RabinowitzHZero,\CCrit) \leq \mathfrak{S}^\pm_{f,G}(Z) \,\, .
        \]
    
        \textit{First Assumption:}
       Assumption (A1'$_{(E_0,I,I_1)}$) in \cite[Subsec.~2.5]{Compact_Moduli_Spaces_Gradient_Floer_Lines} states that, for every sequence $(\widehat{w}_n)_n \subseteq \Gamma(I)$ with energy $E(\widehat{w}_n)$ uniformly bounded by $E_0 = \eps \,$, a subsequence converges in $C^0_{\mathrm{loc}}(I_1,\loopspace \times \R)$ to some $\widehat{w} \in \Gamma(I_1) \,$. Using the version of Gromov compactness for gradient flow lines on an infinite half-open interval (Proposition \ref{prop: Gromov Compactness}, Remarks \ref{rmk: Gromov compactness - limit gradient flow line} and \ref{rmk: Gromov compactness - half-open interval}) we see that assumption (A1'$_{(E_0,I,I_1)}$) holds true.

        \textit{Second Assumption:}
        To show that also assumption (A2$_{(E_0,I)}$) holds true, we argue as in \cite[Cor.~4.2]{Compact_Moduli_Spaces_Gradient_Floer_Lines} and verify the hypotheses of \cite[Lemma~2.5]{Compact_Moduli_Spaces_Gradient_Floer_Lines}. A coarse metric on $\loopspace \times \R \,$, \ie a distance function whose induced topology is coarser than the ($\Cinfty \times \R$)-topology, is given by the sum of the $C^0$-distance of loops in $\loopspace$ and the standard distance on $\R \,$, that is
        \[
        (\loopspace \times \R) \times (\loopspace \times \R) \longrightarrow \R_{\geq 0} \, , \quad \big( (v_1,\tau_1) \comma (v_2,\tau_2) \big) \mapsto d_{C^0}(v_1,v_2) + |\tau_1 - \tau_2| \,\, ,
        \]
        For given $(v,\tau(\alpha)) \in \CCrit$ we choose $\kappa > 0 \,$, a ($\Cinfty \times \R$)-open neighborhood $\mathscr{U} \subseteq \loopspace \times \R$ and a ($\Cinfty \times \R$)-continuous function $\Xi : \mathscr{U} \rightarrow \R_{\geq 0} \,$, vanishing on $\mathscr{U} \cap \CCrit \,$, as in Proposition \ref{prop: Uniform Exponential Decay}. Given $I_0 = [s_0,\infty) \subseteq I$ and 
        $w = (u,\eta) \in \Cinfty(I_0, \mathscr{U}) \cap \Gamma(I_0)$ with existing limit $w(\infty) \in \CCrit \cap \mathscr{U} \,$, because $w$ is an $\RabinowitzHZero$-gradient flow line and by choice of $\mathscr{U}$ and $\Xi \,$, the inequalities \eqref{eq: Uniform Exponential Decay - d(u, u_infty)} and \eqref{eq: Uniform Exponential Decay - |eta - eta(infty)|} hold. This means
        \begin{align*}
            d_{C^0}(u(s_0,\,\cdot\,),u(\infty,\,\cdot\,)) + |\eta(s_0) - \eta(\infty)| \leq \tfrac{2}{\kappa} \, \Xi(w(s_0)) \, e^{-\kappa(s_0-s_0)} = \tfrac{2}{\kappa} \, \Xi(w(s_0)) \,\, .
        \end{align*}
        Hence the function $\tfrac{2}{\kappa}\,\Xi : \mathscr{U} \rightarrow \R_{\geq 0}$ has all the properties required in \cite[Lemma~2.5]{Compact_Moduli_Spaces_Gradient_Floer_Lines}. This shows that the hypotheses of \cite[Lemma~2.5]{Compact_Moduli_Spaces_Gradient_Floer_Lines} are satisfied, so that (A2$_{(E_0,I)}$) holds as well.

    \textit{Conclusion:}
        We have successfully verified all assumptions of \cite[Thm.~2.12]{Compact_Moduli_Spaces_Gradient_Floer_Lines}. Applying this theorem to the sequence $(w_n|_{[R+1,\infty)})_n \subseteq \Gamma([R+1,\infty)) $, having energy uniformly bounded by $\eps$ and existing limits $w_n(\infty) \in \CCrit \,$, we extract a subsequence converging in $C^0([R+2,\infty), \loopspace \times \R)$ to some $w \in \Gamma([R+2,\infty))$ with existing limit $w(\infty) \in \CCrit $ in the ($\Cinfty \times \R$)-topology. By uniqueness of the limit, $w$ must agree with the ($\Cinftyloc(\R \times \bbS^1,W) \times \Cinftyloc(\R,\R)$)-limit $\Bar{w}$ on $[R+2,\infty)\,$, which finishes the proof of the claim.
\end{proof}

\noindent The compactness of $\moduliH{[0,R]}$ in the subspace topology of $\BcalHat = [0,R] \times \Bcal$ now follows from the previous lemma and the exponential decay estimates in Proposition \ref{prop: Uniform Exponential Decay}.

\begin{proof}[Proof of Proposition \ref{prop: Non_Empty Moduli Spaces - Compactness}]
    Since $\BcalHat$ is metrizable, it suffices to show that $\moduliH{[0,R]}$ is sequentially compact. Let $(r_n,w_n)_n \subseteq \moduliH{[0,R]}$ be an arbitrary sequence. By Lemma \ref{lem: Non_Empty Moduli Spaces - Compactness C^0(loop)-convergence}, a subsequence (suppressed in the notation) converges both in $[0,R] \times \Cinftyloc(\R \times \bbS^1) \times \Cinftyloc(\R)$ and in $[0,R] \times C^0(\Rbar,\loopspace \times \R)$ to some $(r,w) \in \moduliH{[0,R]} \,$. It remains to show that also $w_n \overset{n\rightarrow \infty}{\longrightarrow} w$ in $\Bcal \,$. To this end, fix a chart $f_*^{-1}$ for $\Bcal$ centered at $w = (u,\eta) \,$, as constructed in Appendix \ref{app subsec: Bcal' Atlas}. Since the domain of $f_*^{-1}$ is $C^0$-open (see Remark \ref{app rmk: Bcal' Atlas - f_* C^0-open image}), after discarding finitely many terms, we may assume that the $w_n$ all lie in the chart domain, so that we can write
    \[
    w_n = (u_n,\eta_n) = f_*(\xi_n, x^-_n, x^+_n, \eta_n - \tau(\alpha)) 
    \]
    with $\xi_n \in W^{1,p}_{\kappa_0}(\R \times \bbS^1,u^*TW)$ and $x_n^\pm \in \R^{2N-1} \,$. It follows from $u_n(\pm \infty,0) \convergence{n} u(\pm \infty,0)$ in $W$ and \eqref{app eq: Bcal' Atlas - u(pm infty,t) = (t,0) + x^pm} that $x_n^\pm \convergence{n} 0 \,$. By \eqref{eq: Analytic Set-Up - choice kappa_0 exponential decay} we may choose a constant $\kappa > \kappa_0 \,$, a $(\Cinfty \times\R)$-open neighborhood $\mathscr{U} \subseteq \loopspace \times \R$ of $w(+\infty) \in \CCrit$ and a function $\Xi : \mathscr{U} \rightarrow \R_{\geq 0}$ as in Proposition \ref{prop: Uniform Exponential Decay}. As $w_n \convergence{n} w $ in $C^0(\Rbar,\loopspace \times \R)$, there exists some sufficiently large $s_0 \geq R+1$ so that $w(\Rbar_{\geq s_0}) \subseteq \mathscr{U}$ as well as (after discarding finitely many terms) $w_n(\Rbar_{\geq s_0}) \subseteq \mathscr{U}$ for every $n\,$. In particular, we may apply the estimates \eqref{eq: Uniform Exponential Decay - partial_s}, \eqref{eq: Uniform Exponential Decay - d(u, u_infty)} and \eqref{eq: Uniform Exponential Decay - |eta - eta(infty)|} to each $w_n$ and $w \,$.
    
    From $\eta_n \convergence{n} \eta$ in $\Cinftyloc(\R,\R)$, the exponential decay estimate
    \[
    |\partial_s \eta(s)| + |\partial_s \eta_n(s)| + |\eta(s) - \tau(\alpha)| + |\eta_n(s) - \tau(\alpha)| \leq C \, e^{-\kappa(s-s_0)} \qquad \Forall s \geq s_0 
    \]
    with constant $C >0$ independent of $n$ (by \eqref{eq: Uniform Exponential Decay - partial_s}, \eqref{eq: Uniform Exponential Decay - |eta - eta(infty)|}, Remark \ref{rmk: Uniform Exponential Decay} \ref{it: Uniform Exponential Decay - Xi bounded}) and dominated convergence it now follows that $\eta_n - \tau(\alpha )\convergence{n} \eta - \tau(\alpha)$ in $W^{1,p}_{\kappa_0}(\R_{\geq s_0},\R) \,$.

    It remains to show that $\xi_n \convergence{n} 0$ in $W^{1,p}_{\kappa_0}(\R \times \bbS^1,u^*TW) \,$. From $u_n \convergence{n} u$ in $\Cinftyloc(\R \times \bbS^1,W)$ and $x_n^\pm \convergence{n} 0$ we conclude that 
    \begin{equation}
    \label{eq: Compactness - xi_n Cloc convergence to zero}
    \xi_n \convergence{n} 0 \quad \text{ in } \Cinftyloc(\R \times \bbS^1,u^*TW) \,\, .
    \end{equation}
    Let $\varphi^+ : U^+ \rightarrow \bbS^1 \times \R^{2N-1}$ be a tubular neighborhood around $u(\pm \infty, \,\cdot\,)$ as in Lemma \ref{lem: Analytic Set-Up - tubular nbhd}, corresponding to the choice of chart $f_*^{-1}$ for $\Bcal \,$. After possibly choosing $s_0$ larger, we may assume that $u(\Rbar_{\geq s_0} \times \bbS^1) \subseteq U^+$ and $u_n(\Rbar_{\geq s_0} \times \bbS^1) \subseteq U^+$ for every $n \,$.
    We will show that $\xi_n^{\varphi^+} := \d_{u} \varphi^+ [\xi_n]$ tends to zero in $W^{1,p}_{\kappa_0}(\R_{s_0} \times \bbS^1,\R^{2N}) \,$. For some constant $C > 0 \,$, independent of $n$, we have
    \begin{align}
    \label{eq: Compactness - estimates xi_n^phi}
    \begin{cases}
         |\xi_n^{\varphi^+}(s,t)| \leq C e^{- \kappa s} \\
         |\partial_s \xi_n^{\varphi^+}(s,t)| \leq C e^{- \kappa s} \\
        | \partial_t \xi_n^{\varphi^+}(s,t)| \leq C e^{- \kappa s} & \Forall s \geq s_0 \comma t \in \bbS^1 \,\, .
    \end{cases}
    \end{align}
    Indeed, the first estimate holds by \eqref{eq: Uniform Exponential Decay - d(u, u_infty)}, the second by \eqref{eq: Uniform Exponential Decay - partial_s} and the third one by the first two estimates and the first component of the gradient equation \eqref{eq: s-dependent gradient eq}. Details are carried out in Appendix \ref{app subsec: Supplement Compactness Moduli Space in BcalHat}. The function $C e^{-(\kappa - \kappa_0) s} \in L^p(\R_{\geq s_0} \times \bbS^1)$ is thus a dominating function for $e^{\kappa_0 s} \, |\xi_n^{\varphi^+}| \,$, so that $e^{\kappa_0 s} \, |\xi_n^{\varphi^+}| \convergence{n} 0$ in $L^p$ by pointwise convergence to the zero function (see \eqref{eq: Compactness - xi_n Cloc convergence to zero}) and the dominated convergence theorem. In other words $\xi_n^{\varphi^+} \convergence{n} 0$ in $L^p_{\kappa_0} \,$. Similarly the derivatives of $\xi_n^{\varphi^+}$ tend to the zero function in $L^p_{\kappa_0} \,$, so that altogether $\xi_n^{\varphi^+} \convergence{n} 0$ in $W^{1,p}_{\kappa_0} (\R_{\geq s_0} \times \bbS^1,\R^{2N}) \,$.

    Proving analogous statements near minus infinity, we conclude that $\eta_n - \tau(\alpha) \convergence{n} \eta - \tau(\alpha)$ in $W^{1,p}_{\kappa_0}(\R,\R)$ and that $\xi_n \convergence{n} 0$ in $W^{1,p}_{\kappa_0}(\R \times \bbS^1,u^*TW) \,$. This shows $f_*^{-1}(w_n ) \convergence{n} f_*^{-1}(w)$ and hence $w_n \convergence{n} w$ in $\Bcal \,$.
\end{proof}

\subsection{Fredholm property}
\label{subsec: Non_Empty Moduli Spaces - Fredholm}

We now prove Proposition \ref{prop: Non_Empty Moduli Spaces - Fredholm}, that is that $\Fsection_r$ is a Fredholm section of index $\dim(\Sigma) = 2N-1 \,$. Because we want to avoid additional topological assumptions on $\Sigma \,$, we do not invoke the index formula in \cite[Prop.~4.1]{Cieliebak_Frauenfelder}. Instead, we will use that the index for $r=0$ on $\moduliH{0}$ can readily be computed via the spectral flow, that the index is locally constant and that $\Bcal $ is connected. Making this precise will require a bit of work.

\begin{convention}
Throughout this subsection, we will work with the Riemannian metric $g := \omega(\,\cdot\,,J\,\cdot\,)$ on $W = S \Sigma \,$. Appearing exponential maps or parallel transport are always understood with respect to $g$.
\end{convention}

\begin{definition}
\label{def: moduliB(w) and moduliB^infty}
Recall the smooth structure on $\Rbar = [-\infty ,\infty ] $ defined by the diffeomorphism in \eqref{eq: Rbar smooth structure}. We will write
\[
\Bcal^\infty := \Bcal \cap \big(\Cinfty(\Rbar \times \bbS^1,W) \times \Cinfty(\Rbar,\R) \big) \,\, .
\]
Considering the evaluations $\ev_{\pm}: \Bcal \rightarrow \CCrit \comma w \mapsto w(\pm \infty) \,$, for every $w \in \Bcal$ we denote by $\Bcal(w) \subseteq \Bcal$ the Banach submanifold
\[
\Bcal(w) := (\ev_{-})^{-1}(w(-\infty)) \cap (\ev_{+})^{-1}(w(+\infty)) \,\, .
\]
(This is a submanifold since $(\ev_-, \ev_+) : \Bcal \rightarrow \CCrit \times \CCrit$ is a submersion.)
\end{definition}

Observe that $\moduliH{r} \subseteq \Bcal^\infty$ for every $r \geq 0$ because elements in $\moduliH{r}$ exhibit exponential convergence to its limits in derivatives of every degree (being asymptotically gradient flow lines of the Morse-Bott functional $\RabinowitzHZero$) and hence extend smoothly to $\Rbar \,$.

For every $w = (u,\eta) \in \Bcal^\infty$ we can canonically identify the tangent space of $\Bcal(w)$ at $w$ with
\[
T_w \Bcal(w) = W^{1,p}_{\kappa_0}(\R \times \bbS^1,u^*TW) \oplus W^{1,p}_{\kappa_0}(\R,\R) \,\, .
\]
The tangent spaces fit into a split exact sequence of Banach spaces
\begin{align*}
    0 \longrightarrow  T_w \Bcal(w) \longrightarrow T_w \Bcal \longrightarrow T_{w(-\infty)} \CCrit \oplus T_{w(+\infty)} \CCrit \longrightarrow 0 \,\, ,
\end{align*}
where the first map is the differential of the inclusion and the second map the differential of $(\ev_-,\ev_+) \,$. Thus, we have a non-canonical splitting
\begin{align}
\label{eq: T_w B = T_w B(w) oplus fin dim}
    T_w \Bcal \cong T_w \Bcal(w) \oplus \big( T_{w(-\infty)} \CCrit \oplus T_{w(+\infty)} \CCrit \big) \cong T_w \Bcal(w) \oplus \R^{2\dim(\Sigma)} \,\, .
\end{align}
Since the second summand in the splitting \eqref{eq: T_w B = T_w B(w) oplus fin dim} is finite-dimensional, the vertical differential $\Dvert \Fsection_r (w) : T_w \Bcal \rightarrow \Ecal_w$ at $w \in \moduliH{r}$ is Fredholm iff the restriction 
\[
\Dvert (\Fsection_r |_{\Bcal(w)})(w) = \Dvert \Fsection_r (w)|_{T_w \Bcal(w)} : T_w \Bcal(w) \rightarrow \Ecal_w
\]
is Fredholm. If this is the case, then
\begin{equation}
\label{eq: ind Dvert Fsection_r = ind Dvert Fsection|_B(w) + 2 dim Sigma}
\mathrm{ind}(\Dvert \Fsection_r(w)) = \mathrm{ind}(\Dvert \Fsection_r(w) |_{T_w \Bcal(w)}  ) + 2 \dim(\Sigma) \,\, .
\end{equation}

\noindent It is convenient to introduce a derivative of a section which extends the vertical differential whenever defined.

\begin{definition}
\label{def: derivative D}
We define the $\D$-linearization.
\begin{enumerate}[label=(\alph*)]
    \item Given $w_0 = (u_0,\eta_0) \in \Bcal^\infty$ and a smooth section $\mathcal{G} : \Bcal(w_0) \rightarrow \Ecal|_{\Bcal(w_0)} \,$. Let $\Phi$ be the trivialization of $\Ecal|_{\Bcal(w_0)}$ given by parallel transport, that is at $w = (\exp_{u_0}(\xi),\eta) \in \Bcal(w_0)$ close to $w_0$ the map on the fiber is given by
\begin{align*}
    &\Phi^{-1}_{w} : \Ecal_{w_0} = L^p_{\kappa_0}(u_0^*TW) \oplus L^p_{\kappa_0}(\R,\R) \overset{\cong}{\longrightarrow} \Ecal_w =  L^p_{\kappa_0}(u^*TW) \oplus L^p_{\kappa_0}(\R,\R) \\
    &\Phi^{-1}_{w} \big[ (\mathfrak{u},\mathfrak{n}) \big] = (\widetilde{\mathfrak{u}} , \mathfrak{n}) \comma \text{ where } \widetilde{\mathfrak{u}}(s,t) := \mathrm{P}_{(s,t,\xi(s,t))} \big[ \mathfrak{u}(s,t) \big] \,\, .
\end{align*}
Here $\mathrm{P}_{(s,t,\chi)} : T_{u_0(s,t)} W \overset{\cong}{\longrightarrow} T_{\exp_{u_0(s,t)}(\chi)}W \,$, for any $\chi \in T_{u_0(s,t)} W \,$, denotes parallel transport with respect to $g = \omega(\,\cdot\,,J\,\cdot\,)$ along the path $[0,1] \rightarrow W \comma \sigma \mapsto \exp_{u_0(s,t)}(\sigma \chi) \,$. Define the bounded linear map
\[
\D \mathcal{G} (w_0) : \, T_{w_0} \Bcal(w_0) \longrightarrow \Ecal_{w_0}
\]
as the differential at $w_0$ of the smooth map
\[
\Bcal(w_0) \longrightarrow \Ecal_{w_0} \comma \quad w \mapsto \Phi_{w} \big[ \mathcal{G}(w) \big]\,\, ,
\]
taking values in the Banach space $\Ecal_{w_0} \,$.
\item If $\mathcal{G} : \Bcal \rightarrow \Ecal$ is a section on $\Bcal$, consider the restricted section $\mathcal{G}|_{\Bcal(w_0)} : \Bcal(w_0) \rightarrow \Ecal|_{\Bcal(w_0)}$ and set
\[
\D \mathcal{G} (w_0) := \D (\mathcal{G}|_{\Bcal(w_0)})(w_0) \, : \,\, T_{w_0} \Bcal(w_0) \longrightarrow \Ecal_{w_0} \,\, .
\]
\end{enumerate}
\end{definition}

\begin{remark}
\label{rmk: def derivative D}
Observe that, if $w_0$ is a zero of $\mathcal{G} :\Bcal \rightarrow \Ecal$, then the just defined operator $\D \mathcal{G}(w_0)$ agrees with the vertical differential at $w_0$ of $\mathcal{G} |_{\Bcal(w_0)} \,$, that is 
\[
\D \mathcal{G} (w_0) = \Dvert (\mathcal{G}|_{\Bcal(w_0)})(w_0) = (\Dvert \mathcal{G} (w_0))|_{T_{w_0}\Bcal(w_0)} \,\, .
\]
\end{remark}

\noindent The crucial point regarding the previous definition is that for each $w_0$ we now compute the differential of $\mathcal{G}$ at $w_0$ in a distinguished trivialization, which only depends on $w_0 $ and no further auxiliary choices.

\begin{lemma}
\label{lem: Fredholm D Fsection_0}
For every $w \in \Bcal^\infty$ the operator $\D \Fsection_0 (w) : T_w \Bcal(w) \rightarrow \Ecal_w$ is Fredholm of Fredholm index $\mathrm{ind}(\D \Fsection_0(w) ) = - \dim(\Sigma) = - (2N-1) \,$.
\end{lemma}

\noindent We will prove the lemma in the following. Before that we show how the lemma immediately entails Proposition \ref{prop: Non_Empty Moduli Spaces - Fredholm}.

\begin{proof}[Proof of Proposition \ref{prop: Non_Empty Moduli Spaces - Fredholm}]
    Given $0 \leq r \leq R$ and a zero $w_0 = (u_0,\eta_0) \in \moduliH{r} \subseteq \Bcal^\infty$ of $\Fsection_r \,$. Because
    \[
   (\Fsection_r - \Fsection_0)(u,\eta) = \beta_r { { -\eta \, J(u) (X_H - X_{H_0})(u) }\choose {\int_0^1  (H - H_0)(u) \, \d t } }
    \]
    does not differentiate $(u,\eta)$ and since $\supp (\beta_r) \subseteq I_{R+1} := [-R-1,R+1] \,$, the operator $\D(\Fsection_r - \Fsection_0)(w_0) = \D \Fsection_r (w_0) - \D \Fsection_0(w_0)$ factors continuously through 
    \[
    W^{1,p}( I_{R+1} \times \bbS^1, \,u_0^*TW) \oplus W^{1,p}(I_{R+1},\,\R)
    \]
    and hence is compact. Combining this with Lemma \ref{lem: Fredholm D Fsection_0}, we deduce that $ \D \Fsection_r(w_0) = \Dvert \Fsection_r (w_0) |_{T_{w_0}\Bcal(w_0)}$ is a Fredholm operator of index $-\dim(\Sigma) \,$. The proposition follows from our preliminary considerations around equation \eqref{eq: ind Dvert Fsection_r = ind Dvert Fsection|_B(w) + 2 dim Sigma}.
\end{proof}

\noindent It remains to prove Lemma \ref{lem: Fredholm D Fsection_0}. We first revisit the covariant Hessian of the Rabinowitz action functional $\RabinowitzHZero\,$, then prove the Fredholm property and finally show the index formula.

\subsubsection{Covariant Hessian of $\RabinowitzHZero$}
\label{subsubsec: Fredholm - Covariant Hessian}

Let us temporarily write $\mathcal{T}$ for the Banach bundle over $H^{1}(\bbS^1,W) \times \R$ whose fiber at $(v,\tau)$ is given by
\[
\mathcal{T}_{(v,\tau)} = L^2(\bbS^1,v^*TW) \oplus \R \,\, .
\]
Then $\nabla \RabinowitzHZero : H^1(\bbS^1,W) \times \R \rightarrow \mathcal{T}$ is a smooth section. The \textit{covariant Hessian} 
\[
\nabla^2 \RabinowitzHZero(v,\tau) \, : \,H^1(\bbS^1,v^*TW) \oplus \R \longrightarrow L^2(\bbS^1,v^*TW) \oplus \R
\]
at $(v,\tau) \in \Cinfty(\bbS^1,W) \times \R$ is the differential at $(v,\tau)$ of the trivialized section $\nabla \RabinowitzHZero$ in a trivialization of $\mathcal{T}$ determined by $g$-parallel transport along paths $\sigma \mapsto \exp_{v}(\sigma \xi) $ (entirely analogous to Definition \ref{def: derivative D}). We record that $\nabla^2 \RabinowitzHZero(v,\tau)$ is selfadjoint as unbounded operator on $\big( L^2(\bbS^1,v^*TW) \oplus \R, \, (g_{J})_{(v,\tau)} \big) \,$, so that the spectrum of $\nabla^2 \RabinowitzHZero(v,\tau)$ is real, countable, accumulates only at infinity and consists of eigenvalues only. In particular, $\nabla^2 \RabinowitzHZero(v,\tau) : H^1(\bbS^1,v^*TW) \oplus \R \rightarrow L^2(\bbS^1,v^*TW) \oplus \R$ is a Fredholm operator of index zero.\footnote{
Since the spectrum is discrete, we can choose $\mu \in \R$ in the resolvent set of $A:= \nabla^2 \RabinowitzHZero(v,\tau) \,$. Then $A - \mu \, \id$ is bijective, hence Fredholm of index zero. Because the inclusion $H^1(\bbS^1,v^*TW) \oplus \R \hookrightarrow L^2(\bbS^1,v^*TW) \oplus \R$ is compact, $A$ is a compact perturbation of $A-\mu \, \id$ and thus also Fredholm of index zero.
}

Considering the Rabinowitz action functional as Lagrange multiplier functional, as detailed in \cite[App.~C]{Cieliebak_Frauenfelder}, we can compare the covariant Hessian of $\RabinowitzHZero$ at $(v,\tau)$ with the covariant Hessian of the classical action functional $\nabla^2 \ActionClass{\tau H_0} (v) : H^1(\bbS^1,v^*TW) \rightarrow L^2(\bbS^1,v^*TW) \,$. This gives
\begin{align}
    &\nabla^2 \RabinowitzHZero(v,\tau) : \, H^1(\bbS^1,v^*TW) \oplus \R \longrightarrow L^2(\bbS^1,v^*TW) \oplus \R \comma \notag\\
    &\nabla^2 \RabinowitzHZero(v,\tau) \cdot { {\mathfrak{v}} \choose {\mathfrak{t}} } = \left(\begin{matrix}
   \nabla^2 \ActionClass{\tau H_0}(v) \cdot \mathfrak{v}   + \mathfrak{t} \, J(v) X_{H_0}(v)
    \\
    -\int_0^1 \d H_0(v) \cdot \mathfrak{v} \, \d t 
    \end{matrix}\right) \,\, , \label{eq: Fredholm - Hessian RabinowitzHZero}
\end{align}
the vector field $J(v) X_{H_0}(v) = - \nabla H_0(v)$ along $v$ being the adjoint of $-\int_0^1 \d H_0(v) \, \d t \,$.
We denote by 
\[
J_0 = \left( \begin{matrix}
   0 & - \mathbbm{1} \\
   \mathbbm{1} & 0
\end{matrix} \right) \quad \text{and} \quad \omega_0 = \sum_{j=1}^{N} \d x_j \wedge \d y_j
\]
the standard almost complex structure respectively the standard symplectic form on $\R^{2N}$. Note that $\omega_0(\,\cdot\,, J_0\,\cdot\,)$ is the Euclidean inner product on $\R^{2N} \,$.
It is well-known, see \cite[\S3.2.3]{Weber} for example, that in a fixed unitary trivialization
\[
\Psi : (v^*TW, J ,\omega) \overset{\cong}{\longrightarrow} (\bbS^1 \times \R^{2N} , \, J_0 , \, \omega_0)
\]
the trivialized covariant Hessian
\begin{align*}
    \nabla^2_{\Psi} \ActionClass{\tau H_0} (v) : \, H^1(\bbS^1,\R^{2N}) \longrightarrow L^2(\bbS^1,\R^{2N}) \comma \,\, \nabla^2_{\Psi} \ActionClass{\tau H_0} (v) := \Psi_* \circ \nabla^2 \ActionClass{\tau H_0}(v) \circ \Psi_*^{-1} 
\end{align*}
is of the form
\begin{equation}
\label{eq: Fredholm - Hessian ActionClass trivialized}
    \nabla^2_{\Psi} \ActionClass{\tau H_0} (v) \cdot \zeta = - J_0 \, \partial_t \zeta - S(t) \, \zeta \,\, ,
\end{equation}
where $S(t) \in \R^{2N \times 2N}$ is the smooth loop of symmetric matrices given by
\begin{equation}
\label{eq: Fredholm - Hessian ActionClass trivialized S}
S(t) \, x = \Psi_t \big( J(v) \, (\nabla_t \Psi^{-1} ) \, x + (\nabla_{\Psi_t^{-1} x} J) \, \partial_t v + \tau \, \nabla_{\Psi_t^{-1} x} \nabla H_0 \big)  \comma \quad x \in \R^{2N} \, .
\end{equation}
Owing to \eqref{eq: Fredholm - Hessian RabinowitzHZero} and \eqref{eq: Fredholm - Hessian ActionClass trivialized}, we can now write the trivialized Hessian of the Rabinowitz action functional as
\begin{align}
\label{eq: Fredholm - Hessian RabinowitzHerzo trivialized}
\nabla^2_\Psi \RabinowitzHZero (v,\tau) \cdot { \zeta \choose \mathfrak{t}}= - \left( \begin{matrix}
 J_0 \, \partial_t \zeta + S(t) \, \zeta + \mathfrak{t} \, R(t) \\
\int_0^1 \langle b(t) , \zeta(t) \rangle \, \d t
\end{matrix} \right) \comma \quad { \zeta \choose \mathfrak{t}} \in H^1(\bbS^1,\R^{2N}) \oplus \R
\end{align}
where $S$ is defined in \eqref{eq: Fredholm - Hessian ActionClass trivialized S} and $R(t) := -\Psi_t (J(v) \, X_{H_0}(v)) \in \R^{2N}$ as well as $b(t) := \d H_0(v) \circ \Psi_t^{-1} \in \mathrm{Hom}(\R^{2N},\R) \cong \R^{2N} $ are vector-valued loops. Here $\langle \cdot ,\cdot\rangle$ denotes the Euclidean inner product on $\R^{2N}$.

\subsubsection{$\D \Fsection_0$ is Fredholm}
\label{subsubsec: Fredholm - D Fsection_0 Fredholm}

The $\D$-linearization (Definition \ref{def: derivative D}) of the section $\Fsection_0 = \partial_s - \nabla \RabinowitzHZero$ at $w = (u,\eta) \in \Bcal^\infty$ is given by
\begin{align}
    &\D \Fsection_0 (w) : \, W_{\kappa_0}^{1,p}(\R \times \bbS^1,u^*TW) \oplus W^{1,p}_{\kappa_0}(\R,\R) \longrightarrow  L_{\kappa_0}^{p}(\R \times \bbS^1,u^*TW) \oplus L^{p}_{\kappa_0}(\R,\R) \notag\\
    &\D \Fsection_0 (w) \cdot { \mathfrak{u} \choose \mathfrak{n} } = { \nabla_s \mathfrak{u} \choose  \partial_s \mathfrak{n} } - \nabla^2 \RabinowitzHZero(w(s)) \cdot { \mathfrak{u}(s,\cdot\,) \choose \mathfrak{n}(s) } \,\, .   \label{eq: Fredholm - D-linearization Fsection_0}
\end{align}
Fix a smooth unitary trivialization
\[
\Psi : (u^*TW, J, \omega) \overset{\cong}{\longrightarrow} \big( (\Rbar \times \bbS^1) \times \R^{2N}, \, J_0 , \, \omega_0 \big) 
\]
over $\Rbar \times \bbS^1 $ and conjugate the $\D$-linearization with the induced pushforward maps
\begin{align}
\label{eq: Fredholm - Dtilde Fsection_0 trivialized}
    \Psi_* \circ  \D \Fsection_0 (w)  \circ \Psi_*^{-1} = { \partial_s + \Gamma(s,t) \choose  \partial_s} - \nabla^2_{\Psi_{(s,\,\cdot\,)}} \RabinowitzHZero(w(s))  \,\, ,
\end{align}
where $\Gamma(s,t) \in \R^{2N \times 2N}$ is the connection matrix of $u^*\nabla$ for trivialization $\Psi$ evaluated along the vector field $\partial_s \,$. Note that $\Gamma(s,t) \overset{s \rightarrow \pm \infty}{\longrightarrow} 0$ uniformly in $t \in \bbS^1$ since, due to the smooth structure on $\Rbar$ defined by \eqref{eq: Rbar smooth structure}, the smooth vector field $\partial_s$ on $\R \times \bbS^1$ extends to a $C^1$-vector field on $\Rbar \times \bbS^1 $ with $\partial_s|_{(\pm \infty, t)} = 0 \,$.

Let $I : L^p_{\kappa_0} \rightarrow L^p$ be the isometry given by $I(f) := e^{\kappa_0 \rho(s) s} f \,$, where we remind the reader that $\rho \in \Cinfty(\R,[-1,1])$ is a smooth increasing function with $\rho(s) = \pm 1$ for $\pm s \geq 1 \,$. The isometry $I$ restricts to an isometry $I_1 : W^{1,p}_{\kappa_0} \rightarrow W^{1,p} \,$. Conjugating the operator in \eqref{eq: Fredholm - Dtilde Fsection_0 trivialized} with these isometries yields the operator
\begin{align*}
     &(I \oplus I) \circ \Psi_* \circ  \D \Fsection_0 (w)  \circ \Psi_*^{-1} \circ (I_1^{-1} \oplus I_1^{-1}) 
    \\
    = \,\, &{ \partial_s + \Gamma(s,t) \choose \partial_s} - \nabla^2_{\Psi_{(s,\,\cdot\,)}} \RabinowitzHZero(w(s)) - \kappa_0 \, (\rho'(s) s + \rho(s)) \, \id \,\, .
\end{align*}
Since $\Gamma(\pm \infty, \,\cdot \,) = 0$ and $\rho'(s) = 0$ for $|s| \geq 1 \,$, this operator is a compact perturbation of the operator
\begin{align}
    &D_{w,\Psi} : \, W^{1,p}(\R \times \bbS^1,\R^{2N}) \oplus W^{1,p}(\R,\R) \longrightarrow    L^{p}(\R \times \bbS^1,\R^{2N}) \oplus L^{p}(\R,\R) \notag\\
    &D_{w,\Psi} := \partial_s - \nabla^2_{\Psi_{(s,\,\cdot\,)}} \RabinowitzHZero(w(s) ) - \kappa_0 \, \rho(s) \, \id \,\, , \label{eq: Fredholm - def D_(w,Psi)}
\end{align}
\cf \cite[Lemma 3.18]{Robbin_Salamon_Spectral_Flow}.

\begin{lemma}
\label{lem: Fredholm - ind D Fsection_0 = ind D_(w,Psi)}
For $w = (u,\eta) \in \Bcal^\infty$ and a smooth unitary trivialization $\Psi$ of $u^*TW$ over $\Rbar \times \bbS^1 \,$, the operators $\D \Fsection_0(w)$ and $D_{w,\Psi}$ are both Fredholm and their Fredholm indices agree with the spectral flow of the path $(A(s))_{s \in \R}$ of operators
\begin{align*}
    &A(s) : \, H^1(\bbS^1,\R^{2N}) \oplus \R \rightarrow L^2(\bbS^1,\R^{2N}) \oplus \R \\
    &A(s) :=  -\nabla^2_{\Psi_{(s,\,\cdot\,)}} \RabinowitzHZero(w(s) ) - \kappa_0 \, \rho(s) \, \id \,\, .
\end{align*}
\end{lemma}

\begin{proof}
   By the previous discussion, it is clear that $D \Fsection_0(w)$ is Fredholm iff $D_{w,\Psi}$ is and, if this is the case, their Fredholm indices agree. Notice that $A(s)$ is Fredholm of index zero and selfadjoint as unbounded operator on $L^2(\bbS^1,\R^{2N})  \oplus \R$ with the Euclidean $L^2$-inner product since $\nabla^2_{\Psi_{(s,\,\cdot\,)}} \RabinowitzHZero(w(s))$ is (see Subsection \ref{subsubsec: Fredholm - Covariant Hessian}). The path $(A(s))_{s \in \R}$ has invertible limits 
\[
A(\pm \infty) =  - \nabla^2_{\Psi_{(\pm \infty,\,\cdot\,)}} \RabinowitzHZero(w(\pm \infty)) \mp \kappa_0 \, \id
\]
due to $\kappa_0 > 0$ satisfying \eqref{eq: Analytic Set-Up - choice kappa_0 spec Hessian}. Hence, by the Robbin-Salamon type result \cite[Thm.~A]{Frauenfelder_Weber_Spectral_Flow}, the operator $D_{w,\Psi} = \partial_s + A(s)$ is Fredholm for $p=2 $ and its Fredholm index agrees with the spectral flow of $(A(s))_{s \in \R} \,$. From the case $p=2$ one obtains the case $p >2 \,$, following the line of arguments in \cite[Ch.~2.3]{Salamon_Lectures_on_Floer_homology}.
\end{proof}

\noindent The previous lemma showed the Fredholm assertion in Lemma \ref{lem: Fredholm D Fsection_0}.

\subsubsection{Index of $\D \Fsection_0$}
\label{subsubsec: Fredholm - index D Fsection_0}

To prove the index formula we will crucially use that $\Bcal$ is a connected Banach manifold. It should come to no surprise that this entails that also $\Bcal^\infty$ is connected, as the next lemma clarifies.
  
\begin{lemma}
\label{lem: Bcal^infty connected}
The subset $\Bcal^\infty$ is path-connected in the subspace topology of $\Bcal \,$.
\end{lemma}

\begin{proof}
    Given arbitrary $w_0 = (u_0,\eta_0) \comma w_1 = (u_1,\eta_1) \in \Bcal^\infty \,$, we construct a path from $w_0$ to $w_1$ within $\Bcal^\infty\,$. Connecting $w_{0}, w_1$ to asymptotically constant elements of $\Bcal^\infty$ with same endpoints via a path in $\Bcal^\infty$ (using that $\Bcal^\infty$ is mapped via a chart for $\Bcal$ to a convex set, namely the intersection of an open ball with the vector space $\Cinfty(\Rbar \times \bbS^1,\R^{2N}) \times \R^{2\dim(\Sigma)}\times \Cinfty(\Rbar,\R)$), we may assume without loss of generality that $w_0$ and $w_1$ are asymptotically constant, say $w_{0,1}(\pm s) = w_{0,1}(\pm \infty)$ for $ s \geq s_0 \,$. As $w_0 , w_1 \in \Bcal \,$, we have
    $u_0 \simeq u_1 \text{ rel } \CCrit \,$. Hence we can find a homotopy $(u_\sigma)_{0 \leq \sigma \leq 1}$ from $u_0$ to $u_1 \,$, \ie a continuous map
    \[
     (\Rbar \times \bbS^1) \times [0,1] \rightarrow W \comma \quad ((s,t),\sigma) \mapsto u_\sigma(s,t) \, ,
    \]
    with
    \begin{equation}
    \label{eq: Bcal^infty connected - homotopy asymptotics}
        (u_\sigma(\pm \infty , \,\cdot\,),\tau(\alpha)) \in \CCrit \quad \Forall \sigma \in [0,1] \, .
    \end{equation}
     Applying the Whitney approximation theorem to the continuous paths $[0,1] \rightarrow \Sigma \comma \sigma \mapsto u_\sigma(\pm \infty,0) \,$,
     we can asymptotically modify $(u_\sigma)_\sigma$ and choose $s_0$ larger so that \eqref{eq: Bcal^infty connected - homotopy asymptotics} still holds, the paths $[0,1] \rightarrow \Sigma \comma \sigma \mapsto u_\sigma(\pm \infty,0) \,$, are smooth
    and $u_\sigma(\pm s,t) = u_\sigma(\pm \infty,t)$ for $s \geq s_0 \,$. It follows that the homotopy $(u_\sigma)_\sigma$ is smooth on $\{\sigma = 0,1\} \cup \{|s| \geq s_0\} \,$. Again by Whitney approximation, there exists a smooth homotopy $(\widetilde{u}_\sigma)_{0 \leq \sigma \leq 1}$ so that
   \begin{align*}
       \widetilde{u}_{0,1} = u_{0,1} \text{ and } \widetilde{u}_{\sigma}(s,t)  = u_\sigma(s,t) \text{ for } |s| \geq s_0 \,\, .
   \end{align*}
  Then $(w_\sigma)_{0 \leq \sigma \leq 1} := (\widetilde{u}_\sigma, \,(1-\sigma) \eta_0 + \sigma \eta_1)_{0 \leq \sigma \leq 1}$ is a path from $w_0$ to $w_1$ in $\Bcal^\infty \,$.  
\end{proof}

\begin{lemma}
\label{lem: Fredholm - ind D Fsection_0 locally constant}
The map
\[
\Bcal^\infty \rightarrow \Z \comma \quad w \mapsto \mathrm{ind}(\D \Fsection_0(w)) \,\, ,
\]
is locally constant, where $\Bcal^\infty$ carries the subspace topology of $\Bcal \,$.
\end{lemma}

\begin{proof}
   Let $w_0 = (u_0,\eta_0) \in \Bcal^\infty$ be arbitrary but fixed. For a sufficiently small, fiberwise convex, open neighborhood $\mathcal{O} \subseteq TW$ of the zero section and small open balls $B^\pm \subseteq \R^{2N-1}$ of the origin respectively, consider a smooth map
   \begin{align*}
       f : u_0^*\mathcal{O} \times B^- \times B^+ \longrightarrow W \,\, ,
   \end{align*}
   as defined in equation \eqref{app eq: Bcal' Atlas - map f} in the appendix. (Such a map depends on various auxiliary choices, which in turn depend on $u_0$.) The associated pushforward map
   \begin{align*}
   &f_* : W^{1,p}_{\kappa_0}(u^*\mathcal{O}) \times B^- \times B^+ \times W^{1,p}_{\kappa_0}(\R,\R) \longrightarrow  \Bcal \\
   &f_*(\xi, x^-,x^+,\widehat{\eta}) := (u, \widehat{\eta} + \tau(\alpha)) \comma \text{ where } u(s,t) := f\big( \xi(s,t), x^-,x^+ \big) \,\, ,
   \end{align*}
   is the inverse of the natural chart for $\Bcal$ centered at $w_0$ associated to $f$ (as constructed in Appendix \ref{app subsubsec: Bcal' Atlas - construction chart}), with chart domain $\Ucal := \mathrm{im}(f_*) \subseteq \Bcal \,$. Notice that $f$ may also depend on $(s,t)$, being hidden in the factor $u_0^*\mathcal{O} \,$.
   Because $u_0^*\mathcal{O} \times B^- \times B^+$ deformation retracts to the zero section of $u_0^*TW \,$, which is homotopy equivalent to $\Rbar \times \bbS^1 \simeq \bbS^1 \,$, the Hermitian bundle $(f^*TW,J,\omega)$ is trivial. Therefore we can choose a smooth unitary trivialization
   \[
   \Psi : \, (f^*TW,J,\omega) \longrightarrow \big( (u_0^*\mathcal{O} \times B^- \times B^+) \times \R^{2N} \comma J_0 , \omega_0 \big) \,\, .
   \]
   Observe that, given $w= (u,\eta) = f_*(\xi,x^-,x^+,\eta- \tau(\alpha)) \in \Ucal \cap \Bcal^\infty \,$, this induces a unitary trivialization $\Psi^u$ of $(u^*TW,J,\omega)$ by setting
   \[
   \Psi^u_{(s,t)} := \Psi_{(\xi(s,t),x^-,x^+)} \, : \, T_{u(s,t)} W \longrightarrow \R^{2N} \,\, .
   \]
   By Lemma \ref{lem: Fredholm - ind D Fsection_0 = ind D_(w,Psi)}, the Fredholm index of $\D \Fsection_0(w)$ agrees with the index of $D_{w, \Psi^u} \,$. We will show that the map
   \begin{align*}
       \Ucal \cap \Bcal^\infty &\longrightarrow \mathcal{L}\big(W^{1,p}(\R \times \bbS^1,\R^{2N}) \oplus W^{1,p}(\R,\R) \comma L^p(\R \times \bbS^1,\R^{2N}) \oplus L^p(\R,\R) \big) \notag \\
       w = (u,\eta) &\longmapsto D_{w,\Psi^u}
   \end{align*}
   to the space of bounded linear operators (with the operator norm topology)
   is continuous at $w_0$. If this has been shown, then the stability of the Fredholm index under small perturbations implies that $w = (u,\eta) \mapsto \mathrm{ind}(D_{w,\Psi^u}) = \mathrm{ind}(\D \Fsection_0(w))$ is locally constant around $w_0 \,$.

   Now let $(w_n)_{n \geq 1} = (u_n,\eta_n)_{n \geq 1} \subseteq \Ucal \cap \Bcal^\infty$ be an arbitrary sequence tending to $w_0$ in $\Bcal \,$. By definition of the operator $D_{w,\Psi^u}$ in \eqref{eq: Fredholm - def D_(w,Psi)}, we have
   \begin{equation*}
   D_{w_n,\Psi^{u_n}} - D_{w_0,\Psi^{u_0}} = \nabla^2_{\Psi^{u_0}} \RabinowitzHZero(w_0) - \nabla^2_{\Psi^{u_n}} \RabinowitzHZero(w_n) \,\, .
   \end{equation*}
   By \eqref{eq: Fredholm - Hessian RabinowitzHerzo trivialized}, we can write the trivialized covariant Hessians as
  \begin{align*}
     \nabla^2_{\Psi^{u_n}} \RabinowitzHZero (w_n) \cdot { \widehat{u} \choose \widehat{\eta}}= - \left( \begin{matrix}
 J_0 \, \partial_t \widehat{u} + S_n \, \widehat{u} + \widehat{\eta} \, R_n \\
\int_0^1 \langle b_n, \widehat{u} \rangle \, \d t
\end{matrix} \right) \,\, ,  
  \end{align*}
  with $S_n \in \Cinfty(\R \times \bbS^1,\R^{2N \times 2N})$ and $R_n , \, h_n \in \Cinfty(\R \times \bbS^1,\R^{2N}) \,$, and similarly for $\nabla^2_{\Psi^{u_0}} \RabinowitzHZero(w_0) \,$. Therefore
  \begin{align*}
   (D_{w_n,\Psi^{u_n}} - D_{w_0,\Psi^{u_0}})    { \widehat{u} \choose \widehat{\eta}}= \left( \begin{matrix}
 (S_n - S_0) \, \widehat{u} + \widehat{\eta} \, (R_n -R_0) \\
\int_0^1 \langle b_n - b_0 , \, \widehat{u} \rangle \, \d t
\end{matrix} \right) \,\, .
  \end{align*}
  The convergence $w_n \overset{n \rightarrow \infty}{\longrightarrow} w_0$ in $\Bcal$ and the definition of the trivializations $\Psi^{u_n}$ above in terms of the common trivialization $\Psi$ imply that $R_n \overset{n \rightarrow \infty}{\longrightarrow} R_0$ and $b_n  \overset{n \rightarrow \infty}{\longrightarrow} b_0$ in $C^0(\Rbar \times \bbS^1)$ respectively. Looking at \eqref{eq: Fredholm - Hessian ActionClass trivialized S}, the difference $S_n - S_0$ can be written as sum of two terms: One of which tends to zero in $L^p_{\kappa_0}(\R \times \bbS^1)\,$, the other one tends to zero in $C^0(\Rbar \times \bbS^1) \,$. (We show this in detail in Appendix \ref{app subsec: Supplement Fredholm index locally constant}.) Then Lemma \ref{lem: Fredholm - auxiliary operator norm estimate} below implies that the operator norm of $D_{w_n,\Psi^{u_n}} - D_{w_0,\Psi^{u_0}}$ tends to zero as $n \rightarrow \infty \,$. This finishes the proof.
\end{proof}

\noindent Let us state the lemma alluded to in the proof above.

\begin{lemma}
\label{lem: Fredholm - auxiliary operator norm estimate}
Given matrix-valued functions $A \in L^p(\R \times \bbS^1,\R^{2N \times 2N}) $, $ B \in L^\infty (\R \times \bbS^1,\R^{2N \times 2N}) $ and vector-valued functions $ R, b  \in L^\infty(\R \times \bbS^1,\R^{2N }) \,$. Then the operator
\begin{align*}
    &L : W^{1,p}(\R \times \bbS^1,\R^{2N}) \oplus W^{1,p}(\R,\R) \longrightarrow L^p(\R \times \bbS^1,\R^{2N}) \oplus L^p(\R,\R) \\
   &L \, { \widehat{u} \choose \widehat{\eta}} := \left( \begin{matrix}
        (A(s,t) + B(s,t)) \, \widehat{u}(s,t) + \widehat{\eta}(s) R(s,t) \\
        \int_0^1 \langle b(s,t) , \widehat{u}(s,t) \rangle \, \d t
    \end{matrix} \right) 
\end{align*}
is a bounded linear operator whose operator norm is bounded by
\begin{align*}
    \lVert L \rVert \leq c \left( \lVert A \rVert_{L^p} + \lVert B \rVert_{L^\infty} + \lVert R \rVert_{L^\infty} + \lVert  b \rVert_{L^\infty}\right) \,\, ,
\end{align*}
with a constant $c > 0$ independent of $A,B, R$ and $b$. 
\end{lemma}

\begin{proof}
    Since $p >2 \,$, the product of functions induces a continuous bilinear map $L^p(\R \times \bbS^1) \times W^{1,p}(\R \times \bbS^1) \rightarrow L^p(\R \times \bbS^1) \,$, which yields an estimate $\lVert A \widehat{u} \rVert_{L^p} \leq c \, \lVert A \rVert_{L^p} \, \lVert \widehat{u} \rVert_{W^{1,p}} \,$. Using Jensen's inequality, moreover
    \begin{align*}
       \int_{-\infty}^{+\infty} \big| \int_0^1 \langle b , \widehat{u} \rangle \, \d t \big|^p \, \d s \leq \int_{-\infty}^{+\infty} \int_0^1 |\langle b , \widehat{u} \rangle|^p \, \d t \, \d s \leq \lVert b \rVert_{L^\infty}^p \, \lVert \widehat{u} \rVert_{L^p}^p \,\, ,
    \end{align*}
    which bounds the $L^p$-norm of the term in the lower line.
\end{proof}

\begin{lemma}
\label{lem: Fredholm - index on moduli^H_0}
For $w \in \moduliH{0} \,$, the operator $\D \Fsection_0(w)$ has Fredholm index $- \dim(\Sigma) \,$.
\end{lemma}

\begin{proof}
    Recall from Lemma \ref{lem: moduli_0^H} that $w$ is $s$-independent, \ie constant. By Lemma \ref{lem: Fredholm - ind D Fsection_0 = ind D_(w,Psi)}, the index of $\D \Fsection_0(w)$ agrees with the spectral flow of the path of operators
    \[
    A(s) = - \nabla^2 \RabinowitzHZero(w) - \kappa_0 \, \rho(s) \, \id \,\, ,
    \]
    \ie with the number of net crossings of eigenvalues of $A(s)$ from negative to positive as $s$ varies from $s=-\infty$ to $s=+\infty\,$.
    Because $\kappa_0 > 0$ is strictly smaller than the smallest absolute value of a non-zero eigenvalue of $\nabla^2 \RabinowitzHZero(w)$ by \eqref{eq: Analytic Set-Up - choice kappa_0 spec Hessian}, thus $\mathrm{ind}(\D \Fsection_0(w)) = - \dim (\ker \nabla^2 \RabinowitzHZero(w)) \,$. 
    Since $\RabinowitzHZero$ is Morse-Bott by Lemma \ref{lem: Crit(RabinowitzHZero)}, the kernel of $\nabla^2 \RabinowitzHZero(w)$ has dimension $\dim(\CCrit) = \dim(\Sigma) \,$.
\end{proof}

\noindent Our preparations allow us to finally finish the proof of Lemma \ref{lem: Fredholm D Fsection_0}.

\begin{proof}[Proof of Lemma \ref{lem: Fredholm D Fsection_0}]
That $\D \Fsection_0 (w)$ is Fredholm is asserted in Lemma \ref{lem: Fredholm - ind D Fsection_0 = ind D_(w,Psi)}. From Lemma \ref{lem: Fredholm - index on moduli^H_0} we know that $\mathrm{ind}(\D \Fsection_0(w)) = - \dim(\Sigma)$ for every $w \in \moduliH{0} \subseteq \Bcal^\infty \,$. Since the index is locally constant on $\Bcal^\infty$ (Lemma \ref{lem: Fredholm - ind D Fsection_0 locally constant}) and $\Bcal^\infty$ is connected (Lemma \ref{lem: Bcal^infty connected}), we conclude that $\mathrm{ind}(\D \Fsection_0 (w)) = -\dim(\Sigma)$ for every $w \in \Bcal^\infty \,$.
\end{proof}

\noindent We have successfully proven Lemma \ref{lem: Fredholm D Fsection_0}, concluding the proof of Proposition \ref{prop: Non_Empty Moduli Spaces - Fredholm}.

\subsection{Transversality for $r=0$}
\label{subsec: Non-Empty Moduli Spaces - Transversality r=0}

The proof of Proposition \ref{prop: Non_Empty Moduli Spaces - Transversality} that we will present in this subsection is inspired but different from the one Fauck gives \cite[App.~C]{FauckThesis}.

We fix $w_0 = (u_0,\eta_0) \in \moduliH{0}$ and recall from Lemma \ref{lem: moduli_0^H} that 
$w_0 $ is $s$-independent, sitting constantly at $\CCrit$. In particular $\eta_0 \equiv \tau(\alpha) \,$. Abusing notation, we will also write $w_0 = (u_0,\tau(\alpha)) \in \CCrit$ for the constant value of $w_0 \in \Cinfty(\R,\loopspace \times \R) \,$. Let us fix a tubular neighborhood $U_0 \cong \bbS^1 \times \R^{2N-1}$ around the loop $u_0$ as in Lemma \ref{lem: Analytic Set-Up - tubular nbhd}. This induces an $s$-independent trivialization $u_0^*TW \cong (\Rbar \times \bbS^1) \times \R^{2N}$ and a chart for $\Bcal$ centered at $w_0 \,$, as in Appendix \ref{app subsubsec: Bcal' Atlas - construction chart}, which allows us to identify a neighborhood of $w_0$ in $\Bcal$ with an open neighborhood of the origin in
\[
T_{w_0} \Bcal = W^{1,p}_{\kappa_0}(\R \times \bbS^1,\R^{2N}) \oplus \R^{2N-1} \oplus \R^{2N-1} \oplus W^{1,p}_{\kappa_0}(\R,\R) \,\, ,
\]
in which $\R^{2N-1} = \R^{2N-1} \times \{0\} \subseteq \R^{2N} \,$. 
Let $\gamma^\pm \in \Cinfty(\R,[0,1])$ be monotone cutoff functions with $\gamma^\pm(s) = 1$ for $\pm s \geq 2$ and $\gamma^\pm(s) = 0$ for $\pm s \leq 1 \,$. Due to \eqref{app eq: Bcal' Atlas - exp^gaux at u_0} and \eqref{app eq: Bcal' Atlas - map f}, a tangent vector $(\xi,x^-,x^+ ,\widehat{\eta}) \in T_{w_0} \Bcal$ may be identified with
\begin{equation}
\label{eq: Transversality - wHat}
\widehat{w} := { {\xi + \gamma^-(s) x^- + \gamma^+(s) x^+ } \choose {  \widehat{\eta}  } } \in W^{1,p}_{\mathrm{loc}}(\R \times \bbS^1,\R^{2N}) \oplus W^{1,p}_{\kappa_0}(\R,\R) \,\, .
\end{equation}
Consider the pair of Hilbert spaces
\[
   W := H^1(\bbS^1,\R^{2N}) \oplus \R \quad \subseteq \quad  H := L^2(\bbS^1,\R^{2N}) \oplus \R  \,\, ,
\]
and abbreviate minus the covariant Hessian of $\RabinowitzHZero$ at $w_0 \in \loopspace \times \R$ by
\[
   A := - \nabla^2 \RabinowitzHZero(w_0) \in \mathcal{L}(W,H)  \,\, .
\]
Then $A$ is an unbounded selfadjoint operator on the Hilbert space $(H,(g_J)_{w_0}) \,$. The inner product $(g_J)_{w_0}$ and the standard inner product on $H$ (using the Euclidean $L^2$-inner product in the first factor) induce equivalent norms.

\begin{lemma}
\label{lem: Transversality - total differential}
The vertical differential of $\Fsection_0 = \partial_s - \nabla \RabinowitzHZero$ at $w_0 \in \moduliH{0}$ in direction $(\xi,x^-,x^+,\widehat{\eta}) \in T_{w_0} \Bcal$ is given by
\[
\Dvert \Fsection_0(w_0) [(\xi, x^-,x^+, \widehat{\eta})] = \partial_s \widehat{w} + A \widehat{w} \,\, ,
\]
with $\widehat{w}$ defined in \eqref{eq: Transversality - wHat}.
\end{lemma}

\begin{proof}
The restriction of $\Dvert \Fsection_0(w_0)$ to the infinite-dimensional component $T_{w_0} \Bcal(w_0) = W^{1,p}_{\kappa_0}(\R \times \bbS^1,\R^{2N}) \oplus W^{1,p}_{\kappa_0}(\R,\R)$ is the $\D$-linearization, as noted in Remark \ref{rmk: def derivative D}. As $u_0$ and the trivialization are $s$-independent, we have $\nabla_s = \partial_s \,$. Hence, formula \eqref{eq: Fredholm - D-linearization Fsection_0} for the $\D$-linearization simplifies to
\[
\Dvert \Fsection_0(w_0) |_{W^{1,p}_{\kappa_0}(\R \times \bbS^1) \oplus W^{1,p}_{\kappa_0}(\R)} = \D \Fsection_0 (w_0) = \partial_s + A \,\, .
\]
Next we compute $\Dvert \Fsection_0(w_0)$ on the finite-dimensional component. Since\footnote{
In the upper entry we suppress the tubular coordinates and consider the image of $x^\pm \in \R^{2N-1} \times \{0\} \subseteq \R^{2N}$ under the natural projection $\R^{2N} \longrightarrow  \R/ \Z \times \R^{2N-1} = \bbS^1 \times \R^{2N-1} \,$.
} 
\[
{u_0 (s,\,\cdot\,) + \gamma^-(s) x^- + \gamma^+(s) x^+\choose \tau(\alpha)} \in \CCrit \qquad \Forall s \in \R \comma \,\, \Forall x^-,x^+ \in \R^{2N-1}
\]
and $- \nabla \RabinowitzHZero$ vanishes identically on $\CCrit \subseteq \Crit(\RabinowitzHZero)$ with vertical differential $A$ at $w_0 \in \CCrit \,$,
the partial derivative in direction of $(x^-,x^+)$ computes to
\begin{align*}
\Dvert \Fsection_0(w_0) [(0,x^-x,^+,0)] = \partial_s { \gamma^- \, x^- + \gamma^+ \, x^+ \choose 0} = (\partial_s + A) \, { \gamma^- \, x^- + \gamma^+ \, x^+ \choose 0} \,\, .
\end{align*}
\end{proof}

\begin{lemma}
\label{lem: Transversality - ker D Fsection_0}
There is a well-defined injective linear map
\begin{equation}
\label{eq: Transversality - ker D Fsection_0 injective map}
\ker \Dvert \Fsection_0(w_0) \hookrightarrow \set{\widehat{w}}{ \widehat{w} \text{ is } s\text{-independent and } \widehat{w} \in \ker A} \,\, ,
\end{equation}
mapping $(\xi,x^-,x^+,\widehat{\eta})$ to $\widehat{w}$ defined in \eqref{eq: Transversality - wHat}. Thus $\dim (\ker \Dvert \Fsection_0(w_0)) \leq \dim(\Sigma) \,$.
\end{lemma}
\begin{proof}
    The kernel of the Hessian $ \nabla^2 \RabinowitzHZero(w_0)  = - A$ has dimension $\dim(\CCrit)=\dim(\Sigma) $ due to $\RabinowitzHZero$ being Morse-Bott (Lemma \ref{lem: Crit(RabinowitzHZero)}), so the target vector space in \eqref{eq: Transversality - ker D Fsection_0 injective map} has dimension $\dim(\Sigma) \,$. It is clear that the assignment $(\xi,x^-,x^+,\widehat{\eta}) \mapsto \widehat{w} \,$, with $\widehat{w}$ defined by \eqref{eq: Transversality - wHat}, is linear. It is also injective since $\widehat{w}(\pm \infty) = (x^\pm,0)\,$. By Lemma \ref{lem: Transversality - total differential}, for $(\xi,x^-,x^+,\widehat{\eta}) \in \ker \Dvert \Fsection_0(w_0)$ it holds
    \begin{equation}
    \label{eq: Transversality - partial_s wHat + A wHat = 0}
    \partial_s \widehat{w}(s) + A \,\widehat{w}(s) = 0  \qquad \Forall s \in \R
    \end{equation}
    and we will show that this implies that $\widehat{w}$ is $s$-independent with $\widehat{w} \in \ker A \,$.
    
   By the spectral theorem, the Hilbert space $H$ decomposes orthogonally
    \[
    H = \ker(A) \oplus E_+ \oplus E_+ \,\, ,
    \]
    where $E_\pm$ denote the positive respectively negative eigenspace of the unbounded selfadjoint operator $A$. Moreover, the corresponding orthogonal projections leave $W= \mathrm{dom}(A)$ invariant and commute with $A$. Define $A_\pm := A|_{E_\pm \cap W}$ and denote by $\widehat{w}_\pm(s) $ the orthogonal projection of $\widehat{w}(s) $ to $E_\pm$ and by $\widehat{w}_0(s)$ the projection to $\ker(A) \,$. Applying the projections to equation \eqref{eq: Transversality - partial_s wHat + A wHat = 0} yields
    \begin{equation}
    \label{eq: Transversality - wHat plusminus eq}
    \partial_s \widehat{w}_\pm + A_\pm \widehat{w}_\pm = 0 \quad \text{and} \quad \partial_s \widehat{w}_0 = \partial_s \widehat{w}_0 + A \widehat{w}_0 = 0 \,\, .
    \end{equation}
    In particular $\widehat{w}_0$ is $s$-independent. We will show $\widehat{w}_\pm = 0 \,$, finishing the proof.

    By Hille-Yosida, see \cite[Thm.~7.3.4]{Functional_Analysis_Salamon}, the negative definite selfadjoint operators $-A_+$ and $A_-$ are infinitesimal generators of strongly continuous semigroups
    \[
    [0,\infty) \longrightarrow \mathcal{L}(E_\pm, E_\pm) \comma \,\, s \mapsto  e^{\mp A_\pm s}
    \]
    with $\lVert e^{\mp A_\pm s} \rVert \leq 1$ for every $s \geq 0 \,$. Hence, for given initial value the homogeneous Cauchy problem for $\mp A_\pm$ has a unique weak solution for positive times, see \cite[Thm.~7.6.2]{Functional_Analysis_Salamon}. By uniqueness and \eqref{eq: Transversality - wHat plusminus eq}, it must hold
    \[
    \widehat{w}_+(s + \sigma) = e^{- A_+ s} \, \widehat{w}_+(\sigma) \qquad \Forall \sigma \in \R \comma s \geq 0 \,\, ,
    \]
    and consequently
    \begin{equation}
    \label{eq: Transversality - semigroup argument}
    |\widehat{w}_+(s+\sigma)|_H \leq \lVert e^{- A_+ s} \rVert_{\mathcal{L}(E_+,E_+)} \, |\widehat{w}_+(\sigma)|_H \leq |\widehat{w}_+(\sigma)|_H \qquad \Forall \sigma \in \R \comma s \geq 0 \,\, .
    \end{equation}
    Now notice, since ${\gamma^\pm(s) x^\pm \choose 0} \in \ker A$ and  $(\xi,\widehat{\eta}) \in W^{1,p}_{\kappa_0}(\R \times \bbS^1) \oplus W^{1,p}_{\kappa_0}(\R) $, we have  $\widehat{w}_+(\pm \infty) =  { \xi \choose \widehat{\eta} }_+(\pm \infty) = 0 \,$. Thus the right-hand side in \eqref{eq: Transversality - semigroup argument} tends to zero as $\sigma \rightarrow -\infty\,$, so that $|\widehat{w}_+(s)|_H = 0$ for every $s \in \R \,$. Hence $\widehat{w}_+ = 0 \,$. A similar argument shows $\widehat{w}_- = 0\,$, using $\widehat{w}_- (-s + \sigma) = e^{A_- s}\,  \widehat{w}_-(\sigma) \,$.
\end{proof}

\begin{remark}
\label{rmk: Transversality - ker D Fsection_0}
Although we do not need it, clearly the map in the preceding lemma is surjective as well. 
\end{remark}

\noindent We can now prove Proposition \ref{prop: Non_Empty Moduli Spaces - Transversality}.

\begin{proof}[Proof of Proposition \ref{prop: Non_Empty Moduli Spaces - Transversality}]
    Let $w_0 \in \Fsection_0^{\,-1}(0)=\moduliH{0}$ be arbitrary. Then $\Dvert \Fsection_0(w_0)$ is a Fredholm operator of index $\dim (\Sigma) \,$, as we know from Proposition \ref{prop: Non_Empty Moduli Spaces - Fredholm}. In particular, its kernel is automatically a complemented subspace. Using that $\Dvert \Fsection_0(w_0)$ has index $\dim(\Sigma)$ and that its kernel has dimension at most $\dim(\Sigma)$ by Lemma \ref{lem: Transversality - ker D Fsection_0}, we conclude that $\Dvert \Fsection_0(w_0)$ is surjective.
\end{proof}

\noindent The proof of Proposition \ref{prop: moduli^H_[0,R](z) non-empty} is now complete.

\subsection{Modifications of the perturbation argument}
\label{subsec: Non-Empty Moduli Spaces - Modifications}

Verifying all hypotheses in Theorem \ref{thm: Non-Empty Moduli Space - abstract perturbation thm} was a cumbersome task. In this concluding subsection, we briefly explain two minor variants of the perturbation argument, in which certain hypotheses can be omitted. (Of course, they each come at a price, see below...) 

\begin{enumerate}[label=(\roman*)]
    \item The assumption in Theorem \ref{thm: Non-Empty Moduli Space - abstract perturbation thm} that the vertical differential of $\Fsection$ has the \textit{same} Fredholm index at every zero is in fact redundant: In the proof of \cite[Thm.~D]{Abstract_Perturbations}, using the notation therein,
     even without this assumption one still obtains a compact manifold $\widetilde{e}^{\,-1}(Z) \subseteq \BcalHat$ of possibly mixed dimension with boundary
    \begin{align*}
    \hspace*{14mm} \partial (\widetilde{e}^{\,-1}(Z)) = (e_0|_{\Fsection_0^{-1}(0)})^{-1}(Z)  \sqcup \widetilde{e}^{\, -1}_R(Z) = \{\mathrm{pt}.\} \sqcup \widetilde{e}^{\, -1}_R(Z)  \,\, .
    \end{align*}
    Since the boundary of a compact manifold (possibly of mixed dimension) cannot consist of a single point, this still shows that $\widetilde{e}_R^{\, -1}(Z)$ is non-empty and the proof of \cite[Thm.~D]{Abstract_Perturbations} goes through.

    This variant of Theorem \ref{thm: Non-Empty Moduli Space - abstract perturbation thm} would allow us to only work with the ambient Banach manifold $\Bcal' \,$, for which the Fredholm property but not the index formula in Proposition \ref{prop: Non_Empty Moduli Spaces - Fredholm} remains true.
    \item We could have applied a similar perturbation theorem by Abbondandolo-Haug-Schlenk \cite[Thm.~A.1]{Action_Selector_SimpleConstruction} with Banach manifold $\Bcal_z := \ev^{-1}(\R \times \{z\}) \subseteq \Bcal$ and Fredholm section $\Fsection|_{[0,R] \times \Bcal_z}$ of index $1$. The hypothesis therein that the bundle $\Ecal|_{\Bcal_z}$ is trivial is satisfied since the general linear group of $L^p_{\kappa_0}(\R \times \bbS^1,\R^{2N}) \oplus L^p_{\kappa_0}(\R , \R)$ is contractible due to a version of Kuiper's theorem, see \cite[Prop.~5]{GL_Banach_contractible}. The benefit here is that one does not need to show extendibility to a B-bundle pair.
\end{enumerate}

\section{Outlook}
\label{sec: Outlook}

\noindent As mentioned right at the beginning in Subsection \ref{subsec: Introduction - Motivation}, it makes perfect sense to ask the analogous question also for general (not necessarily Zoll) contact forms:
\begin{question*}
    Let $(\Sigma,\alpha)$ be a contact manifold and $D \subseteq \Sigma$ be an open smooth domain for which $\alpha$ has a periodic Reeb orbit intersecting $D$. Does every sufficiently $C^0$-small perturbed contact form $e^f \alpha \,$, with $f \equiv 0$ on $\Sigma \backslash D \,$, also have a periodic Reeb orbit intersecting $D \,$?
\end{question*}

\noindent A reasonable approach is to first find an appropriate adaption of Theorem \ref{mthm: Rabinowitz gradient flow line}. In fact, the proof of Theorem \ref{mthm: Rabinowitz gradient flow line} would for the most parts go through if one replaced the Zoll contact assumption with the following weaker Morse-Bott assumption for fixed $\tau > 0 \,$:
\begin{itemize}
    \item[(MB$_\tau$)] The subset $N_\tau \subseteq \Sigma$ of the $\tau$-periodic Reeb orbits is a closed submanifold of $\Sigma \,$, the rank of $\d \alpha|_{N_\tau}$ is locally constant and $T_p N_\tau = \ker(\d \phi_{R_\alpha}^\tau - \id_{T_p \Sigma})$ for every $p \in N_\tau \,$.
\end{itemize}
Notice that, opposed to \cite[Sec.~2.3]{SFT_Compactness} for example, we only require this to hold for fixed $\tau $ instead of for every $\tau \in \R \,$.

The bottleneck is now the final step in the proof of Proposition \ref{prop: moduli^H_[0,R](z) non-empty}: We would like to find a closed submanifold $S \subseteq \Sigma$ of codimension $\mathrm{codim}_\Sigma (S) = \dim(N_\tau) \,$, which intersects $N_\tau$ transversally in a single point (or more generally in an odd number of points), and then apply Theorem \ref{thm: Non-Empty Moduli Space - abstract perturbation thm} with $Z = \R \times S \,$. We call such $S$ a \textit{slice for $N_\tau \,$}. In the Zoll case we have $N_{\tau(\alpha)} = \Sigma$ and can take $S$ to be any singleton of $\Sigma \,$. The existence of a slice $S$ places some rather restrictive topological assumptions on $N_\tau $ and $\Sigma \,$: For example, for the intersection product we have 
\[
[N_\tau] \cdot [S] = 1 \in  \Z / 2 \Z \cong H_0(\Sigma; \Z / 2 \Z) \,\, ,
\]
so that the homology groups $H_*(\Sigma; \Z / 2 \Z)$ in degrees $*=\dim(N_\tau) ,\dim(\Sigma) - \dim(N_\tau)$ must be nonzero. Due to these topological restrictions and since the notation in the Zoll case can be simplified considerably, we decided to present in this work only the Zoll case. We nonetheless propose the following generalization of Theorem \ref{mthm: Rabinowitz gradient flow line} to the Morse-Bott case, which we plan to prove in future work:

Let $(\Sigma,\alpha)$ be a contact manifold for which (MB$_\tau$) holds true for $\tau > 0 \,$. Let $J$ be an almost complex structure of SFT-type on the symplectization $S \Sigma $ and let $H_0 \in \Hcal^{h_0}$ be a defining Hamiltonian. 

\begin{conjecture*}
Given $\eps > 0 $ and a $C^1$-datum $\Dscr \,$, there exists a constant $\delta = \delta(J,H_0, h_0,\eps,\Dscr) > 0$ with the following significance: For every $H \in \Hclassparam$ and every slice $S \subseteq \Sigma$ for $N_\tau$ there exist a $(\nabla^{g_J} \RabinowitzH)$-flow line $w = (u,\eta) \in \Cinfty(\R,\loopspace\times\R)$, a sequence $(s_n)_{n \in \N}$ of positive numbers tending to infinity and critical points $(v^-,\tau^-), \, (v^+,\tau^+) \in \Crit(\RabinowitzH)$ with
\begin{enumerate}[label=\arabic*), itemsep=0.75ex]
     \item $\tau^-, \tau^+ \not= 0 \,$,
    \item $w(\pm s_n) \overset{n \rightarrow \infty}{\longrightarrow} (v^\pm, \tau^\pm)$ in $\Cinfty(\bbS^1,S\Sigma) \times \R \,$,
    \item $\RabinowitzH(w(s)) , \, \RabinowitzHZero(w(s)) \in [\tau- \eps, \, \tau + \eps]$ for every $s \in \R \,$,
    \item $u(0,0) \in \R \times S \,$.
\end{enumerate}
\end{conjecture*}

\noindent 

\appendix

\section{Proof of Lemma \ref{lem: Cieliebak-Frauenfelder}}
\label{app sec: Cieliebak-Frauenfelder estimate}

\noindent In this section we show Lemma \ref{lem: Cieliebak-Frauenfelder} by adapting the proof of \cite[Prop.~3.2 and 3.4]{Cieliebak_Frauenfelder} to the current setting. To rewrite the interpolated gradient equation more succinctly, it is convenient to introduce some notation. 

For a Hamiltonian $H$ and parameters $(r,s) \in \R_{\geq 0} \times \R$ we define the Hamiltonian
\[
H_{r,s} := H_0 + \beta_r(s) (H - H_0) \in \Cinfty(W) \,\ ,
\]
whose Hamiltonian vector field is
\[
X_{H_{r,s}} = X_{H_0} + \beta_r(s) \, (X_{H} - X_{H_0} ) \,\, .
\]
This allows us to express the interpolation functional $\RabinowitzHrs$ as
\begin{align*}
    \RabinowitzHrs(v,\tau) = \int_{\bbS^1} v^* \lambda - \tau \int_0^1 H_{r,s}(v(t)) \, \d t
\end{align*}
and its gradient as
\begin{equation}
\label{app eq: Cieliebak-Frauenfelder - gradient(RabinowitzHrs)}
 \nabla \RabinowitzHrs (v,\tau) = { {- J(v) (\partial_t v - \tau X_{H_{r,s}}(v))} \choose {-\int_0^1  H_{r,s}(v) \, \d t} } \,\, .
\end{equation}
Its squared $g_J$-norm is
\begin{align}
\label{app eq: Cieliebak-Frauenfelder - gradient(RabinowitzHrs) norm}
    |\nabla \RabinowitzHrs(v,\tau)|_{g_J}^2 = \lVert \partial_tv - \tau X_{H_{r,s}}(v) \rVert_{L^2}^2 + \big| \int_0^1 H_{r,s}(v) \, \d t \big|^2 \,\, ,
\end{align}
where the $L^2$-norm is with respect to the metric $\omega(\,\cdot\, , J \,\cdot\,) \,$. We stipulate that, when writing a norm of a vector field or a 1-form, by default we mean its $\omega(\,\cdot\, , J \,\cdot\,)$-norm.

\begin{proof}[Proof of Lemma \ref{lem: Cieliebak-Frauenfelder}]
For every $\kappa > 0 $ set $U_\kappa := H_0^{-1}((-\kappa,\kappa)) \,$. Since $H_0$ is a defining Hamiltonian, the family $(U_\kappa)_{\kappa >0}$ constitutes a neighborhood basis of $\Sigma = H_0^{-1}(0)$ in $W = S \Sigma$ of open sets, which are precompact for $\kappa$ sufficiently small.

Now let $\COneDatum$ be an arbitrary $C^1$-datum which we keep fixed throughout. We set
\[
C_\lambda := \lVert \lambda|_{U_\Sigma} \rVert_\infty \in (0,\infty) \,\, .
\]
We will prove the lemma in three steps.
\begin{step}
\label{app step: Cieliebak-Frauenfelder - 1}
There exist constants $\kappa > 0$ and $C > 0$ such that $U_\kappa \subseteq U_\Sigma$ and with the following significance. For every $H \in \Hclassdelta{\kappa} \,$, every $(r,s) \in \R_{\geq 0} \times \R$ and every $(v,\tau) \in \loopspace \times \R$ it holds
\begin{align*}
  \big( v(t) \in U_\kappa \quad \Forall t \in \bbS^1 \big) \,\,\, \Longrightarrow \,\,\, |\tau| \leq C \left( |\RabinowitzHrs(v,\tau)| + |\nabla \RabinowitzHrs(v,\tau)|_{g_J} \right)
\end{align*}
\end{step}
\noindent Since $H_0$ is a defining Hamiltonian for $\Sigma \,$, we have $\lambda(X_{H_0}) |_\Sigma \equiv 1 \,$. We can thus choose 
$0 < \kappa < \min\{\tfrac{1}{4} \comma \tfrac{c}{2}\}$ (with $c$ from the $C^1$-datum $\Dscr$) so that $ U_\kappa \subseteq U_\Sigma$ and
\[
\lambda(X_{H_0}) \geq \tfrac{1}{2} \quad \text{ on } U_\kappa \,\, .
\]
We now set (notice that the denominator is positive)
\begin{align*}
    C := \frac{\max\{1, C_\lambda\}}{\min\{\tfrac{1}{2} \comma c\} - 2 \kappa} > 0
\end{align*}
and claim that $\kappa$ and $C$ have the desired property.

To verify this, let $H \in \Hclassdelta{\kappa} \comma (r,s) \in \R_{\geq 0} \times \R$ and $(v,\tau) \in \loopspace \times \R $ be arbitrary. We assume $v \subseteq U_\kappa $ and have to bound $|\tau| \,$.

Since $H$ complies with $\Dscr \,$, it holds $\lambda(X_H) \geq c$ on $U_\Sigma \supseteq U_\kappa \,$. Consequently
\begin{align}
\label{app eq: Cieliebak-Frauenfelder - lambda(X_Hrs)}
    \lambda(X_{H_{r,s}}) = (1-\beta_r(s)) \underbrace{\lambda(X_{H_0})}_{\geq \tfrac{1}{2}} + \beta_r(s) \underbrace{\lambda(X_H)}_{\geq c} \geq \min\{\tfrac{1}{2} \comma c\} > 0 \quad \text{ on } U_\kappa \, .
\end{align}
Moreover, using $\lVert H - H_0 \rVert_\infty \leq \kappa $ because $H \in \Hclassdelta{\kappa} \,$, we also have 
\begin{align}
\label{app eq: Cieliebak-Frauenfelder - norm H_rs}
|H_{r,s}| \leq |H_0| + \beta_r(s) \, |H - H_0| \leq 2 \kappa \quad \text{ on } U_\kappa = H_0^{-1}((-\kappa,\kappa)) \, .
\end{align}
From the assumption $v \subseteq U_\kappa \subseteq U_\Sigma$ we infer
\begin{align}
    \big| \int_0^1 \lambda_v \big(\partial_t v - \tau X_{H_{r,s}}(v) \big) \, \d t \big| &\leq C_\lambda \int_0^1 |\partial_t v - \tau X_{H_{r,s}}(v)| \, \d t \notag\\
    &\leq C_\lambda \, \lVert \partial_t v - \tau X_{H_{r,s}}(v) \rVert_{L^2} \notag\\
    &\hspace*{-2mm}\overset{\text{\eqref{app eq: Cieliebak-Frauenfelder - gradient(RabinowitzHrs) norm}}}{\leq} C_\lambda \, |\nabla \RabinowitzHrs(v,\tau)|_{g_J} \,\, ,  \label{app eq: Cieliebak-Frauenfelder - int lambda(vdot - tau X_H)}
    \end{align}
(the second inequality is H\"older's inequality).
Now use the estimates in \eqref{app eq: Cieliebak-Frauenfelder - lambda(X_Hrs)}, \eqref{app eq: Cieliebak-Frauenfelder - norm H_rs} and \eqref{app eq: Cieliebak-Frauenfelder - int lambda(vdot - tau X_H)} as well as the assumption $v \subseteq U_\kappa$ again to bound
\begin{align*}
  |\RabinowitzHrs(v,\tau)| &= \Big| \int_0^1 \lambda_v(\partial_t v) \, \d t - \tau \int_0^1 H_{r,s}(v) \, \d t \Big|  \\
  &=  \Big| \tau \int_0^1 \lambda_v(X_{H_{r,s}}(v)) \, \d t + \int_0^1 \lambda_v (\partial_t v - \tau X_{H_{r,s}}(v)) \, \d t - \tau \int_0^1 H_{r,s}(v) \, \d t \Big| \\
  &\geq \Big| \tau \int_0^1 \lambda_v(X_{H_{r,s}}(v)) \,\d t \Big| - \Big| \int_0^1 \lambda_v (\partial_t v - \tau X_{H_{r,s}}(v)) \,\d t \Big| - \Big| \tau \int_0^1 H_{r,s}(v) \, \d t \Big| \\
  &\geq |\tau| \, \min\{\tfrac{1}{2} \comma c\} - C_\lambda \, |\nabla \RabinowitzHrs(v,\tau)|_{g_J} - 2 \kappa \, |\tau| \,\, .
\end{align*}
Rearrange this to obtain
\begin{align*}
    |\tau| \leq \frac{|\RabinowitzHrs(v,\tau)| + C_\lambda \, |\nabla \RabinowitzHrs(v,\tau)|_{g_J}}{\min\{\tfrac{1}{2} \comma c \} - 2 \kappa} \leq C \left( |\RabinowitzHrs(v,\tau)| + |\nabla \RabinowitzHrs(v,\tau)|_{g_J}\right) \,\, ,
\end{align*}
which finishes Step \ref{app step: Cieliebak-Frauenfelder - 1}.

\begin{step}
\label{app step: Cieliebak-Frauenfelder - 2}
For every $\kappa > 0$ with $U_\kappa \subseteq U_\Sigma$ there exists $\eps = \eps(\kappa) > 0$ with the following significance. For every $H \in \Hclassdelta{\tfrac{\kappa}{8}}$, every $(r,s) \in \R_{\geq 0} \times \R$ and every $(v,\tau) \in \loopspace \times \R$ it holds
\begin{align*}
    |\nabla \RabinowitzHrs(v,\tau)|_{g_J} \leq \eps \,\,\, \Longrightarrow \,\,\, \big( v(t) \in U_\kappa \quad \Forall t \in \bbS^1 \big) \,\, .
\end{align*}
\end{step}

\noindent Since $\BCOne \subseteq \Cinfty(\overline{U_\Sigma})$ is a $C^1$-bounded subset (where $\BCOne$ is from the $C^1$-datum $\Dscr$), we can choose a constant $D > 0$ so that
\begin{align*}
    \lVert \d f \rVert_\infty \leq D  \qquad \Forall f \in \BCOne \,\, . 
\end{align*}
In particular, by definition of complying with a $C^1$-datum, for every Hamiltonian $H$ it holds
\begin{align}
    H \text{ complies with } \Dscr \,\,\, \Longrightarrow \,\,\, \lVert \d H|_{U_\kappa}\rVert_\infty \leq \lVert \d H|_{\overline{U_\Sigma}}\rVert_\infty \leq D \,\, . \label{app eq: Cieliebak-Frauenfelder - constant C_Dscr choice}
\end{align}
Now choose $\eps > 0$ so that
\begin{align*}
    \eps < \min\big\{\tfrac{3}{8} \kappa \comma \tfrac{1}{4} \kappa \, \big(\lVert \d H_{0}|_{U_\kappa} \rVert_\infty + D \big)^{-1} \big\} \,\, .
\end{align*}
We claim that $\eps$ has the stated property. To this end, let $H \in \Hclassdelta{\tfrac{\kappa}{8}} \comma (r,s) \in \R_{\geq 0} \times \R$ and $(v,\tau) \in \loopspace \times \R $ be arbitrary. In particular $\lVert H - H_0 \rVert_\infty \leq \tfrac{\kappa}{8} \,$.
Assume that there exists some $t \in \bbS^1$ with $v(t) \notin U_\kappa \,$. We have to show that
\[
|\nabla \RabinowitzHrs (v,\tau)|_{g_J} > \eps \,\, .
\]
Notice that either $v$ is contained in the complement of $U_{\kappa /2 }$ or that there exist $t_1,t_2 \in \bbS^1$ with $|H_0(v(t_1))| \geq \kappa$ and $|H_0(v(t_2))| < \tfrac{\kappa}{2} \,$.
\\

\noindent \textbf{Case 1:} Suppose $v(t) \notin U_{\kappa/2}$ for every $t \in \bbS^1 \,$.

For every $t \in \bbS^1 \,$, the reverse triangle inequality yields
\begin{align*}
    |H_{r,s}(v(t))|  &\geq \big|\, |H_0(v(t))| - \beta_r(s) | (H-H_0)(v(t)) |  \,\big| \\
    &\geq |H_0(v(t))| - \lVert H -H_0 \rVert_\infty \geq \tfrac{\kappa}{2} - \tfrac{\kappa}{8} = \tfrac{3}{8} \kappa \,\, .
\end{align*}
In particular, $H_{r,s}(v)$ does not change signs. Then
\[
|\nabla \RabinowitzHrs(v,\tau)|_{g_J} \overset{\text{\eqref{app eq: Cieliebak-Frauenfelder - gradient(RabinowitzHrs) norm}}}{\geq} \big| \int_0^1 H_{r,s}(v) \, \d t \big| \geq \tfrac{3}{8} \kappa > \eps \,\, .
\]
\noindent \textbf{Case 2:} Suppose there exist $t_1,t_2 \in \bbS^1$ with $|H_0(v(t_1))| \geq \kappa$ and $|H_0(v(t_2))| \leq \tfrac{\kappa}{2} \,$.

Choose representatives of $t_1,t_2 \in \R/\Z$ in $\R $ (also denoted by $t_1,t_2 $) with $t_2 - t_1 \in (0,1) \,$. Without loss of generality we may assume that\footnote{
Let $t_1' := \max\{t \in [t_1,t_2]  : \,\, |H_0(v(t))| \geq \kappa\} \,$. Then $t_2 - t_1' \in (0,1) \,$. Now replace $t_1$ with $t_1' \,$.
}
\[
v(t) \in U_\kappa \qquad \Forall t \in (t_1,t_2) \,\, .
\]
It follows from $\lVert H - H_0 \rVert_\infty \leq \tfrac{\kappa}{8}$ and the triangle inequality that $|H_{r,s}(v(t_1))| \geq \tfrac{7}{8} \kappa$ and $|H_{r,s}(v(t_2))| \leq \tfrac{5}{8} \kappa \,$, hence
\[
\tfrac{1}{4} \kappa \leq |H_{r,s}(v(t_2)) - H_{r,s}(v(t_1))| \,\, .
\]
Using that $H$ complies with $\Dscr$ since $H \in \Hclassdelta{\tfrac{\kappa}{8}}$, by \eqref{app eq: Cieliebak-Frauenfelder - constant C_Dscr choice} we have
\begin{align*}
\lVert \d H_{r,s} |_{U_\kappa} \rVert_\infty \leq (1-\beta_r(s)) \, \lVert \d H_0 |_{U_\kappa}\rVert_{\infty} + \beta_r(s) \, \lVert \d H |_{U_\kappa} \rVert_\infty \leq  \lVert \d H_0 |_{U_\kappa}\rVert_{\infty} + D \,\, .
\end{align*}
Keeping in mind the trivial relation
\[
\d H_{r,s} (X_{H_{r,s}}) = \omega (X_{H_{r,s}} , X_{H_{r,s}}) = 0 \,\, ,
\]
we estimate
\begin{align*}
    \tfrac{1}{4} \kappa &\leq |H_{r,s}(v(t_2)) - H_{r,s}(v(t_1))|  \\
    &= \Big| \int_{t_1}^{t_2} (\d H_{r,s})_v(\partial_t v) \, \d t \Big| = \Big| \int_{t_1}^{t_2} (\d H_{r,s})_v(\partial_t v - \tau X_{H_{r,s}}(v) ) \ \d t\Big| \\
    &\leq \lVert \d H_{r,s}|_{U_\kappa} \rVert_\infty \int_{t_1}^{t_2}  |\partial_t v - \tau X_{H_{r,s}}(v)| \,\d t \\
    &\leq \big( \lVert \d H_0 |_{U_\kappa} \rVert_\infty + D \big) \int_0^1 |\partial_t v - \tau X_{H_{r,s}}(v)| \, \d t \\
    &\leq \big( \lVert \d H_0 |_{U_\kappa} \rVert_\infty + D \big) \, |\nabla \RabinowitzHrs(v,\tau)|_{g_J} \,\, ,
\end{align*}
where the penultimate inequality uses $t_2 - t_1 < 1$ and that $v$ is 1-periodic and the last inequality follows by first applying H\"older and then \eqref{app eq: Cieliebak-Frauenfelder - gradient(RabinowitzHrs) norm}.
Thus
\begin{align*}
    |\nabla \RabinowitzHrs(v,\tau)|_{g_J} \geq \frac{\kappa}{4 (\lVert \d H_0 |_{U_\kappa} \rVert_\infty + D)} > \eps \,\, .
\end{align*}
This finishes Case 2 and consequently also Step \ref{app step: Cieliebak-Frauenfelder - 2}.

\begin{step}
    Choose constants $C ,\kappa > 0$ as in Step \ref{app step: Cieliebak-Frauenfelder - 1}. Choose $\eps = \eps(\kappa) > 0$ as in Step \ref{app step: Cieliebak-Frauenfelder - 2} and set $\delta := \min\{\tfrac{\kappa}{8} ,\, \eps , \, 1 \} \,$. Then $C,\delta$ have the property stated in Lemma \ref{lem: Cieliebak-Frauenfelder}. 
\end{step}
 \noindent Given $H \in \Hclassdelta{\delta} $ and $(v,\tau) \in \loopspace \times \R$ as well as $(r,s) \in \R_{\geq 0} \times \R $ and assume that $|\nabla \RabinowitzHrs(v,\tau)|_{g_J} \leq \delta \,$. Observe that $H$ is an element of $\Hclassdelta{\tfrac{\kappa}{8}} \subseteq \Hclassdelta{\kappa} \,$.
 It follows from Step \ref{app step: Cieliebak-Frauenfelder - 2} that $v \subseteq U_\kappa \,$. By Step \ref{app step: Cieliebak-Frauenfelder - 1} hence
 \begin{align*}
     |\tau| \leq C \left( |\RabinowitzHrs(v,\tau)| + | \nabla \RabinowitzHrs(v,\tau)|_{g_J} \right) \leq  C \left( |\RabinowitzHrs(v,\tau)| + \delta \right) \leq C \, |\RabinowitzHrs(v,\tau)| + C \,\, .
 \end{align*}
\end{proof}

\section{Local Computations in the Banach Manifold $\Bcal$}
\label{app sec: Banach Mfd Bcal}

\subsection{Smooth atlas for $\Bcal'$}
\label{app subsec: Bcal' Atlas}

The set $\Bcal'$ was introduced in Definition \ref{def: Analytic Set-Up - Banach mfd}. We define charts for $\Bcal' \,$, the collection of which constitutes a smooth atlas.

We begin by observing that the smooth and asymptotically constant $w \in \Bcal'$ form a $C^0$-dense subset of $\Bcal' \,$. In particular, the subset 
\begin{equation*}
    (\Bcal')^\infty := \Bcal' \cap \big(\Cinfty(\Rbar \times \bbS^1,W) \times \Cinfty(\Rbar,\R) \big)
\end{equation*}
is a $C^0$-dense subset of $\Bcal' \,$. 

\subsubsection{Construction of a chart}
\label{app subsubsec: Bcal' Atlas - construction chart}
Now fix $w_0 = (u_0,\eta_0) \in (\Bcal')^\infty$ and let $(u_0^\pm,\tau(\alpha)) := w_0(\pm \infty) \in \CCrit \,$. We will construct a chart for $\Bcal'$ centered at $w_0$ and depending on various auxiliary choices. For simplicity we assume that $u_0^-(\bbS^1) \cap u_0^+(\bbS^1) = \varnothing \,$. (Otherwise $u_0^+ = u_0^- (t_0 + \,\cdot\,)$ for some $t_0 \in \bbS^1$ and we define the auxiliary data around $u_0^+$ in terms of the auxiliary choices for $u_0^- \,$.) We fix disjoint tubular neighborhoods
\[
\varphi^\pm = (\vartheta^\pm, y^\pm) : W \supseteq U^\pm \overset{\cong}{\longrightarrow} \bbS^1 \times \R^{2N-1}
\]
around the loops $u^\pm_0$ as in Lemma \ref{lem: Analytic Set-Up - tubular nbhd}. For every real number $a > 0$ we denote by $U^\pm_a \subseteq U^\pm$ the preimage of $\bbS^1 \times (-a,a)^{2N-1}$ under $\varphi^\pm \,$. Note that $U^\pm_a$ has compact closure contained in $U^\pm \,$. We choose an auxiliary Riemannian metric $\gaux$ on $W$ with $\gaux|_{U^\pm_2} = (\varphi^\pm|_{U_2^\pm})^{*} g_{\mathrm{euc}} \,$, the pullback of the Euclidean metric on $\bbS^1 \times \R^{2N-1} \,$. Fix an open, fiberwise convex neighborhood $\Ocal \subseteq TW$ of the zero section so that
\begin{enumerate}[label=\arabic*)]
    \item there exists a larger open, fiberwise convex neighborhood $\Ocal' \subseteq TW$ containing the closure of $\Ocal$ and for which
\[
\pi_{TW} \times \exp^{\gaux} : \, \Ocal' \longrightarrow W \times W
\]
is a diffeomorphism onto an open neighborhood of the diagonal,
\item $\exp^{\gaux}_p(\Ocal'_p) \subseteq U_2^\pm$ for every $p \in U_{1}^\pm \,$.
\end{enumerate}
Since $\varphi^\pm : (U^\pm_2 , \gaux) \rightarrow (\bbS^1\times \R^{2N-1},g_\mathrm{euc})$ is a local isometry, the $\gaux$-exponential map on $U^\pm_1$ is simply given by
\begin{equation}
\label{app eq: Bcal' Atlas - exp^gaux on U^pm}
(\varphi^\pm \circ \exp^{\gaux}_p)(\nu) = (\exp_{\varphi^\pm(p)}^{g_{\mathrm{euc}}} \circ \,\d_p \varphi^\pm)(\nu) = \varphi^\pm(p) + \d_p \varphi^\pm [ \nu ] \quad \Forall \nu \in \Ocal'_p \comma p \in U_1^\pm \, .
\end{equation}
Here, by slight abuse of notation, we regard $\d_p \varphi^\pm[\nu] \in \R^{2N}$ as element of $\bbS^1 \times \R^{2N-1}$ via the covering map $\R^{2N} \rightarrow \bbS^1 \times \R^{2N-1} \,$. We will also tacitly lift elements in $\bbS^1 \times \R^{2N-1} \,$, which are sufficiently close to $([0],0) \in \bbS^1 \times \R^{2N-1} \,$, to $\R^{2N} \,$.

By \eqref{app eq: Bcal' Atlas - exp^gaux on U^pm}, the hypersurfaces 
\[
\Sigma \cap U^\pm_1 = \{y_{2N-1}^\pm = 0 \comma \,\, |y_i^\pm| < 1 \,\,\, \Forall i=1,\ldots , 2N-2\}
\]
are totally geodesic with respect to $\gaux \,$.

Since $u_0$ tends uniformly to $u_0^\pm$ as $s \rightarrow \pm \infty \,$, we can fix $s_0 > 0$ so that $u_0(s,t) \in U^\pm_{1}$ for all $|s| \geq s_0  \comma t \in \bbS^1 \,$. Now let $\gamma \in \Cinfty(\R,[0,1])$ be a smooth cutoff function with 
\begin{align*}
   \gamma(s) = \begin{cases}
        0 & \text{ for } |s| \leq s_0 + 1  \\
        1 & \text{ for } |s| \geq s_0 + 2  \, ,
    \end{cases}
\end{align*}
being decreasing on $\R_{\leq 0}$ and increasing on $\R_{\geq 0} \,$. 

For every vector $\nu = ((s,t),\nu) \in u_0^*TW$ and $x^\pm \in \R^{2N}$ define the vector $\widehat{\nu}_{(x^-,x^+)} = ((s,t),\widehat{\nu}_{(x^-,x^+)}) \in u_0^*TW$ by
\begin{align*}
    \widehat{\nu}_{(x^-,x^+)} := \begin{cases}
        \nu & \text{, if } |s| \leq s_0 + 1\\
        \nu + \gamma(s) \, (\d_{u_0(s,t)} \varphi^\pm)^{-1}[x^\pm] &\text{, if } \pm s \geq s_0  \, .
    \end{cases}
\end{align*}
Let $B \subseteq \R^{2N}$ be the intersection of $\R^{2N-1} \times \{0\}$ with an open $2N$-ball centered at the origin of sufficiently small radius so that $\widehat{\nu}_{(x^-,x^+)} \in \Ocal'_{u_0(s,t)}$  for every $(s,t) \in \Rbar \times \bbS^1 \comma \nu \in \Ocal_{u_0(s,t)} $ and $x^-,x^+ \in B \,$.
For $((s,t),\nu) \in u_0^*\Ocal$ with $\pm s \geq s_0$ and $x^\pm \in B \,$, it follows from \eqref{app eq: Bcal' Atlas - exp^gaux on U^pm} that 
\begin{equation}
\label{app eq: Bcal' Atlas - exp^gaux at u_0}
    (\varphi^\pm \circ \exp_{u_0(s,t)}^{\gaux})(\widehat{\nu}_{(x^-,x^+)}) = (\varphi^\pm \circ u_0)(s,t) + \d_{u_0(s,t)}\varphi^\pm [\nu] + \gamma(s) \, x^\pm 
\end{equation}
for every $((s,t),\nu) \in u_0^*TW$ and $x^\pm \in B \,$. In particular, for $((s,t),\nu) = ((\pm \infty,0),0) $ we obtain
\begin{equation}
\label{app eq: Bcal' Atlas - exp^gaux at u_0(infty,t)}
    (\varphi^\pm \circ \exp_{u_0(\pm \infty,t)}^{\gaux})(\widehat{0}_{(x^-,x^+)}) = (\varphi^\pm \circ u^\pm_0)(t)  + \gamma(\pm \infty) \, x^\pm = (t,0) +  x^\pm \,\, ,
\end{equation}
so that, since the last component of the right-hand side vanishes, we have
\begin{equation}
\label{app eq: Bcal' Atlas - exp^gaux at u_0(infty,t) in Sigma}
    \exp_{u_0(\pm \infty,t)}^{\gaux}(\widehat{0}_{(x^-,x^+)}) \in \Sigma \cap U_1^\pm \,\, .
\end{equation}
Now define a smooth map $f : u_0^* \Ocal \times B \times B \rightarrow W$ by
\begin{align}
\label{app eq: Bcal' Atlas - map f}
f\big( ((s,t),\nu) \comma x^- , x^+ \big) := \exp_{u_0(s,t)}^{\gaux}(\widehat{\nu}_{(x^-,x^+)}) \,\, .
\end{align}
Let us write
\[
W_{\kappa_0}^{1,p}(\R \times \bbS^1,u_0^*\Ocal) := \set{\xi \in W^{1,p}_{\kappa_0}(\R \times \bbS^1,u_0^*TW)}{\xi(\R \times \bbS^1) \subseteq u_0^*\Ocal}
\]
and observe that this is in fact an open subset of $W^{1,p}_{\kappa_0}(\R \times \bbS^1,u_0^*TW)$ due to the continuous Sobolev embedding
\[
W^{1,p}_{\kappa_0}(\R \times \bbS^1,u_0^*TW) \hookrightarrow C^0(\Rbar \times \bbS^1,u_0^*TW) \,\, .
\]
The map $f$ in \eqref{app eq: Bcal' Atlas - map f} induces the pushforward map
\begin{align*}
    &f_* : W^{1,p}_{\kappa_0}(\R \times \bbS^1,u_0^*\Ocal) \times B \times B \times W^{1,p}_{\kappa_0}(\R,\R) \longrightarrow \Bcal' \\
     &f_*(\xi, x^-, x^+, \eta) := (u, \eta + \tau(\alpha)) \, , \quad \text{where } u(s,t) := f\big( ((s,t), \xi(s,t)) \comma x^- , x^+ \big) \,\, .
\end{align*}
The cylinder component of $f_*$ is illustrated in Figure \ref{fig: Bcal' chart construction}.

\begin{lemma}
\label{app lem: Bcal' Atlas - f_* well-defined}
The map $f_*$ is well-defined, \ie takes on values in $\Bcal' \,$.
\end{lemma}
\begin{proof}
Let $u := f(\xi,x^-,x^+) \,$. We must show that $(u,\eta+\tau(\alpha)) \in \Bcal' \,$. Observe that 
\begin{equation}
\label{app eq: Bcal' Atlas - u(pm infty,t) = (t,0) + x^pm}
\varphi^\pm(u(\pm \infty, t)) = (t,0) + x^\pm \in \Sigma
\end{equation}
by \eqref{app eq: Bcal' Atlas - exp^gaux at u_0(infty,t)} and \eqref{app eq: Bcal' Atlas - exp^gaux at u_0(infty,t) in Sigma}, so that $u^\pm = u(\pm\infty,\,\cdot\,)$ is a flow line of $\partial_\vartheta = \tau(\alpha)  R_\alpha$ and hence $(u^\pm,\eta^\pm + \tau(\alpha))=(u^\pm,\tau(\alpha)) \in \CCrit \,$. Moreover, by \eqref{app eq: Bcal' Atlas - exp^gaux at u_0} and \eqref{app eq: Bcal' Atlas - u(pm infty,t) = (t,0) + x^pm} it holds
\begin{align*}
    \varphi^+ (u(s,t)) - \varphi^+(u^+(t)) = \varphi^\pm(u_0(s,t)) - (t,0) + \d_{u_0(s,t)} \varphi^+ [\xi(s,t)] \quad \Forall  s \geq s_0 +2\, .
\end{align*}
This map lies in $W^{1,p}_{\kappa_0}(\R_{\geq s_0+2} \times \bbS^1)$ because $(s,t) \mapsto \varphi^+(u_0(s,t)) - (t,0)$ lies in $W^{1,p}_{\kappa_0}(\R_{\geq s_0+2} \times \bbS^1)$ (since $u_0 \in \Bcal'$) and $\d_{u_0} \varphi^+[\xi] \in W^{1,p}_{\kappa_0}(\R_{\geq s_0 + 2} \times \bbS^1) \,$. An analogous argument works for $s \leq - s_0 - 2 \,$.
\end{proof}

\begin{lemma}
\label{app lem: Bcal' Atlas - f_* injective}
The map $f_*$ is injective.
\end{lemma}

\begin{proof}
    If $u = f(\xi,x^-,x^+) \,$, then $x^\pm = \varphi^\pm(u(\pm \infty,0))$ by \eqref{app eq: Bcal' Atlas - u(pm infty,t) = (t,0) + x^pm} and
    \[
    \widehat{\xi}_{(x^-,x^+)} = (\exp^{\gaux}_{u_0})^{-1}(u) \,\, ,
    \]
    which determines $\xi$ in terms of $u$ and $(x^-,x^+) \,$. This shows that $f_*$ is injective.
\end{proof}

\begin{remark}
\label{app rmk: Bcal' Atlas - f_* C^0-open image}
We note that the image of $f_*$ is $C^0$-open in $\Bcal' \,$. That is, given $(u,\eta) \in \mathrm{im}(f_*) \,$, every $(\widetilde{u},\widetilde{\eta}) \in \Bcal' \,$, for which $\widetilde{u}$ is sufficiently close to $u$ in $C^0(\Rbar \times \bbS^1,W) \,$, lies in $\mathrm{im}(f_*)\,$.
\end{remark}

\begin{figure}
    \centering
    \includegraphics[width=0.6\linewidth]{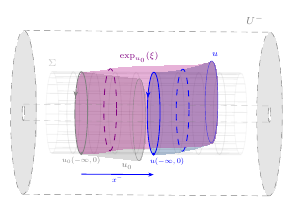}
    \caption{Negative asymptotics of cylinder {\color{blue}$u = \exp_{u_0}^{\gaux}(\widehat{\xi}_{(x^-,x^+)})$} near the cylinder {\color{gray}$u_0 \,$}. The cylinder {\color{violet}$\exp_{u_0}^{\gaux}(\xi)=\exp_{u_0}^{\gaux}(\widehat{\xi}_{(0,0)})$} has asymptotic loop {\color{gray}$u_0^- = u_0(-\infty,\,\cdot\,) \,$}; for $s \leq -s_0-2$ (up to the dashed loops), it is a shift of $u$ by $-x^-$ in the tubular coordinates $(U^-,\varphi^-) $ around $u_0^- \,$; for $s \geq -s_0 -1\,$, it agrees with $u \,$.
    }
    \label{fig: Bcal' chart construction}
\end{figure}

\subsubsection{The atlas}
\label{app subsubsec: Bcal' Atlas - Atlas}
An atlas for $\Bcal'$ is now given by taking the collection $\mathscr{A} := \{(\mathrm{im}(f_*) , f_*^{-1}) \} \,$, varying over all admissibly auxiliary data that were necessary to define $f$. In order to verify that $\mathscr{A}$ indeed defines an atlas, we first record that the chart domains cover $\Bcal' \,$. In fact, countably many charts suffice.

\begin{lemma}
\label{app lem: Bcal' Atlas - chart domains cover Bcal'}
Countably many of the domains $\mathrm{im}(f_*)$ from charts in $\mathscr{A}$ cover $\Bcal' \,$.
\end{lemma}

\begin{proof}
    We can choose finitely many tubular neighborhoods $(U_i,\varphi_i)_{i=1}^n$ as in Lemma \ref{lem: Analytic Set-Up - tubular nbhd} so that $(U_i)_{i=1}^n$ is a cover of $\Sigma $ by open subsets of $W$ and moreover, given any two closed $R_\alpha$-Reeb orbits $v^-,v^+ \subseteq \Sigma \,$, one of the following holds true:
        \begin{enumerate}[label=\arabic*)]
            \item There exists $1 \leq i \leq n$ with $v^-,v^+ \subseteq U_i \,$.
            \item There exist $1 \leq i^-, i^+ \leq n$ with $v^- \subseteq U_{i^-} \comma v^+ \subseteq U_{i^+}$ and the closures of $U_{i^-}$ and $U_{i^+}$ are disjoint.
        \end{enumerate}
    This may be seen by viewing Reeb orbits in $\Sigma$ as points in the quotient $\Sigma  / \bbS^1$ of $\Sigma$ by the free Reeb action (which is a smooth compact manifold).
    
    We consider pairs $I = (i^-,i^+)$ with $1 \leq i^-,i^+ \leq n$ so that $i^-=i^+$ or the closures of $U_{i^-}$ and $U_{i^+}$ are disjoint. For each such pair $I= (i^-,i^+)$ choose an auxiliary metric $g_I $ on $W$, so that $g_I|_{ U_{i^\pm}}$ agrees with the pullback of the Euclidean metric via $\varphi_{i^\pm} \,$. Consider the exhaustion of $W = S \Sigma$ by the compact subsets $K_j := [-j,j] \times \Sigma \comma j \in \N \,$. The injectivity radius of each $g_I$ on $K_j$ is bounded below by some constant $c_{Ij} > 0 \,$. Now choose a countable, dense subset $\Sigma'$ of $\Sigma \,$.
    
    For each $I = (i^-,i^+) $ and $j, n \in \N$ consider the subset $\mathscr{S}_{Ijn} \subseteq \Cinfty(\R \times \bbS^1,W)$ consisting of cylinders $u$ having image contained in $K_j \,$, being $s$-independent for $|s| \geq n$ and so that $u(\pm \infty,\,\cdot\,) \subseteq \Sigma \cap U_{i^\pm}$ is a periodic orbit of $\tau(\alpha) R_\alpha \,$. (The last condition means that $(u(\pm\infty,\,\cdot\,),\tau(\alpha)) \in \CCrit \,$.)
    By standard arguments from differential topology, there exists a countable subset $\mathscr{S}'_{Ijn} \subseteq \mathscr{S}_{Ijn}$ having the following properties
    \begin{enumerate}[label=\arabic*)]
        \item $u'(\pm \infty,0) \in \Sigma'$ for every $u' \in \mathscr{S}'_{Ijn} \,$,
        \item\label{app it: Bcal' Atlas - dense S'_ijn} $\mathscr{S}'_{Ijn}$ is $C^0$-dense in $\mathscr{S}_{Ijn} \,$, \ie every $u \in \mathscr{S}_{Ijn}$ can be approximated in $C^0(\R \times \bbS^1,K_j)$ by a sequence $(u_\ell')_\ell \subseteq \mathscr{S}'_{Ijn} \,$.
    \end{enumerate}
    For each $(I,j,n)$ and every $u \in \mathscr{S}'_{Ijn}$ we define $f = f^{(I,j,n,u)}$ as in \eqref{app eq: Bcal' Atlas - map f}, say with radius of both the zero section neighborhood $\mathcal{O}$ and the $(2N-1)$-ball $B \subseteq \R^{2N-1} \times \{0\}$ at least $\tfrac{1}{3} c_{Ij} \,$. This yields countably many $f_*^{(I,j,n,u)}$ and we contend that their images cover $\Bcal' \,$.

   To this end, for given $\widetilde{w} =(\widetilde{u},\widetilde{\eta}) \in \Bcal' \,$, choose a pair $I=(i^-,i^+)$ with $\widetilde{u}(\pm \infty,\,\cdot\,) \subseteq U_{i^\pm}$ and $j$ so that $\widetilde{u}$ has image contained in $K_{j} \,$. Now observe that we can approximate $\widetilde{u}$ arbitrarily well in $C^0(\Rbar \times \bbS^1,W)$ by elements from $\bigcup_{n \in \N} \mathscr{S}_{Ijn}$ and thus by elements from $\bigcup_{n \in \N} \mathscr{S}'_{Ijn} \,$. For $u \in \mathscr{S}'_{Ijn}$ having $C^0$-distance from $\widetilde{u}$ smaller than $\tfrac{1}{6} c_{Ij} \,$, say, it holds $\widetilde{w} \in \mathrm{im}(f_*^{(I,j,n,u)})$ 
\end{proof}

\subsubsection{Transition maps are smooth}
\label{app subsubsec: Bcal' Atlas - Transition maps smooth}
That the charts in $\mathscr{A}$ have smooth transition maps hinges upon the next local smoothness result on Banach spaces.

Given a smooth function $F : (\Rbar \times \bbS^1) \times \R^{m} \times \R^d \longrightarrow \R^n \,$, we denote by $\partial_s F$ and $\partial_t F$ its partial derivatives with respect to $\R \subseteq \Rbar$ respectively $\bbS^1\,$ and by $D_2 F \in \mathcal{L}(\R^{d},\R^{n})$ its partial derivative in $\R^d$-direction. Similarly, $D_2^k F \in \mathcal{L}(\R^d,\ldots,\R^d;\R^n)$ denotes the $k$-th partial derivative in $\R^d$-direction. By convention $D_2^0 F := F \,$.

\begin{lemma}
\label{app lem: Bcal' Atlas - smoothness lemma}
Suppose $F \in \Cinfty(\Rbar \times \bbS^1 \times \R^m \times \R^d, \R^n)$ satisfies
\begin{itemize}
    \item[] for every $k \geq 0 $ and every $x_0 \in \R^d$ there exist a neighborhood $U(x_0) \subseteq \R^d$ of $x_0$ and a non-negative function $h = h_{(k,x_0)} \in L^p_{\kappa_0}(\R \times \bbS^1,\R)$ so that
    \begin{equation}
    \label{app eq: Bcal' Atlas - smoothness lemma assumption}
    |D^k_2 F((s,t),0,x)| + |\partial_s D^k_2 F( (s,t),0,x)| + |\partial_t D^k_2 F( (s,t),0,x)| \leq h(s,t)
    \end{equation}
    for every $(s,t) \in \R \times \bbS^1$ and $x \in U(x_0) \,$.
\end{itemize}
Then 
\begin{align*}
    &F_* : W^{1,p}_{\kappa_0}(\R \times \bbS^1,\R^m) \times \R^d \longrightarrow W^{1,p}_{\kappa_0}(\R \times \bbS^1,\R^n) \\
    &F_*(\xi,x) \, (s,t) := F((s,t),\xi(s,t),x) \quad \text{for } \xi \in W^{1,p}_{\kappa_0}(\R \times \bbS^1,\R^m) \text{ and } x \in \R^d
\end{align*}
is a well-defined smooth map between Banach spaces.
\end{lemma}

\begin{remark}
\label{app rmk: Bcal' Atlas - smoothness lemma}
There is a version of the lemma for $F$ being defined only on a suitable open subset of $\Rbar \times \bbS^1 \times \R^m \times \R^d \,$.
\end{remark}

\noindent Lemma \ref{app lem: Bcal' Atlas - smoothness lemma} can be proved in a manner analogous to \cite[Lemma~A.5]{M_Schwarz}. Care has to be taken due to the additional parameter $x$ and because we work with weighted Sobolev spaces. We omit the rather lengthy proof of Lemma \ref{app lem: Bcal' Atlas - smoothness lemma} but use it to show that the transition maps in $\mathscr{A}$ are smooth.

\begin{lemma}
\label{app lem: Bcal' Atlas - smooth transition maps}
Given $f_{1\,*}^{-1} \comma f_{2\,*}^{-1} \in \mathscr{A}$ whose domains intersect non-trivially, \ie $\Ucal_{12} := \mathrm{im}(f_{1\,*}) \cap \mathrm{im}(f_{2\,*}) \not= \varnothing \,$. Then
\[
\mathcal{V}_{12} := f_{1\,*}^{-1}(\Ucal_{12}) \subseteq W^{1,p}_{\kappa_0}(\R \times \bbS^1,u_1^*TW) \times B_1 \times B_1 \times W^{1,p}_{\kappa_0}(\R,\R)
\]
is an open subset and
\[
f_{2\,*}^{-1} \circ f_{1\,*} : \, \mathcal{V}_{12} \longrightarrow W^{1,p}_{\kappa_0}(\R \times \bbS^1,u_2^*TW) \times B_2 \times B_2 \times W^{1,p}_{\kappa_0}(\R,\R)
\]
is a smooth map between open subsets of Banach spaces.
\end{lemma}

\begin{proof}
    By \eqref{app eq: Bcal' Atlas - u(pm infty,t) = (t,0) + x^pm} the transition map $f_{2\,*}^{-1} \circ f_{1\,*}$ restricted to the finite-dimensional components $B_{1,2} \times B_{1,2}$ is just given by
    \begin{equation}
    \label{app eq: Bcal' Atlas - x_2^pm def}
    x_2^\pm = (\varphi_2^\pm \circ (\varphi_1^\pm)^{-1})(x_1^\pm) \in B_2 \qquad \text{for } x_1^\pm \in B_1 \,\, .
    \end{equation}
    To make sense of \eqref{app eq: Bcal' Atlas - x_2^pm def}, we remind the reader that $B_{1,2}$ are intersections of $\R^{2N-1} \times \{0\}$ with balls of small radius so that we can regard $x_{1,2}^\pm \in B_{1,2}$ as elements of $\bbS^1 \times \R^{2N-2} \times \{0\} $ via the natural projection in the first factor (which is a local diffeomorphism).
    
    The nontrivial part of the lemma is showing smoothness of the first component of $f_{2\,*}^{-1} \circ f_{1\,*} \,$. For $j=1,2 \,$, choose smooth global trivializations $\Psi_j : u_j^*TW \overset{\cong}{\rightarrow} (\Rbar \times \bbS^1) \times \R^{2N}$, which induce toplinear isomorphisms $W^{1,p}_{\kappa_0}(\R \times \bbS^1,u_j^*TW) \cong W^{1,p}_{\kappa_0}(\R \times \bbS^1,\R^{2N}) \,$. We define the smooth map
    \begin{align*}
        F : (\Rbar \times \bbS^1) \times \R^{2N} \times (\R^{2N-1} \times \R^{2N-1} ) \longrightarrow \R^{2N}
    \end{align*}
     implicitly by the equation
     \begin{align*}
         f_1 \big(  (s,t) , \nu   , \, (x_1^-,x_1^+) \big) = f_2 \big(  (s,t) , F((s,t),\nu,(x^-_1,x^+_1) )   , \, (x_2^-,x_2^+) \big) \,\, ,
     \end{align*}
     with $\nu \in \R^{2N} \comma x^\pm_1 \in \R^{2N-1} \cong \R^{2N-1} \times \{0\} \subseteq \R^{2N}$ and $x_2^\pm$ defined by \eqref{app eq: Bcal' Atlas - x_2^pm def}. Strictly speaking, $F$ is only defined on a suitable open subset of $(\Rbar \times \bbS^1) \times \R^{2N} \times (\R^{2N-1} \times \R^{2N-1})$. The map $F$ is uniquely determined by the implicit equation and smooth because $f_{j}((s,t),\,\cdot\,,(x^-_j,x^+_j))$ is a diffeomorphism for fixed $(s,t)$ and $(x^-_j,x^+_j) $ by its definition in equation \eqref{app eq: Bcal' Atlas - map f}. 
     
     By construction, the first component of $f_{2 \, *}^{-1} \circ f_{1 \,*}$ agrees with the composition of the projection along the Lagrange multiplier component with
     \[
     F_* : W^{1,p}_{\kappa_0}(\R \times \bbS^1 , \R^{2N}) \times (\R^{2N-1} \times \R^{2N-1}) \longrightarrow W^{1,p}_{\kappa_0}(\R \times \bbS^1,\R^{2N}) \,\, ,
     \]
     where $F_*$ is defined as in Lemma \ref{app lem: Bcal' Atlas - smoothness lemma}. We will now verify that $F$ satisfies the hypothesis of Lemma \ref{app lem: Bcal' Atlas - smoothness lemma}.
     
     By \eqref{app eq: Bcal' Atlas - exp^gaux at u_0}, \eqref{app eq: Bcal' Atlas - map f} and definition of $F$, for $|s|$ sufficiently large,
     \begin{align}
     \label{app eq: Bcal' Atlas - chart transitions exp^g_1}
    \exp^{g_1}_{u_1}\widehat{(\Psi_1^{-1} [\nu]}_{(x_1^-,x_1^+)}) = (\varphi_1^\pm)^{-1}\big( (\varphi_1^\pm \circ u_1) + \d_{u_1}\varphi_1^\pm [\Psi_1^{-1} [\nu]] + x_1^\pm \big) \,\, ,    
     \end{align}
     agrees with
     \begin{align}
    \label{app eq: Bcal' Atlas - chart transitions exp^g_2}
         \exp^{g_2}_{u_2}\widehat{(\Psi_2^{-1} [F]}_{(x_2^-,x_2^+)}) = (\varphi_2^\pm)^{-1}\big( (\varphi_2^\pm \circ u_2) + \d_{u_2}\varphi_2^\pm [\Psi_2^{-1} [F]] + x_2^\pm \big) \,\, ,
     \end{align}
     where $x_2^\pm$ is defined in terms of $x_1^\pm$ by \eqref{app eq: Bcal' Atlas - x_2^pm def}. For $j=1,2 \,$, we abbreviate 
     \[
     \widetilde{u}_j^{\pm}(s,t) := \varphi_j^\pm(u_j(s,t)) - (t,0) \in \bbS^1 \times \R^{2N-1}
     \]
      and $A_{j}^\pm(s,t) := \d_{u_j(s,t)} \varphi_j^\pm \circ (\Psi_j)_{(s,t)}^{-1} \,$. We observe that $\widetilde{u}_{j}^\pm \in W^{1,p}_{\kappa_0}$ (since $u_j \in \Bcal'$) and that $A_j^+$ is a smooth map in $\Cinfty([s_0,+\infty] \times \bbS^1,\mathrm{GL}(2N,\R))$ and similarly for $A_j^- \,$. Equating \eqref{app eq: Bcal' Atlas - chart transitions exp^g_1} with \eqref{app eq: Bcal' Atlas - chart transitions exp^g_2}
     and rearranging yields 
     \begin{align}
     \label{app eq: Bcal' Atlas - F asymptotic intermediate result}
       F  = (A_2^\pm)^{-1}  \big[ - \widetilde{u}_2^\pm - (t,0)  - x_2^\pm + (\varphi_2^\pm \circ (\varphi^\pm_1)^{-1}) \big( \widetilde{u}_1^\pm + (t,0) + x_1^\pm  + A_1^\pm [\nu]\big)  \big] \,\, .
     \end{align}
     Having chosen the tubular maps $\varphi_j^\pm$ as in Lemma \ref{lem: Analytic Set-Up - tubular nbhd}, the maps
     \[
     \varphi^\pm_j : U_j^\pm \cap \Sigma \overset{\cong}{\longrightarrow} \bbS^1 \times \R^{2N-2} \times \{0\}
     \]
     are $\bbS^1$-equivariant, \ie they conjugate the flow of $\tau(\alpha)R_\alpha$ on $\Sigma$ with the $\bbS^1$-action on the first component of $\bbS^1 \times \R^{2N-2} \times \{0\} \,$. Hence 
     \begin{align*}
         x_2^\pm + (t,0) \overset{\text{\eqref{app eq: Bcal' Atlas - x_2^pm def}}}{=} (\varphi_2^\pm \circ (\varphi_1^\pm)^{-1})(x_1^\pm) + (t,0) =  (\varphi_2^\pm \circ (\varphi_1^\pm)^{-1})(x_1^\pm + (t,0)) \,\, .
     \end{align*}
     Plugging this relation into \eqref{app eq: Bcal' Atlas - F asymptotic intermediate result}, we obtain
     \begin{align}
        \hspace*{4mm}F( (s,t),\nu,(x^-_1,x^+_1)) 
         &= (A_2^\pm(s,t))^{-1} \Big[ - \widetilde{u}_2^\pm(s,t)  - (\varphi_2^\pm \circ (\varphi_1^\pm)^{-1})((t,0) + x_1^\pm)  \notag\\
        &\hspace*{4mm}+ (\varphi_2^\pm \circ (\varphi^\pm_1)^{-1}) \big( \widetilde{u}_1^\pm(s,t) + (t,0) + x_1^\pm   + A_1^\pm(s,t) [\nu]\big)  \Big] \,\, . \label{app eq: Bcal' Atlas - F asymptotic}
     \end{align}
     Using formula \eqref{app eq: Bcal' Atlas - F asymptotic}, that $\widetilde{u}_{1,2}^\pm \in W^{1,p}_{\kappa_0}$ and that $(A_{2}^\pm)^{-1}$ is bounded and has bounded first derivatives, one readily verifies that the assumption \eqref{app eq: Bcal' Atlas - smoothness lemma assumption} in Lemma \ref{app lem: Bcal' Atlas - smoothness lemma} holds for every $k \geq 0$ and every $(x^-,x^+) \in \R^{2N-1} \times \R^{2N-1} \,$. For example, for large $|s|$ and $(\widetilde{x}^-,\widetilde{x}^+)$ close to $(x^-,x^+)$ due to \eqref{app eq: Bcal' Atlas - F asymptotic} we may estimate
     \begin{align*}
      &|F((s,t),0,(\widetilde{x}^-,\widetilde{x}^+))| \\
      \leq\,\,&C_1 \, |\widetilde{u}_2^\pm| + C_1\,\big| (\varphi_2^\pm \circ (\varphi_1^\pm)^{-1}) (\widetilde{u}_1^\pm + (t,0) + \widetilde{x}^\pm) - (\varphi_2^\pm \circ (\varphi_1^\pm)^{-1})((t,0) + \widetilde{x}^\pm) \big| \\
      \leq\,\,&C_1 \, |\widetilde{u}_2^\pm| + C_2 \, |\widetilde{u}_1^\pm| \in L^{p}_{\kappa_0} \,\, , 
     \end{align*}
     where the second inequality follows from a Lipschitz estimate for $\varphi_2^\pm \circ (\varphi_1^\pm)^{-1}$ on a compact neighborhood of $\set{x^\pm + (t,0)}{t \in \bbS^1}\,$. Similar estimates hold for $\partial_s F$ and $\partial_t F$ at $((s,t),0,(\widetilde{x}^-,\widetilde{x}^+)) \,$, showing that \eqref{app eq: Bcal' Atlas - smoothness lemma assumption} holds for $k=0 \,$.

    Therefore the hypothesis of Lemma \ref{app lem: Bcal' Atlas - smoothness lemma} holds, hence $F_*$ is smooth and so is the first component of $f_{2\,*}^{-1} \circ f_{1\,*} \,$.
\end{proof}

\subsubsection{Conclusion}
\label{app subsubsec: Bcal' Atlas - Conclusion}

The atlas $\mathscr{A}$ thus gives $\Bcal'$ a smooth manifold structure.

\begin{corollary}
\label{app cor: Bcal' Atlas - Bcal' smooth Banach mfd}
$(\Bcal', \mathscr{A})$ is a Hausdorff smooth Banach manifold without boundary.
\end{corollary}

\begin{proof}
    Since chart transitions are smooth (Lemma \ref{app lem: Bcal' Atlas - smooth transition maps}) and charts in $\mathscr{A}$ cover $\Bcal'$ (Lemma \ref{app lem: Bcal' Atlas - chart domains cover Bcal'}), the pair $(\Bcal',\mathscr{A})$ is a smooth Banach manifold with a uniquely determined topology for which the charts in $\mathscr{A}$ have open domains and are homeomorphisms onto their images. For every two distinct elements in $\Bcal'$ there exist charts with disjoint chart domains, hence $\Bcal'$ is Hausdorff.
\end{proof}

\noindent We supplement the following result regarding the topology of $\Bcal' $ induced by the atlas $\mathscr{A} \,$.

\begin{corollary}
\label{app cor: Bcal' Atlas - Bcal' second countable}
$\Bcal'$ is second-countable and metrizable.
\end{corollary}

\begin{proof}
    Since the Sobolev spaces $W^{1,p} \cong W^{1,p}_{\kappa_0}$ are separable, each chart $f_*^{-1}$ for $\Bcal'$ is modeled on a separable (hence second-countable) Banach space. Because $\Bcal'$ admits a countable subatlas by Lemma \ref{app lem: Bcal' Atlas - chart domains cover Bcal'}, it is therefore second-countable. Now every Hausdorff second-countable Banach manifold is metrizable, see \cite[Cor.~2]{Palais_ParacompactBanachMfds}.
\end{proof}

\subsection{Supplement to the proof of Proposition \ref{prop: Non_Empty Moduli Spaces - Compactness}}
\label{app subsec: Supplement Compactness Moduli Space in BcalHat}

We prove the estimates \eqref{eq: Compactness - estimates xi_n^phi} in the proof of Proposition \ref{prop: Non_Empty Moduli Spaces - Compactness}. We recall that
\[
w_n = (u_n,\eta_n) = f_*\big(\xi_n,x_n^-,x_n^+,\eta_n - \tau(\alpha)\big) \in \Bcal
\]
all lie in the fixed chart centered at $w=(u,\eta) \,$. We also remind the reader of the notation $\xi_n^{\varphi^+} := \d_u \varphi^+ [\xi_n] $ and analogously define $u_n^{\varphi^+} := \varphi^+ \circ u_n$ and $u^{\varphi^+} := \varphi^+ \circ u \,$.

Observe that \eqref{app eq: Bcal' Atlas - exp^gaux at u_0} can asymptotically (for $s \gg 0$) be rewritten as
\begin{equation*}
u_n^{\varphi^+} = u^{\varphi^+} + \xi_n^{\varphi^+} + u^{\varphi^+}_n(+\infty,0) = u^{\varphi^+} + \xi_n^{\varphi^+} + u^{\varphi^+}_n(+\infty,t) - u^{\varphi^+}(+\infty,t) \, ,
\end{equation*}
using that $u_n^{\varphi^+}(+\infty,t) = u_n^{\varphi^+}(+\infty,0) + (t,0) $ and $u^{\varphi^+}(+\infty,t) = (t,0) \,$. Hence
\begin{equation}
 \label{app eq: Compactness estimates - xi_n formula} 
 \xi_n^{\varphi^+}(s,t) = u_n^{\varphi^+}(s,t) - u_n^{\varphi^+}(+\infty,t) - u^{\varphi^+}(s,t) + u^{\varphi^+}(+\infty,t) \,\, .
\end{equation}
Recall that, to make sense of this formula, we have to consider the image of $\xi_n^{\varphi^+}(s,t)$ under the projection $\mathrm{pr}_1 : \R^{2N} \rightarrow \bbS^1 \times \R^{2N-1}\,$. We can also interpret the equation to hold in $\R^{2N}$ if we take the inverse of the covering map $\mathrm{pr}_1$ restricted to the sheet containing the origin (the right-hand side converges uniformly to $([0],0) \in \bbS^1 \times \R^{2N-1}$ as $s \rightarrow +\infty$).

Using \eqref{app eq: Compactness estimates - xi_n formula} and \eqref{eq: Uniform Exponential Decay - d(u, u_infty)} and assuming $\Xi$ to be bounded (see Remark \ref{rmk: Uniform Exponential Decay} \ref{it: Uniform Exponential Decay - Xi bounded}), we estimate
\begin{align*}
    |\xi_n^{\varphi^+}(s,t)| &\leq |u_n^{\varphi^+}(s,t)- u_n^{\varphi^+}(+\infty,t)| + |u^{\varphi^+}(s,t)- u^{\varphi^+}(+\infty,t)| \\
    &\leq C_1 \big( d_W(u_n(s,t), u_n(+\infty,t)) + d_W(u(s,t), u(+\infty,t)) \big) \leq C_2 \, e^{-\kappa s} \,\, ,
\end{align*}
which is already the first estimate in \eqref{eq: Compactness - estimates xi_n^phi}.

Similarly, deriving \eqref{app eq: Compactness estimates - xi_n formula} in $s$-direction and using \eqref{eq: Uniform Exponential Decay - partial_s}, we obtain
\begin{align*}
    |\partial_s \xi_n^{\varphi^+}(s,t)| \leq C_3 \big( |\partial_s u_n(s,t)| + |\partial_s u(s,t)| \big) \leq C_4 \, e^{-\kappa s} \,\, ,
\end{align*}
the second estimate in \eqref{eq: Compactness - estimates xi_n^phi}.

Asymptotically the $w_n = (u_n,\eta_n)$ are gradient flow lines of $\RabinowitzHZero \,$, hence in the tubular coordinates 
\begin{equation}
\label{app eq: Compactness estimates - Rabinowitz-Floer eq coordinates}
0 = \partial_s u_n^{\varphi^+} + (\varphi^+_*J)(u_n^{\varphi^+}) \big( \partial_t u_n^{\varphi^+} - \eta_n \, (\varphi^+_*X_{H_0})(u_n^{\varphi^+}) \big) \quad \text{ for } s \text{ large}
\end{equation}
and an analogous equation holds for $w=(u,\eta) \,$. Moreover, the limits are orbits of $\partial_\vartheta = \tau(\alpha) (\varphi^+_* R_\alpha)$ and $H_0$ is a defining Hamiltonian, so
\begin{equation}
\label{app eq: Compactness estimates - partial_t limits}
\partial_t u_n^{\varphi^+} (+\infty,t) = \partial_{\vartheta} =\tau(\alpha) (\varphi^+_* X_{H_0})(u_n^{\varphi^+}(+\infty,t)) \,\, .
\end{equation}
Deriving \eqref{app eq: Compactness estimates - xi_n formula} in $t$-direction and using the gradient equation \eqref{app eq: Compactness estimates - Rabinowitz-Floer eq coordinates} and equation \eqref{app eq: Compactness estimates - partial_t limits} for the $t$-derivatives of the limits, we obtain
\begin{align}
    \partial_t \xi_n^{\varphi^+}(s,t) &= \partial_t u_n^{\varphi^+}(s,t) - \partial_t u_n^{\varphi^+}(+\infty,t) -\partial_t u^{\varphi^+}(s,t) + \partial_t u^{\varphi^+}(+\infty,t) \notag\\
    &=  (\varphi^+_*J)(u_n^{\varphi^+}) \, \partial_s u_n^{\varphi^+} + \eta_n \, (\varphi_*^+ X_{H_0})(u_n^{\varphi^+}) - \tau(\alpha) (\varphi^+_*X_{H_0})(u_n^{\varphi^+}(+\infty,t))  \notag\\
    &\hspace*{4mm} - (\varphi^+_*J)(u^{\varphi^+}) \, \partial_s u^{\varphi^+} - \eta \, (\varphi_*^+ X_{H_0})(u^{\varphi^+})  + \tau(\alpha) (\varphi^+_*X_{H_0})(u^{\varphi^+}(+\infty,t))\,\, . \label{app eq: Compactness estimates - partial_t xi_n}
\end{align}
By the triangle inequality and a Lipschitz estimate for the $\R^{2N}$-valued function $(\varphi_*^+X_{H_0})$ on a compact neighborhood of $\set{(t,0)}{t\in \bbS^1} \subseteq \bbS^1 \times \R^{2N-1} \,$, we have
\begin{align*}
    &\hspace*{7mm} \big| (\varphi^+_*J)(u_n^{\varphi^+}) \, \partial_s u_n^{\varphi^+} + \eta_n \, (\varphi_*^+ X_{H_0})(u_n^{\varphi^+}) - \tau(\alpha) (\varphi^+_*X_{H_0})(u_n^{\varphi^+}(+\infty,t)) \big|\\
    &\leq C_5 \big( |\partial_s u_n^{\varphi^+}| + |\eta_n (s) - \tau(\alpha)| +   |u_n^{\varphi^+}(s,t) - u_n^{\varphi^+}(+\infty,t)| \big) \leq C_6 \, e^{- \kappa s} \,\, .
\end{align*}
The latter inequality holds by combining the exponential decay estimates in \eqref{eq: Uniform Exponential Decay - partial_s}, \eqref{eq: Uniform Exponential Decay - d(u, u_infty)} and \eqref{eq: Uniform Exponential Decay - |eta - eta(infty)|}. This bounds the norm of the first line in \eqref{app eq: Compactness estimates - partial_t xi_n} by $C e^{-\kappa s}$ (for a suitable constant $C>0$) and analogously one can bound the second line by $C e^{-\kappa s} \,$, thus showing the desired exponential decay estimate for $|\partial_t \xi_n^{\varphi^+}|$ in \eqref{eq: Compactness - estimates xi_n^phi}.

\subsection{Supplement to Lemma \ref{lem: Fredholm - ind D Fsection_0 locally constant}}
\label{app subsec: Supplement Fredholm index locally constant}

We show an assertion made in the proof of Lemma \ref{lem: Fredholm - ind D Fsection_0 locally constant}. We slightly simplify the notation therein for our purposes.

Given $w_0 = (u_0,\eta_0) \in \Bcal^\infty\,$, a sufficiently small fiberwise convex, open neighborhood $\mathcal{O} \subseteq TW$ of the zero section and a smooth map $f: u_0^*\mathcal{O} \times \widehat{B} \rightarrow W \,$, where $\widehat{B} := B^- \times B^+ \subseteq \R^{d} \comma d := 2 (2N-1) \,$, is the product of two $(2N-1)$-balls around the origin. The map $f$ induces a pushforward map of sections $f_* \,$, defined by
\begin{align*}
    (f_*(\xi,x)) (s,t) := f( \xi(s,t) , x ) \comma \quad \xi \in W^{1,p}_{\kappa_0}(\R \times \bbS^1,u_0^*\mathcal{O}) \comma x \in \widehat{B} \,\, .
\end{align*}
Here we abridge $f(\xi(s,t),x) = f\big( ( (s,t),\xi(s,t)) \comma x \big) \,$. Let
\begin{align*}
    \Psi : f^*TW \overset{\cong}{\longrightarrow} (u_0^*TW \times \widehat{B}) \times \R^{2N}
\end{align*}
be a smooth, unitary trivialization over the base $u_0^*TW \times \widehat{B}$. Analogous to the proof of Lemma \ref{lem: Fredholm - ind D Fsection_0 locally constant}, for $u =  f_*(\xi,x)$ we define the induced trivialization
\begin{align*}
    \Psi^u : u^*TW \overset{\cong}{\longrightarrow} (\Rbar \times \bbS^1) \times \R^{2N} \comma \quad \Psi^u_{(s,t)} := \Psi_{(\xi(s,t),x)} : \, T_{u(s,t)}W \overset{\cong}{\longrightarrow} \R^{2N} \,\, .
\end{align*}
Following formula \eqref{eq: Fredholm - Hessian ActionClass trivialized S}, for $w = (u,\eta) \in \Bcal^\infty \comma u = f_*(\xi,x) \,$, we set
\begin{align}
\label{app eq: Supplement Fredholm index locally constant - Stilde}
    \widetilde{S}^w := \Psi^u  \big(  J(u) \, (\nabla_t (\Psi^u)^{-1}) + ( \nabla_{(\Psi^u)^{-1}} J) \, \partial_t u  + \eta \,\nabla_{(\Psi^u)^{-1}} \nabla H_0 \big) \,\, ,
\end{align}
so that $\widetilde{S}^w \in \Cinfty(\Rbar \times \bbS^1,\R^{2N \times 2N}) \,$. In Lemma \ref{lem: Fredholm - ind D Fsection_0 locally constant} we have claimed the following.

\begin{lemma}
\label{app lem: Supplement Fredholm locally constant - S_n - S_0 trivialized}
Let $(w_n)_n = ( u_n , \eta_n )_n \subseteq \Bcal^\infty$ be a sequence converging in $\Bcal$ to $w_0=(u_0,\eta_0) \in \Bcal^\infty \,$. Then there exist sequences $(S^1_n)_n \comma (S^2_n)_n \subseteq \Cinfty(\Rbar \times \bbS^1,\R^{2N\times 2N})$ with $S^1_n \rightarrow 0$ in $L^p_{\kappa_0}(\R \times \bbS^1,\R^{2N \times 2N}) \comma S^2_n \rightarrow 0$ in $C^0(\Rbar \times \bbS^1,\R^{2N \times 2N}) $ and
\[
\widetilde{S}^{w_n} - \widetilde{S}^{w_0} = S^1_n + S^2_n \,\, .
\]
\end{lemma}

\begin{proof}
    We fix, once and for all, a smooth global trivialization $u_0^*TW \cong (\Rbar \times \bbS^1) \times \R^{2N} \,$, so that we can regard $f$ as a smooth function on an open subset $U \subseteq (\Rbar \times \bbS^1) \times \R^{2N} \times \R^{d} $ containing $(\Rbar \times \bbS^1) \times \{0\} \times \widehat{B} \,$.
    It will be convenient to introduce the parameter-dependent function
    \begin{align*}
        G_{(\xi,x)} : \Rbar \times \bbS^1 \longrightarrow (\Rbar \times \bbS^1) \times \R^{2N} \times \R^{d} \comma \quad G_{(\xi,x)}(s,t) := \big( (s,t),\xi(s,t) , x \big) 
    \end{align*}
    for $\xi \in W^{1,p}_{\kappa_0}(\R \times \bbS^1,\R^{2N}) \cap \Cinfty(\Rbar \times \bbS^1,\R^{2N} ) $ and $x \in \R^d \,$, so that we can write $u=f_*(\xi,x)$ as $u = f \circ G_{(\xi,x)} \,$. Notice that
    \begin{equation}
    \label{app eq: Supplement Fredholm index locally constant - partial_t G}
    \partial_t G_{(\xi,x)} (s,t) = \partial_t \oplus \partial_t \xi (s,t) \oplus 0 \in T_{(s,t)} (\R \times \bbS^1) \oplus \R^{2N} \oplus \R^{d} \,\, .
    \end{equation}
    We investigate the three summands in \eqref{app eq: Supplement Fredholm index locally constant - Stilde} separately.

    \textit{First term:} Let $u = f_*(\xi,x) = f \circ G_{(\xi,x)}$ be arbitrary. Denote by $A^\Psi \in \Omega^1(U, \R^{2N \times 2N})$ the connection matrix of the pullback connection $f^*\nabla$ with respect to the trivialization $\Psi$. Then $u^*\nabla = (G_{(\xi,x)})^* f^* \nabla$ has connection matrix $(G_{(\xi,x)})^* A^\Psi$ with respect to the trivialization $\Psi^u = (G_{(\xi,x)})^* \Psi \,$. Hence
    \begin{align}
    \label{app eq: Supplement Fredholm index locally constant - connection matrix}
        \Psi^u \big( J(u) \, (\nabla_t (\Psi^u)^{-1} \big) = J_0 \, \Psi^u\big( \nabla_t (\Psi^u)^{-1}  \big) = J_0 \, A^{\Psi}_{G_{(\xi,x)}} (\partial_t G_{(\xi,x)}) \,\, .
    \end{align}
    Here we used that $\Psi^u$ is unitary, so that it intertwines $J(u)$ and the standard almost complex structure $J_0$ on $\R^{2N} \,$.
    
    Now let $u_n = f_*(\xi_n , x_n)$ be the given sequence tending to $u_0 = f_*(0,0) \,$. Then $\xi_n \rightarrow 0$ in $W^{1,p}_{\kappa_0}(\R \times \bbS^1,\R^{2N})$ and $x_n \rightarrow 0$ in $\R^{d} \,$. In particular, $(\xi_n)_n $ also converges to the zero function in $C^0(\Rbar \times \bbS^1,\R^{2N}) \,$. By \eqref{app eq: Supplement Fredholm index locally constant - partial_t G} and \eqref{app eq: Supplement Fredholm index locally constant - connection matrix}, the difference of the first terms in \eqref{app eq: Supplement Fredholm index locally constant - Stilde} is
    \begin{align*}
     &\hspace*{3mm}\Psi^{u_n} \big( J(u_n) \, \nabla_t (\Psi^{u_n})^{-1} \big)  -  \Psi^{u_0} \big( J(u_0) \, \nabla_t (\Psi^{u_0})^{-1} \big)  \\
     &= J_0 \Big( \underbrace{ ( A^{\Psi}_{G_{(\xi_n,x_n)}} - A^\Psi_{G_{(0,0)}} )(\partial_t) }_{ \rightarrow 0 \text{ in } C^0(\Rbar \times \bbS^1)} + \underbrace{A^\Psi_{G_{(\xi_n,x_n)}} (\partial_t \xi_n) }_{ \rightarrow 0 \text{ in } L^p_{\kappa_0}(\R \times \bbS^1)} \Big)  \,\, .
    \end{align*}

   \textit{Second term:} In view of \eqref{app eq: Supplement Fredholm index locally constant - partial_t G}, for $u = f_*(\xi,x) = f \circ G_{(\xi,x)}$ we may write
   \begin{align*}
       \Psi^u \big( (\nabla_{(\Psi^u)^{-1}} J) \, \partial_t u \big) &= \underbrace{\Psi_{G_{(\xi,x)}} \cdot (\nabla_{\Psi^{-1}_{G_{(\xi,x)}}} J ) \cdot (\d f)_{G_{(\xi,x)}} }_{=: Q_{G_{(\xi,x)}} }  \, \big( \underbrace{\partial_t \oplus \partial_t \xi}_{ \in \, T_{G_{(\xi,x)}} U} \big) \,\, ,
   \end{align*}
   with smooth matrix-valued 1-form $Q \in \Omega^1(U,\R^{2N \times 2N})\,$. Hence
   \begin{align*}
    &\hspace*{3mm}\Psi^{u_n} \big( (\nabla_{(\Psi^{u_n})^{-1}} J) \, \partial_t u_n \big) - \Psi^{u_0} \big( (\nabla_{(\Psi^{u_0})^{-1}} J) \, \partial_t u_0 \big)   \\
    &= \underbrace{\big(  Q_{G_{(\xi_n, x_n)}} - Q_{G_{(0,0)}} \big)  (\partial_t)}_{\rightarrow 0 \text{ in } C^0(\Rbar \times \bbS^1)} + \underbrace{Q_{G_{(\xi_n,x_n)}}  (\partial_t \xi_n)}_{\rightarrow 0 \text{ in } L^p_{\kappa_0}(\R \times \bbS^1)}  \,\, .
   \end{align*}

   \textit{Third term:} Let $u_n = f_*(\xi_n,x_n) \,$. Clearly
   \begin{align*}
     &\hspace*{6mm}\Psi^{u_n} \big( \eta_n \, \nabla_{(\Psi^{u_n})^{-1}} \nabla H_0 \big) - \Psi^{u_0} \big( \eta_0 \, \nabla_{(\Psi^{u_0})^{-1}} \nabla H_0 \big) \\
     &=  \eta_n \, \Psi_{G_{(\xi_n,x_n)}} \big( \nabla_{\Psi_{G_{(\xi_n,x_n)}}^{-1} } \nabla H_0\big) - \eta_0 \, \Psi_{G_{(0,0)}} \big( \nabla_{\Psi_{G_{(0,0)}}^{-1} } \nabla H_0\big) \,\convergence{n} 0 \quad \text{in } C^0(\Rbar \times \bbS^1)   
   \end{align*}
   since $\xi_n \rightarrow 0$ in $C^0(\Rbar \times \bbS^1,\R^{2N})$, $ x_n \rightarrow 0$
 in $\R^d$ and $\eta_n \rightarrow \eta_0$ in $C^0(\Rbar,\R) \,$.
\end{proof}

\printbibliography

\end{document}